\documentclass[a4paper, 11pt, reqno]{amsart}
\usepackage{mathrsfs}
\usepackage{amsfonts}
\usepackage[centertags]{amsmath}
\usepackage{amssymb}
\usepackage{amsthm}
\usepackage{graphicx}
\usepackage{caption}
\usepackage{hyperref}
\usepackage{enumerate}
\usepackage[textwidth=16cm, hmarginratio=1:1]{geometry}
\usepackage{appendix}
\usepackage[all]{xy}
\usepackage{float}
\usepackage{amsmath,amsfonts,amssymb,graphicx}
\usepackage{longtable}
\usepackage{resizegather}
\usepackage[perpage,symbol*]{footmisc}

\newtheorem{theorem}{Theorem}[section]
\newtheorem{corollary}[theorem]{Corollary}
\newtheorem{lemma}[theorem]{Lemma}

\newtheorem{definition}[theorem]{Definition}
\newtheorem{remark}[theorem]{Remark}
\newtheorem{example}[theorem]{Example}

\numberwithin{equation}{section}

\DeclareMathOperator*{\rep}{Rep}

\DeclareMathOperator*{\seq}{Seq}
\DeclareMathOperator*{\res}{Res}

\begin{document}

\title{On the minimal affinizations over the quantum affine algebras of type $C_n$}
\author{XIN-YANG FENG , JIAN-RONG LI , YAN-FENG LUO$^{*}$}

\address{Xin-Yang Feng: Department of Mathematics, Lanzhou University, Lanzhou 730000, P. R. China.}
\email{fengxy12@lzu.edu.cn}
\address{Jian-Rong Li: Department of Mathematics, Lanzhou University, Lanzhou 730000, P. R. China.}
\email{lijr@lzu.edu.cn, lijr07@gmail.com}
\address{Yan-Feng Luo: Department of Mathematics, Lanzhou University, Lanzhou 730000, P. R. China.}
\email{luoyf@lzu.edu.cn}

\date{}

\maketitle\footnotetext[1]{Corresponding author.}

\begin{abstract}
In this paper, we study the minimal affinizations over the quantum affine algebras of type $C_n$ by using the theory of cluster algebras. We show that the $q$-characters of a large family of minimal affinizations of type $C_n$ satisfy some systems of equations. These equations correspond to mutation equations of some cluster algebras. Furthermore, we show that the minimal affinizations in these equations correspond to cluster variables in these cluster algebras.

\hspace{0.15cm}


\hspace{0.15cm}

\end{abstract}

\section{Introduction}
Quantum groups are introduced independently by Jimbo \cite{Jim85} and Drinfeld \cite{Dri87}. Quantum affine algebras $U_q\hat{\mathfrak{g}}$ form a family of infinite-dimensional quantum groups. Although the representation theory of quantum affine algebras are studied intensively over the past few decades, the structure of the category $\rep(U_q\hat{\mathfrak{g}})$ of finite dimensional $U_q\hat{\mathfrak{g}}$-modules is far from being understood. For example, general decomposition formulas for tensor products of simple $U_q\hat{\mathfrak{g}}$-modules have not been found. The Weyl character formula for general $U_q\hat{\mathfrak{g}}$-modules has not been found.

Fomin and Zelevinsky in \cite{FZ02} introduced the theory of cluster algebras. It has many applications including quiver representations, Teichm\"{u}ller theory, tropical geometry, integrable systems, and Poisson geometry.

The aim of this paper is to apply the theory of cluster algebras to study the minimal affinizations over the quantum affine algebras of type $C_n$.

Let $\mathfrak{g}$ be a simple Lie algebra and $U_q \hat{\mathfrak{g}}$ the corresponding quantum affine algebra. Hernandez and Leclerc found a remarkable connection between cluster algebras and finite dimensional representations of $U_q \hat{\mathfrak{g}}$ in \cite{HL10}. They show that the Grothendieck ring of some subcategories of the category of all finite dimensional representations of $U_q \hat{\mathfrak{g}}$ have cluster algebra structures. In the paper \cite{HL13}, they apply the theory of cluster algebras to study the $q$-characters a family of $U_q \hat{\mathfrak{g}}$-modules called Kirillov-Reshetkhin modules.

The family of minimal affinizations over $U_q \hat{\mathfrak{g}}$ is an important family of simple $U_q \hat{\mathfrak{g}}$-modules which was introduced in \cite{C95}. The family of minimal affinizations contains the Kirillov-Reshetikhin modules. Minimal affinizations are studied intensively in recent years, see for examples, \cite{CG11}, \cite{CMY13}, \cite{Her07}, \cite{Li15}, \cite{LM13}, \cite{M10}, \cite{MP07}, \cite{MP11}, \cite{MY12a}, \cite{MY12b}, \cite{MY14}, \cite{Nao13}, \cite{Nao14}, \cite{QL14}, \cite{SS14}, \cite{ZDLL15}.

M-systems and dual M-systems of types $A_{n}$, $B_{n}$, $G_{2}$ are introduced in \cite{ZDLL15}, \cite{QL14} to study the minimal affinizations of types $A_{n}$, $B_{n}$, $G_{2}$. Recently, two closed systems which contain a large family of minimal affinizations of type $F_4$ are studied in \cite{DLL15}. The equations in these systems are satisfied by the $q$-characters of minimal affinizations of types $A_{n}$, $B_{n}$, $G_{2}$ and $F_4$. The minimal affinizations of type $C_n$ are much more complicated than the minimal affinizations of types $A_{n}$, $B_{n}$, $G_{2}$. In types $A_{n}$, $B_{n}$, $G_{2}$, all minimal affinizations are special or anti-special, \cite{Her07}, \cite{LM13}. But in type $C_n$, there are minimal affinizations which are neither special nor anti-special, \cite{Her07}. Here a $U_q \hat{\mathfrak{g}}$-module $V$ is called special (resp. anti-special) if there is only one dominant (resp. anti-dominant) monomial in the $q$-character of $V$.

In this paper, we find two closed systems which contain a large family of minimal affinizations of type $C_n$, Theorem \ref{M system of type Cn}, Theorem \ref{dual-M-system}. The equations in these closed systems are of the form
\[
[\mathcal T_1][\mathcal T_2]=[\mathcal T_3][\mathcal T_4]+[\mathcal S],
\]
where $\mathcal T_i$, $i\in \{1,2,3,4\}$, are minimal affinizations of type $C_n$ and $\mathcal S$ is some simple module. We show these equations are satisfied by the $q$-characters of the modules in the systems. We prove that the modules in the system in Theorem \ref{M system of type Cn} are special (Theorem \ref{The special modules of $C_n$}) and the modules in the system in Theorem \ref{dual-M-system} are anti-special (Theorem \ref{anti-special}).

Moreover, we show that every equation in Theorem \ref{M system of type Cn} (resp. Theorem \ref{dual-M-system}) corresponds to a mutation equation in some cluster algebra $\mathscr{A}$ (resp. $\mathscr{\widetilde{A}}$). Furthermore, every minimal affinization in Theorem \ref{M system of type Cn} (resp. Theorem \ref{dual-M-system}) corresponds to a cluster variable in the cluster algebra $\mathscr{A}$ (resp. $\mathscr{\widetilde{A}}$), Theorem \ref{connection with cluster algebra 1} (resp. Theorem \ref{minimal affinizations correspond to cluster variablesII}). Therefore we have verified that the Hernandez-Leclerc conjecture (Conjecture 13.2 in \cite{HL10} and Conjecture 9.1 in \cite{Le10}) is true for the minimal affinizations in Theorem \ref{M system of type Cn} and Theorem \ref{dual-M-system}.

The equations in Theorem \ref{M system of type Cn} and Theorem \ref{dual-M-system} can be used to compute the $q$-characters of the modules in Theorem \ref{M system of type Cn} and Theorem \ref{dual-M-system}. By using these equations, we can also obtain the ordinary characters of the modules in Theorem \ref{M system of type Cn} and Theorem \ref{dual-M-system}. A $U_q \hat{\mathfrak{g}}$-module is also a $U_q \mathfrak{g}$-module since $U_q \mathfrak{g}$ is isomorphic to a subalgebra of $U_q \hat{\mathfrak{g}}$. Usually a simple $U_q \hat{\mathfrak{g}}$-module $V$ is not simple when $V$ is considered as a $U_q \mathfrak{g}$-module. Let $V$ be a minimal affinization in Theorem \ref{M system of type Cn} and Theorem \ref{dual-M-system}. The equations in Theorem \ref{M system of type Cn} and Theorem \ref{dual-M-system} can be used to compute the decomposition formula of $V$ when we decompose $V$ into simple $U_q \mathfrak{g}$-modules.

In \cite{Li15}, the minimal affinizations over $U_q \hat{\mathfrak{g}}$ of type $C_3$ is studied by using the method of extended T-systems. The extended T-system of type $C_3$ is much more complicated than the systems in Theorem \ref{M system of type Cn} and Theorem \ref{dual-M-system}. The extended T-system of type $C_3$ contains not only minimal affinizations but also some other kinds of modules.

Let us describe the organization of the present paper. In Section \ref{background}, we give some background information about cluster algebras and representation theory of quantum affine algebras. In Section \ref{main results}, we write a system of equations satisfied by a large family of minimal affinizations of type $C_n$ explicitly, Theorem \ref{M system of type Cn}. In Section \ref{Relation between the M-systems and cluster algebras}, we study the relationship between the system in Theorem \ref{M system of type Cn} and cluster algebras. In Section \ref{dual M system}, we study the dual system of the system in Theorem \ref{M system of type Cn}. In Section \ref{proof of special}, Section \ref{proof main1} and Section \ref{proof irreducible}, we prove Theorem \ref{The special modules of $C_n$}, Theorem \ref{M system of type Cn} and Theorem \ref{irreducible} given in Section \ref{main results}, respectively.

\section{Background}\label{background}
\subsection{Cluster algebras}
Fomin and Zelevinsky in \cite{FZ02} introduced the theory of cluster algebras. Let $\mathbb{Q}$ be the field of rational numbers and $\mathcal{F}=\mathbb{Q}(x_{1}, x_{2}, \cdots, x_{n})$ the field of rational functions.
A seed in $\mathcal{F}$ is a pair $\Sigma=({\bf y}, Q)$, where ${\bf y} = (y_{1}, y_{2}, \cdots, y_{n})$ is a free generating set of $\mathcal{F}$, and $Q$ is a quiver with vertices labeled by $\{1, 2, \cdots, n\}$. Assume that $Q$ has neither loops nor $2$-cycles. For $k=1, 2, \cdots, n$, one defines a mutation $\mu_k$ by $\mu_k({\bf y}, Q) = ({\bf y}', Q')$. Here ${\bf y}' = (y_1', \ldots, y_n')$, $y_{i}'=y_{i}$, for $i\neq k$, and
\begin{equation}
y_{k}'=\frac{\prod_{i\rightarrow k} y_{i}+\prod_{k\rightarrow j} y_{j}}{y_{k}}, \label{exchange relation}
\end{equation}
where the first (resp. second) product in the right hand side is over all arrows of $Q$ with target (resp. source) $k$, and $Q'$ is obtained from $Q$ by
\begin{enumerate}
\item[(i)] adding a new arrow $i\rightarrow j$ for every existing pair of arrow $i\rightarrow k$ and $k\rightarrow j$;

\item[(ii)] reversing the orientation of every arrow with target or source equal to $k$;

\item[(iii)] erasing every pair of opposite arrows possible created by (i).
\end{enumerate}
The mutation class $\mathcal{C}(\Sigma)$ is the set of all seeds obtained from $\Sigma$ by a finite sequence of mutation $\mu_{k}$. If $\Sigma'=((y_{1}', y_{2}', \cdots, y_{n}'), Q')$ is a seed in $\mathcal{C}(\Sigma)$, then the subset $\{y_{1}', y_{2}', \cdots, y_{n}'\}$ is called a $cluster$, and its elements are called \textit{cluster variables}. The \textit{cluster algebra} $\mathscr{A}_{\Sigma}$ as the subring of $\mathcal{F}$ generated by all cluster variables. \textit{Cluster monomials} are monomials in the cluster variables supported on a single cluster.

In this paper, the initial seed in the cluster algebra we use is of the form $\Sigma=({\bf y}, Q)$, where ${\bf y}$ is an infinite set and $Q$ is an infinite quiver.

\begin{definition}[{Definition 3.1, \cite{GG14}}] \label{definition of cluster algebras of infinite rank}
Let $Q$ be a quiver without loops or $2$-cycles and with a countably infinite number of vertices labelled by all integers $i \in \mathbb{Z}$. Furthermore, for each vertex $i$ of $Q$ let the
number of arrows incident with $i$ be finite. Let ${\bf y} = \{y_i \mid i \in \mathbb{Z}\}$. An infinite initial seed is the pair $({\bf y}, Q)$. By finite sequences of mutation
at vertices of $Q$ and simultaneous mutation of the set ${\bf y}$ using the exchange relation (\ref{exchange relation}), one obtains a family of infinite seeds. The sets of variables in these seeds are called
the infinite clusters and their elements are called the cluster variables. The cluster algebra of
infinite rank of type $Q$ is the subalgebra of $\mathbb{Q}({\bf y})$ generated by the cluster variables.
\end{definition}

\subsection{Quantum affine algebras of type $C_n$}

Let $\mathfrak{g}$ be a simple Lie algebra of type $C_n$ and $I=\{1, \ldots, n\}$ the indices of the Dynkin diagram of $\mathfrak{g}$ (we use the same labeling of the vertices of the Dynkin diagram of $\mathfrak{g}$ as the one used in \cite{Car05}). Let $C=(C_{ij})_{i,j\in I}$ be the Cartan matrix of $\mathfrak{g}$, where $C_{ij}=\frac{2 ( \alpha_i, \alpha_j ) }{( \alpha_i, \alpha_i )}$.

The quantum affine algebra $U_q \hat{\mathfrak{g}}$ in Drinfeld's new realization, see \cite{Dri88}, is generated by $x_{i, n}^{\pm}$ ($i\in I, n\in \mathbb{Z}$), $k_i^{\pm 1}$ $(i\in I)$, $h_{i, n}$ ($i\in I, n\in \mathbb{Z}\backslash \{0\}$) and central elements $c^{\pm 1/2}$, subject to certain relations.

\subsection{Finite-dimensional $U_q \hat{\mathfrak{g}}$-modules and their $q$-characters} In this section, we recall the standard facts about finite-dimensional $U_q\hat{\mathfrak{g}}$-modules and their $q$-characters, see \cite{CP94}, \cite{CP95a}, \cite{FR98}, \cite{MY12a}.

Let $\mathcal{P}$ be the free abelian multiplicative group of monomials in infinitely many formal variables $(Y_{i, a})_{i\in I, a\in \mathbb{C}^{\times}}$ and let $\mathbb{Z}\mathcal{P} = \mathbb{Z}[Y_{i, a}^{\pm 1}]_{i\in I, a\in \mathbb{C}^{\times}}$ be the group ring of $\mathcal{P}$.

For each $j\in I$, a monomial $m=\prod_{i\in I, a\in \mathbb{C}^{\times}} Y_{i, a}^{u_{i, a}} \in \mathcal{P}$, where $u_{i, a}$ are some integers, is said to be \textit{$j$-dominant} (resp. \textit{$j$-anti-dominant}) if and only if $u_{j, a} \geq 0$ (resp. $u_{j, a} \leq 0$) for all $a\in \mathbb{C}^{\times}$. A monomial is called \textit{dominant} (resp. \textit{anti-dominant}) if and only if it is $j$-dominant (resp. $j$-anti-dominant) for all $j\in I$. Let $\mathcal{P}^+ \subset \mathcal{P}$ denote the set of all dominant monomials.

A finite-dimensional $U_q \hat{\mathfrak{g}}$-module is parameterized by a dominant monomial in $\mathbb{Z}\mathcal{P}$. Given a dominant monomial $m$, there is a unique finite-dimensional $U_q \hat{\mathfrak{g}}$-module $L(m)$, \cite{CP94}, \cite{CP95a}.

The theory of $q$-characters is introduced in \cite{FR98}. The $q$-character of a $U_q\hat{\mathfrak{g}}$-module $V$ is given by
\begin{align*}
\chi_q(V) = \sum_{m\in \mathcal{P}} \dim(V_{m}) m \in \mathbb{Z}\mathcal{P},
\end{align*}
where $V_m$ is the $l$-weight space with $l$-weight $m$, see \cite{FR98}. We use $\mathscr{M}(V)$ to denote the set of all monomials in $\chi_q(V)$ for a  finite-dimensional $U_q\hat{\mathfrak{g}}$-module $V$. For $m_+ \in \mathcal{P}^+$, we use $\chi_q(m_+)$ to denote $\chi_q(L(m_+))$. We also write $m \in \chi_q(m_+)$ if $m \in \mathscr{M}(\chi_q(m_+))$.

The following lemma is well-known.
\begin{lemma}
\label{contains in a larger set}
Let $m_1, m_2$ be two monomials. Then $L(m_1m_2)$ is a sub-quotient of $L(m_1) \otimes L(m_2)$. In particular, $\mathscr{M}(L(m_1m_2)) \subseteq \mathscr{M}(L(m_1))\mathscr{M}(L(m_2))$.   $\Box$
\end{lemma}

A finite-dimensional $U_q\hat{\mathfrak{g}}$-module $V$ is said to be \textit{special} if and only if $\mathscr{M}(V)$ contains exactly one dominant monomial. It is called \textit{anti-special} if and only if $\mathscr{M}(V)$ contains exactly one anti-dominant monomial. It is said to be \textit{prime} if and only if it is not isomorphic to a tensor product of two non-trivial $U_q\hat{\mathfrak{g}}$-modules, see \cite{CP97}. Clearly, if a module is special or anti-special, then it is simple.

The elements $A_{i, a} \in \mathcal{P}, i\in I, a\in \mathbb{C}^{\times}$, are defined by
\begin{align*}
A_{i, a} = Y_{i, aq_{i}}Y_{i, aq_{i}^{-1}} \prod_{C_{ji}=-1}Y_{j, a}^{-1} \prod_{C_{ji}=-2}Y_{j, aq}^{-1}Y_{j, aq^{-1}}^{-1}\prod_{C_{ji}=-3}Y_{j, aq^{2}}^{-1}Y_{j, a}^{-1}Y_{j, aq^{-2}}^{-1},
\end{align*} see Section 2.3 in \cite{FM01}.

Let $\mathcal{Q}$ be the subgroup of $\mathcal{P}$ generated by $A_{i, a}, i\in I, a\in \mathbb{C}^{\times}$. Let $\mathcal{Q}^{\pm}$ be the monoids generated by $A_{i, a}^{\pm 1}, i\in I, a\in \mathbb{C}^{\times}$. There is a partial order $\leq$ on $\mathcal{P}$ defined as follows:
\begin{align}
m\leq m' \text{ if and only if } m'm^{-1}\in \mathcal{Q}^{+}. \label{partial order of monomials}
\end{align}
For all $m_+ \in \mathcal{P}^+$, $\mathscr{M}(L(m_+)) \subset m_+\mathcal{Q}^{-}$, see \cite{FM01}.

We will need the concept of \textit{right negative} which is introduced in Section 6 of \cite{FM01}.
\begin{definition}
A monomial m is called \textit{right negative} if for all $a \in \mathbb{C}^{\times}$, for $L= \max\{l\in \mathbb{Z}\mid u_{i,aq^{l}}(m)\neq 0 \text{ for some i $\in I$}\}$ we have $u_{j,aq^{L}}(m)\leq 0$ for $j\in I$.
\end{definition}
For $i\in I, a\in \mathbb{C}^{\times}$, $A_{i,a}^{-1}$ is right-negative. A product of right-negative monomials is right-negative. If $m$ is right-negative and $m'\leq m$, then $m'$ is right-negative, see \cite{FM01}, \cite{Her06}.

As $A^{-1}_{i,a}$ are algebraically independent, for $M$ a product of $A^{-1}_{i,a}$ one can define $v_{i,a}(M)\geq0$ by $M=\prod_{i\in I,a\in \mathbb{C}^{\times}}A^{-v_{i,a}(m)}_{i,a}$. Let $v_{i}(M)=\sum_{a\in \mathbb{C}^{\times}}v_{i,a}(M)$ and $v(M)=\sum_{i\in I}v_{i}(M)$, \cite{Her08}.

For $J\subset I$, we denote by $U_q\mathfrak{\hat{g}}_{J}\subset U_q\mathfrak{\hat{g}}$ the subalgebra generated by the $x^{\pm}_{i,m}, h_{i,m}, k^{\pm}_{i}$ for $i\in J$. For $i\in I$, we denote $U_q\mathfrak{\hat{g}}_{i}=U_q\mathfrak{\hat{g}}_{\{i\}}\simeq U_{q_{i}} {\mathfrak{\hat{sl}}_{2}}$.

The following theorem is useful to show that some monomial is not in the $q$-character of a $U_q\mathfrak{\hat{g}}$-module.

\begin{theorem}[{Theorem 5.1, \cite{Her08}}]\label{Elimination Theorem}
Let $V=L(m)$ be a simple $U_q\mathfrak{\hat{g}}$-module. Suppose that $m'<m$ and $i\in I$ satisfy the following conditions:
\begin{enumerate}[(i)]
\item there is a unique $i$-dominant $M\in (\mathscr{M}(V)\cap m'\mathbb{Z}[A_{i,a}]_{a\in \mathbb{C}^{\times}})-\{m'\}$ and its coefficient is $1$,

\item $\sum_{r\in \mathbb{Z}}x^{+}_{i,r}(V_M)=\{0\}$,

\item $m'$ is not a monomial of $L_i(M)$,

\item if $m''\in \mathscr{M}(U_q\mathfrak{\hat{g}}_iV_M)$ is i-dominant, then $v(m''m^{-1})\geq v(m'm^{-1})$,

\item for all $j\neq i,\  \{m''\in \mathscr{M}(V)\mid v(m''m^{-1})< v(m'm^{-1})\}\cap m'\mathbb{Z}[A^{\pm1}_{j,a}]_{a\in \mathbb{C}^{\times}}=\emptyset$.
\end{enumerate}
Then $m'\notin \mathscr{M}(V)$.
\end{theorem}

\subsection{Minimal affinizations of type $C_n$} \label{definition of minimal affinizations}

From now on, we fix an $a\in \mathbb{C}^{\times}$ and denote $i_s = Y_{i, aq^s}$, $i \in I$, $s \in \mathbb{Z}$. Without loss of generality, we may assume that a simple $U_q \hat{\mathfrak{g}}$-module $L(m_+)$ of type $C_n$ is a minimal affinization of $V(\lambda)$ if and only if $m_+$ is one of the following monomials
\begin{gather}
\begin{align*}
\widetilde{T}^{(s)}_{k_1,k_{2},\ldots,k_n}&=\left(\prod^{k_n-1}_{i_n=0}n_{s+4i_n}\right)\prod^{n-1}_{j=1}\left(\prod^{k_{n-j}-1}_{i_{n-j}=0}(n-j)_{s+4k_n+2\sum_{p=1}^{j-1}k_{n-p}+2i_{n-j}+j}\right),\\
T^{(s)}_{k_1,k_{2},\ldots,k_n}&=\left(\prod^{k_n-1}_{i_n=0}n_{-s-4i_n}\right)\prod^{n-1}_{j=1}\left(\prod^{k_{n-j}-1}_{i_{n-j}=0}(n-j)_{-s-4k_n-2\sum_{p=1}^{j-1}k_{n-p}-2i_{n-j}-j}\right),
\end{align*}
\end{gather}
where $k_{1},\ldots k_{n}\in  \mathbb{Z}_{\geq 0}$, see \cite{CP96a}.

We write $T_{0, 0,\ldots,0}^{(s)}=1$ for any $s \in \mathbb{Z}$. We denote $A_{i,aq^{s}}^{-1}$ by $A_{i,s}^{-1}$. For a dominant monomial $T$, we use $\mathcal{T}$ to denote $L(T)$ the simple $U_q \mathfrak{\hat{g}}$-module with heighest $l$-weight $T$. For example, $\mathcal T^{(s)}_{k_{1},k_{2},\ldots,k_{n}} = L(T^{(s)}_{k_{1},k_{2},\ldots,k_{n}})$, where $k_{1},\ldots k_{n}\in  \mathbb{Z}_{\geq 0}$.

\subsection{The modules $\mathcal{\widetilde{S}}^{(s)}_{k_1,\ldots,k_{n-2},k_{n-1},0}$ and $\mathcal{S}^{(s)}_{k_1,\ldots,k_{n-2},k_{n-1},0}$}

For $s\in \mathbb{Z}$, $k_1,\ldots,k_{n-2}\in \mathbb{Z}_{\geq0}$, $k_{n-1}\in \mathbb{Z}_{\geq1}$, we will need the modules $\mathcal{\widetilde{S}}^{(s)}_{k_1,\ldots,k_{n-2},k_{n-1},0}$ and $\mathcal{S}^{(s)}_{k_1,\ldots,k_{n-2},k_{n-1},0}$ with the following heighest $l$-weights
\begin{gather}
\begin{align*}
&\widetilde{S}^{(s)}_{k_1,\ldots,k_{n-2},k_{n-1},0}=\widetilde{T}^{(s)}_{k_1,\ldots,k_{n-2},k_{n-1},0}\widetilde{T}^{(s+2k_{n-1}+4)}_{k_1,\ldots,k_{n-2}-2,0,0}\quad (k_{n-2}\geq2),\\
&S^{(s)}_{k_1,\ldots,k_{n-2},k_{n-1},0}=T^{(s)}_{k_1,\ldots,k_{n-2},k_{n-1},0}T^{(s+2k_{n-1}+4)}_{k_1,\ldots,k_{n-2}-2,0,0}\quad (k_{n-2}\geq2),\\
&\widetilde{S}^{(s)}_{k_1,\ldots,k_{n-2},k_{n-1},0}=\widetilde{T}^{(s)}_{k_1,\ldots,k_{n-2},k_{n-1},0}\widetilde{T}^{(s+2k_{n-1}+4)}_{k_1,\ldots,k_{n-3}-1,0,0,0}\quad (k_{n-2}=1),\\
&S^{(s)}_{k_1,\ldots,k_{n-2},k_{n-1},0}=T^{(s)}_{k_1,\ldots,k_{n-2},k_{n-1},0}T^{(s+2k_{n-1}+4)}_{k_1,\ldots,k_{n-3}-1,0,0,0}\quad (k_{n-2}=1),
\end{align*}
\end{gather}
respectively.

We use $T^{(s)}_{k_1,\ldots,\underset{i}{k_i},\ldots,k_{n-1},k_{n}}$ to indicate that $k_i$ is in the $i$-th position. For example, the $1$ in $T^{(s)}_{0,\ldots,0,\underset{i}1,0,\ldots,0,k_{n-1},0}$ is in the $i$-th position.

\begin{remark}
For $k_1,\ldots,k_{n-2}\in \mathbb{Z}_{\geq0}, k_{n-1}\in \mathbb{Z}_{\geq1}, s\in \mathbb{Z}$, we have
\begin{align*}
\widetilde{S}^{(s)}_{0,\ldots,0,\underset{m}1,0,\ldots,0,\underset{l}1,0,\ldots,0,k_{n-1},0}=& \widetilde{ T}^{(s)}_{0,\ldots,0,\underset{m}1,0,\ldots,0,\underset{l}1,0,\ldots,0,k_{n-1},0},\\
S^{(s)}_{0,\ldots,0,\underset{m}1,0,\ldots,0,\underset{l}1,0,\ldots,0,k_{n-1},0}=& T^{(s)}_{0,\ldots,0,\underset{m}1,0,\ldots,0,\underset{l}1,0,\ldots,0,k_{n-1},0},\\
\widetilde{S}^{(s)}_{0,\ldots,0,\underset{l}1,0,\ldots,0,k_{n-1},0}=& \widetilde{ T}^{(s)}_{0,\ldots,0,\underset{l}1,0,\ldots,0,k_{n-1},0},\\ S^{(s)}_{0,\ldots,0,\underset{l}1,0,\ldots,0,k_{n-1},0}=& T^{(s)}_{0,\ldots,0,\underset{l}1,0,\ldots,0,k_{n-1},0},
\end{align*}
where $1\leq m<l\leq n-2$.
\end{remark}

\subsection{$q$-characters of $U_q \hat{\mathfrak{sl}}_2$-modules and the Frenkel-Mukhin algorithm}
We recall the results of the $q$-characters of $U_q \hat{\mathfrak{sl}}_2$-modules which are well-understood, see \cite{CP91}, \cite{FR98}.

Let $W_{k}^{(a)}$ be the irreducible representation $U_q \hat{\mathfrak{sl}}_2$ with
highest weight monomial
\begin{align*}
X_{k}^{(a)}=\prod_{i=0}^{k-1} Y_{aq^{k-2i-1}},
\end{align*}
where $Y_a=Y_{1, a}$. Then the $q$-character of $W_{k}^{(a)}$ is given by
\begin{align*}
\chi_q(W_{k}^{(a)})=X_{k}^{(a)} \sum_{i=0}^{k} \prod_{j=0}^{i-1} A_{aq^{k-2j}}^{-1},
\end{align*}
where $A_a=Y_{aq^{-1}}Y_{aq}$.

For $a\in \mathbb{C}^{\times}, k\in \mathbb{Z}_{\geq 1}$, the set $\Sigma_{k}^{(a)}=\{aq^{k-2i-1}\}_{i=0, \ldots, k-1}$ is called a \textit{string}. Two strings $\Sigma_{k}^{(a)}$ and $\Sigma_{k'}^{(a')}$ are said to be in \textit{general position} if the union $\Sigma_{k}^{(a)} \cup \Sigma_{k'}^{(a')}$ is not a string or $\Sigma_{k}^{(a)} \subset \Sigma_{k'}^{(a')}$ or $\Sigma_{k'}^{(a')} \subset \Sigma_{k}^{(a)}$.

Denote by $\chi_q(m_+)$ the simple $U_q \hat{\mathfrak{sl}}_2$-module with
highest weight monomial $m_+$. Let $m_{+} \neq 1$ and $\in \mathbb{Z}[Y_a]_{a\in \mathbb{C}^{\times}}$ be a dominant monomial. Then $m_+$ can be uniquely (up to permutation) written in the form
\begin{align*}
m_+=\prod_{i=1}^{s} \left( \prod_{b\in \Sigma_{k_i}^{(a_i)}} Y_{b} \right),
\end{align*}
where $s$ is an integer, $\Sigma_{k_i}^{(a_i)}, i=1, \ldots, s$, are strings which are pairwise in general position and
\begin{align*}
L(m_+)=\bigotimes_{i=1}^s W_{k_i}^{(a_i)}, \qquad \chi_q(L(m_+))=\prod_{i=1}^s
 \chi_q(W_{k_i}^{(a_i)}).
\end{align*}

For $j\in I$, let
\begin{align*}
\beta_j : \mathbb{Z}[Y_{i, a}^{\pm 1}]_{i\in I; a\in \mathbb{C}^{\times}} \to \mathbb{Z}[Y_{a}^{\pm 1}]_{a\in \mathbb{C}^{\times}}
\end{align*}
be the ring homomorphism such that $\beta_j(Y_{k, a})=1$ for $k\neq j$ and $\beta_j(Y_{j, a}) = Y_{a}$ for all $a\in \mathbb{C}^{\times}$.

Let $V$ be a $U_q \hat{\mathfrak{g}}$-module. Then $\beta_i(\chi_q(V))$, $i=1, 2$, is the $q$-character of $V$ considered as a $U_{q_i} \hat{\mathfrak{sl}_2} $-module.

The Frenkel-Mukhin algorithm introduced in \cite{FM01} is very useful to compute $q$-characters of simple $U_q \mathfrak{\hat{g}}$-modules. It is shown in \cite{FM01} that the Frenkel-Mukhin algorithm works for simple $U_q \mathfrak{\hat{g}}$-modules which are special.

\subsection{Truncated $q$-characters}

In this paper, we need to use the concept truncated $q$-characters, see \cite{HL10} and \cite{MY12a}. Given a set of monomials $\mathcal R\subset\mathcal P$, let $\mathbb{Z}\mathcal{R}\subset \mathbb{Z}\mathcal{P}$ denote the $\mathbb{Z}$-module of formal linear combinations of elements of $\mathcal R$ with integer coefficients. Define
\begin{align*}
\text{trunc}_{\mathcal R}: \mathcal P \rightarrow \mathcal R; \quad m \mapsto
\begin{cases}
    m,  & \text{if}\quad  m\in \mathcal{R}, \\
    0,  & \text{if}\quad  m\notin \mathcal{R},
     \end{cases}
\end{align*}
and extend $\text{trunc}_{\mathcal R}$ as a $\mathbb{Z}$-module map $\mathbb{Z}\mathcal{P}\rightarrow \mathbb{Z}\mathcal{R}$.

Given a subset $U\subset I\times \mathbb{C}^{\times}$, let $\mathcal{Q}_{U}$ be the subgroups of $\mathcal{Q}$ generated by $A_{i, a}$ with $(i, a)\in U$. Let $\mathcal{Q}^{\pm}_{U}$ be the monoid generated by $A^{\pm 1}_{i,a}$ with $(i, a)\in U$. The polynomial $\text{trunc}_{m_{+}\mathcal{Q}^{-}_{U}}\chi_{q}(m_+)$ is called \emph{the $q$-character of $L(m_+)$ truncated to $U$}.

The following theorem can be used to compute some truncated $q$-characters.

\begin{theorem}[Theorem 2.1, \cite{MY12a}]\label{truncated $q$-characters}
Let $U\subset I\times \mathbb{C}^{\times}$ and $m_+\in \mathcal{P}^{+}$. Suppose that $\mathscr M\subset \mathcal P$ is a finite set of distinct monomials such that
\begin{enumerate}[(i)]
\item $\mathscr M\subset m_+\mathcal{Q}^{-}_{U}$,

\item $\mathcal{P}^{+}\cap \mathscr{M}=\{m_+\}$,

\item for all $m\in \mathscr M$ and all $(i, a)\in U$, if $mA^{-1}_{i,a}\notin \mathscr M$, then $mA^{-1}_{i,a}A_{j,b}\notin \mathscr M$ unless $(j, b)=(i, a)$,

\item for all $m\in \mathscr M$ and all $i\in I$, there exists a unique $i$-dominant monomial $M\in \mathscr M$ such that
\end{enumerate}
\begin{align*}
\emph{trunc}_{\beta_{i}(M\mathcal{Q}^{-}_{U})}\chi_{q}(\beta_{i}(M))=\sum_{m^{'}\in m\mathcal{Q}_{\{i\}\times \mathbb{C}^{\times}}\cap \mathscr M}\beta_{i}(m^{'}).
\end{align*}
Then
\begin{align*}
\emph{trunc}_{m_+\mathcal{Q}^{-}_{U}}\chi_{q}(m_+)=\sum_{m\in \mathscr M}m.
\end{align*}
\end{theorem}

Here $\chi_{q}(\beta_{i}(M))$ is the $q$-character of the simple $U_{q_{i}} \hat{\mathfrak{sg}}_{2}$-module with highest weight monomial $\beta_{i}(M)$, and trunc$_{\beta_{i}(M\mathcal{Q}^{-}_{U})}$ is the polynomial obtained from $\chi_{q}(\beta_{i}(M))$ by keeping only the monomials of $\chi_{q}(\beta_{i}(M))$ in the set $\beta_{i}(M\mathcal{Q}^{-}_{U})$.

\section{A closed system which contains a large family of minimal affinizations of type $C_n$ } \label{main results}
In this section, we introduce a closed system of type $C_{n}$ that contains a large family of minimal affinizations.

\subsection{Special modules}

\begin{theorem}[Theorem 3.9, \cite{Her07}]\label{special}
For $s\in \mathbb{Z}, k_1,\ldots,k_{n-1}\in \mathbb{Z}_{\geq 0}$, the module $\mathcal T_{k_1, k_2,\ldots, k_{n-1},0}^{(s)}$ is special.
\end{theorem}

\begin{theorem}\label{The special modules of $C_n$}
For $k_1,\ldots,k_{n-1}\in \mathbb{Z}_{\geq0}$, $k_n\in \mathbb{Z}_{\geq 1}$ and $s\in \mathbb{Z}$,  the modules
\begin{align*}
&\widetilde{\mathcal T}^{(s)}_{0,\ldots,0,k_{n-i},0,\ldots,0,k_n}\ (1\leq i\leq n-1),\\
&\widetilde{\mathcal T}^{(s)}_{0,\ldots,0,k_{n-j},0,\ldots,0,k_{n-i},0,\ldots,0,k_n}\ (1\leq i<j\leq n-1,\ 0\leq k_{n-j}\leq i+1) , \\
&\widetilde{\mathcal T}^{(s)}_{k_1,\ldots,k_m,0,\ldots,0,k_l,0,\ldots,0,k_n}\ (1\leq m<l\leq n-1,\ 0\leq k_1+\cdots+k_m\leq n-l+1),\\
&\widetilde{\mathcal S}^{(s)}_{k_1,\ldots,k_m,0,\ldots,0,k_l,0,\ldots,0,k_{n-1},0}\ (1\leq m<l\leq n-2,\ 0\leq k_1+\cdots+k_m\leq n-l+1),
\end{align*}
are special. In particular,we can use the Frenkel-Mukhin algorithm to compute these modules.
\end{theorem}

We will prove Theorem \ref{The special modules of $C_n$} in Section \ref{proof of special}.

\subsection{A closed system of type $C_n$}

\begin{theorem}\label{M system of type Cn}
For $s\in \mathbb{Z}$, $k_{1},\ldots, k_{n}\in  \mathbb{Z}_{\geq 0}$, we have the following system of equations:

\begin{align}\label{eqn1}
\begin{split}
[\mathcal T^{(s)}_{k_{1},k_{2}-1,k_{3},\ldots,k_{n-1},0}][\mathcal T^{(s)}_{k_{1},k_{2},k_{3},\ldots,k_{n-1},0}]=
&[\mathcal T^{(s)}_{k_{1}-1,k_{2},k_{3},\ldots,k_{n-1},0}][\mathcal T^{(s)}_{k_{1}+1,k_{2}-1,k_{3},\ldots,k_{n-1},0}]\\
&+[\mathcal T^{(s)}_{0,k_{2}-1,k_{3},\ldots,k_{n-1},0}][\mathcal T^{(s)}_{0,k_{1}+k_{2},k_{3},\ldots,k_{n-1},0}],
\end{split}
\end{align}
where $k_1, k_2\geq1;$

\begin{gather}\label{eqn2}
\begin{split}
[\mathcal T^{(s)}_{0,\ldots,0,k_{i},k_{i+1}-1,k_{i+2},\ldots,k_{n-1},0}][\mathcal T^{(s)}_{0,\ldots, 0,k_{i},k_{i+1},k_{i+2},\ldots,k_{n-1},0}]=
&[\mathcal T^{(s)}_{0,\ldots, 0,k_{i}+1,k_{i+1}-1,k_{i+2},\ldots,k_{n-1},0}][\mathcal T^{(s)}_{0,\ldots, 0,k_{i}-1,k_{i+1},k_{i+2}\ldots,k_{n-1},0}]\\
&+[\mathcal T^{(s)}_{0,\ldots,0,\underset{i+1}{k_{i}+k_{i+1}},k_{i+2},\ldots,k_{n-1},0}][\mathcal T^{(s)}_{0,\ldots,0,\underset{i-1}{k_{i}},0,\underset{i+1}{k_{i+1}-1},k_{i+2},,\ldots,k_{n-1},0}],
\end{split}
\end{gather}
where $k_{i},k_{i+1}\geq1$, $1 < i \leq n-2;$

\begin{gather}\label{eqn3}
\begin{split}
[\mathcal T^{(s)}_{k_{1},0,\ldots,0,k_{j}-1,k_{j+1},\ldots,k_{n-1},0}][\mathcal T^{(s)}_{k_{1},0,\ldots,0,k_{j},k_{j+1},\ldots,k_{n-1},0}]=
&[\mathcal T^{(s)}_{k_{1}+1,0,\ldots,0,k_{j}-1,k_{j+1},\ldots,k_{n-1},0}]
[\mathcal T^{(s)}_{k_{1}-1,0,\ldots, 0,k_{j},k_{j+1}\ldots,k_{n-1},0}]\\
&+[\mathcal T^{(s)}_{0,k_{1},0,\ldots,0,k_{j},k_{j+1},\ldots,k_{n-1},0}][\mathcal T^{(s)}_{0,\ldots,0,k_{j}-1,k_{j+1},\ldots,k_{n-1},0}],
\end{split}
\end{gather}
where $k_1, k_j>0$, $2<j\leq n-1;$

\begin{gather}\label{eqn4}
\begin{split}
[\mathcal T^{(s)}_{0,\ldots,0,\underset{i}{k_{i}},0,\ldots,0,\underset{j}{k_{j}-1},k_{j+1},\ldots,k_{n-1},0}][\mathcal T^{(s)}_{0,\ldots, 0,\underset{i}{k_{i}},0,\ldots,0,\underset{j}{k_{j}},k_{j+1},\ldots,k_{n-1},0}]=
&[\mathcal T^{(s)}_{0,\ldots, 0,\underset{i}{k_{i}-1},0,\ldots,0,\underset{j}{k_{j}},k_{j+1},\ldots,k_{n-1},0}] [\mathcal T^{(s)}_{0,\ldots, 0,\underset{i}{k_{i}+1},0,\ldots,0,\underset{j}{k_{j}-1},k_{j+1}\ldots,k_{n-1},0}]\\
&+[\mathcal T^{(s)}_{0,\ldots,0,\underset{i+1}{k_{i}},0, \ldots, 0,\underset{j}{k_{j}},k_{j+1},\ldots,k_{n-1},0}][\mathcal T^{(s)}_{0,\ldots,0,\underset{i-1}{k_{i}},0,\ldots, 0,\underset{j}{k_{j}-1},k_{j+1},\ldots,k_{n-1},0}],
\end{split}
\end{gather}
where $k_i, k_j>0$, $2< i+1< j\leq n-1.$
\vskip 0.3cm
In the following, let $k_{1},\ldots, k_{n-1}\in  \mathbb{Z}_{\geq 0},\ k_n\in \mathbb{Z}_{>0}$ and $s\in \mathbb{Z}$.


\begin{gather}\label{eqn511}
\begin{split}
[\widetilde{\mathcal T}^{(s)}_{k_1,\ldots,k_m-1,0,\ldots,0,k_n}][\widetilde{\mathcal T}^{(s-4)}_{k_1,\ldots,k_m,0,\ldots,0,1,k_n}]=&[\widetilde{\mathcal T}^{(s)}_{k_1,\ldots,k_m,0,\ldots,0,1,k_n-1}][\widetilde{\mathcal T}^{(s-4)}_{k_1,\ldots,k_m-1,0,\ldots,0,k_n+1}]\\
&+[\widetilde{\mathcal T}^{(s+4k_n)}_{k_1,\ldots,k_m-1,0,\ldots,0,0}][\widetilde{\mathcal T}^{(s-4)}_{k_1,\ldots,k_m,0,\ldots,0,{2k_n+1},0}],
\end{split}
\end{gather}
where $1\leq m\leq n-2$ and $0\leq k_1+\cdots+k_m\leq2;$

\begin{gather}\label{eqn512}
\begin{split}
[\widetilde{\mathcal T}^{(s)}_{k_1,\ldots,k_m,0,\ldots,0,k_{n-1}-2,k_n}][\widetilde{\mathcal T}^{(s-4)}_{k_1,\ldots,k_m,0,\ldots,0,k_{n-1},k_n}]=&[\widetilde{\mathcal T}^{(s)}_{k_1,\ldots,k_m,0,\ldots,0,k_{n-1},k_{n}-1}][\widetilde{\mathcal T}^{(s-4)}_{k_1,\ldots,k_m,0,\ldots,0,k_{n-1}-2,k_{n}+1}]\\
&+[\widetilde{\mathcal T}^{(s+4k_n)}_{k_1,\ldots,k_m,0,\ldots,0,k_{n-1}-2,0}][\widetilde{\mathcal T}^{(s-4)}_{k_1,\ldots,k_m,0,\ldots,0,2k_n+k_{n-1},0}],
\end{split}
\end{gather}
where $1\leq m\leq n-2,\ k_{n-1}\geq2$ and $0\leq k_1+\cdots+k_m\leq2;$

\begin{gather}\label{eqn5211}
\begin{split}
[\widetilde{\mathcal T}^{(s)}_{k_1,\ldots,k_m-1,0,\ldots,0,\ldots,0,k_n}][\widetilde{\mathcal T}^{(s-4)}_{k_1,\ldots,k_m,0,\ldots,0,\underset{l}1,0,\ldots,0,k_n}]=&[\widetilde{\mathcal T}^{(s)}_{k_1,\ldots,k_m,0,\ldots,0,\underset{l}1,0,\ldots,0,k_n-1}][\widetilde{\mathcal T}^{(s-4)}_{k_1,\ldots,k_m-1,0,\ldots,0,\ldots,0,k_n+1}]\\
&+[\widetilde{\mathcal S}^{(s-4)}_{k_1,\ldots,k_m,0,\ldots,0,\underset{l}1,0,\ldots,0,2k_n,0}],
\end{split}
\end{gather}
where $1\leq m<l\leq n-2,\ 0\leq k_1+\cdots+k_m\leq n-l+1;$


\begin{gather}\label{eqn5221}
\begin{split}
[\widetilde{\mathcal T}^{(s)}_{k_1,\ldots,k_m,0,\ldots,0,k_l-2,0,\ldots,0,k_n}][\widetilde{\mathcal T}^{(s-4)}_{k_1,\ldots,k_m,0,\ldots,0,k_l,0,\ldots,0,k_n}]=&[\widetilde{\mathcal T}^{(s)}_{k_1,\ldots,k_m,0,\ldots,0,k_l,0,\ldots,0,k_n-1}][\widetilde{\mathcal T}^{(s-4)}_{k_1,\ldots,k_m,0,\ldots,0,k_l-2,0,\ldots,0,k_n+1}]\\
&+[\widetilde{\mathcal S}^{(s-4)}_{k_1,\ldots,k_m,0,\ldots,0,k_l,0,\ldots,0,2k_n,0}],
\end{split}
\end{gather}
where $1\leq m<l\leq n-2,\ k_l\geq2$ and $0\leq k_1+\cdots+k_m\leq n-l+1$.

\end{theorem}

Theorem \ref{M system of type Cn} will be proved in Section \ref{proof main1}.

\begin{example}\label{eample1}
The following are some equations in the system of type $C_3$.

\begin{align*}
&[1_{-2}][1_{-4}2_{-1}]=[1_{-4}1_{-2}][2_{-1}]+[2_{-3}2_{-1}],\\
&[1_{-4}1_{-2}][1_{-6}1_{-4}2_{-1}]=[1_{-4}2_{-1}][1_{-6}1_{-4}1_{-2}]+[2_{-5}2_{-3}2_{-1}],\\
&[3_{-2}][3_{-6}2_{-1}]=[2_{-1}][3_{-6}3_{-2}]+[2_{-5}2_{-3}2_{-1}],\\
&[3_{-6}3_{-2}][3_{-10}3_{-6}2_{-1}]=[3_{-6}2_{-1}][3_{-10}3_{-6}3_{-2}]+[2_{-9}2_{-7}2_{-5}2_{-3}2_{-1}],\\
&[3_{-6}1_{0}][3_{-2}]=[1_{0}][3_{-6}3_{-2}]+[2_{-5}2_{-3}1_{0}],\\
&[3_{-6}3_{-2}][3_{-10}3_{-6}1_{0}]=[3_{-6}1_{0}][3_{-10}3_{-6}3_{-2}]+[2_{-3}2_{-9}2_{-7}2_{-5}1_{0}],\\
&[3_{-4}][3_{-8}2_{-3}1_{0}]=[2_{-3}1_{0}][3_{-8}3_{-4}]+[2_{-7}2_{-5}2_{-3}1_{0}],\\
&[3_{-8}3_{-4}][3_{-12}3_{-8}2_{-3}1_{0}]=[3_{-8}2_{-3}1_{0}][3_{-12}3_{-8}3_{-4}]+[2_{-11}2_{-9}2_{-7}2_{-5}2_{-3}1_{0}].
\end{align*}
\end{example}
\ \\
\begin{example}\label{eample1}
The following are some equations in the system of type $C_4$.

\begin{align*}
&[4_{-3}][4_{-7}1_{0}]=[1_{0}][4_{-7}4_{-3}]+[3_{-6}3_{-4}1_{0}],\\
&[4_{-5}][4_{-9}2_{-3}1_{0}]=[2_{-3}1_{0}][4_{-9}4_{-5}]+[3_{-8}3_{-6}2_{-3}1_0],\\
&[4_{-7}1_0][4_{-11}2_{-5}1_{-2}1_0]=[2_{-5}1_{-2}1_0][4_{-11}4_{-7}1_0]+[3_{-10}3_{-8}2_{-5}1_{-2}1^2_{0}],\\
&[4_{-9}1_{-2}1_0][4_{-13}2_{-7}1_{-4}1_{-2}1_0]=[2_{-7}1_{-4}1_{-2}1_0][4_{-13}4_{-9}1_{-2}1_0]+[3_{-12}3_{-10}2_{-7}1_{-4}1^{2}_{-2}1^{2}_{0}].
\end{align*}
\end{example}
\ \\
Furthermore, we have the following theorem.
\begin{theorem}\label{irreducible}
For each equation in Theorem \ref{M system of type Cn}, all summands on the right hand side are simple.
\end{theorem}
Theorem \ref{irreducible} will be proved in Section \ref{proof irreducible}.

\subsection{A system corresponding to the system in Theorem \ref{M system of type Cn}}
For $k_{1},\ldots,k_{n} \in \mathbb{Z}_{\geq 0}$, $s \in \mathbb{Z}$, let
\[
\mathfrak{m}_{k_1,\ldots,k_n} =\res(\mathcal{T}_{k_1,\ldots,k_n}^{(s)}) \ (\text{resp. } \widetilde{\mathfrak{m}}_{k_1,\ldots,k_n} =\res(\widetilde{\mathcal{T}}_{k_1,\ldots,k_n}^{(s)}))
\]
be the restriction of $\mathcal{T}_{k_1,\ldots,k_n}^{(s)}$\ (resp. $\widetilde{\mathcal{T}}_{k_1,\ldots,k_n}^{(s)}$) to $U_q \mathfrak{g}$. It is clear that $\res(\mathcal{T}_{k_1,\ldots,k_n}^{(s)})$ and $\res(\widetilde{\mathcal{T}}_{k_1,\ldots,k_n}^{(s)})$ do not depend on $s$. Let $\chi(M)$\ (resp. $\chi(\widetilde{M})$) be the character of a $U_q \mathfrak{g}$-module $M$\ (resp. $\widetilde{M}$). By replacing each $[\mathcal{T}_{k_1,\ldots,k_n}^{(s)}]$ (resp. $[\mathcal{\widetilde{T}}_{k_1,\ldots,k_n}^{(s)}]$) in the system in Theorem \ref{M system of type Cn} with  $\chi(\mathfrak{m}_{k_1,\ldots,k_n})$ (resp. $\chi(\widetilde{\mathfrak{m}}_{k_1,\ldots,k_n})$), we obtain a system of equations consisting of the characters of $U_{q}\mathfrak{g}$-modules. The following are two equations in the system.
\begin{equation*}
\begin{split}
&\chi(\mathfrak m_{k_{1},k_{2}-1,k_{3},\ldots,k_{n-1},0})\chi(\mathfrak m_{k_{1},k_{2},k_{3},\ldots,k_{n-1},0})=\chi(\mathfrak m_{k_{1}-1,k_{2},k_{3},\ldots,k_{n-1},0})\chi(\mathfrak m_{k_{1}+1,k_{2}-1,k_{3},\ldots,k_{n-1},0})+\chi(\mathfrak m_{0,k_{2}-1,k_{3},\ldots,k_{n-1},0})\chi(\mathfrak m_{0,k_{1}+k_{2},k_{3},\ldots,k_{n-1},0}),\\
&\chi(\widetilde{\mathfrak m}_{k_1,\ldots,k_m-1,0,\ldots,0,k_n})\chi(\widetilde{\mathfrak m}_{k_1,\ldots,k_m,0,\ldots,0,1,k_n})=\chi(\widetilde{\mathfrak m}_{k_1,\ldots,k_m,0,\ldots,0,1,k_n-1})\chi(\widetilde{\mathfrak m}_{k_1,\ldots,k_m-1,0,\ldots,0,k_n+1})
+\chi(\widetilde{\mathfrak m}_{k_1,\ldots,k_m-1,0,\ldots,0,0})\chi(\widetilde{\mathfrak m}_{k_1,\ldots,k_m,0,\ldots,0,{2k_n+1},0}).
\end{split}
\end{equation*}

\section{Relation between the system in Theorem \ref{M system of type Cn} and cluster algebras}\label{Relation between the M-systems and cluster algebras}

In this section, we will show that the equations in the system in Theorem \ref{M system of type Cn} correspond to mutations in some cluster algebra $\mathscr{A}$. Moreover, every minimal affinization in the system in Theorem \ref{M system of type Cn}  corresponds to a cluster variable in the cluster algebra $\mathscr{A}$.
\subsection{Definition of cluster algebras $\mathscr{A}$}\label{definition of cluster algebra A}

Let $I_1=\{1,2,\ldots,n-1 \}$ and
\begin{align*}
&S_{1}=\{-2i-1\mid i\in \mathbb{Z}_{\geq 0}\},\quad  S_{2}=\{-2i\mid i\in \mathbb{Z}_{\geq 0}\}, \\
&S_{3}=\{2i+1\mid i\in \mathbb{Z}_{\geq 0}\},\ \ \quad  S_{4}=\{2i\mid i\in \mathbb{Z}_{\geq 0}\},\\
&S_{5}=\{-4i-1\mid i\in \mathbb{Z}_{\geq 0}\},\quad  S_{6}=\{-4i-3\mid i\in \mathbb{Z}_{\geq 0}\},\\
&S_{7}=\{-4i\mid i\in \mathbb{Z}_{\geq 0}\},\ \ \ \ \ \quad S_{8}=\{-4i-2\mid i\in \mathbb{Z}_{\geq 0}\}.
\end{align*}
Let
\begin{gather}
\begin{align*}
V=((S_3\cap I_1)\times S_2)\cup((S_4\cap I_1)\times S_1)\cup((\{n\}\cap S_3)\times S_7)\cup((\{n\}\cap S_3)\times S_8)\cup((\{n\}\cap S_4)\times S_5)\cup((\{n\}\cap S_4)\times S_6).
\end{align*}
\end{gather}
A quiver $Q$ for $U_q \hat{\mathfrak{g}}$ of type $C_{n}$ with vertex set $V$ will be defined as follows. The arrows of $Q$ are given by the following rule: there is an arrow from the vertex $(i,r)$ to the vertex $(j,s)$ if and only if $b_{ij}\neq 0$ and $s=r+b_{ij}-d_{i}+d_{j}$. The quiver $Q$ of type $C_{n}$ is the same as the quiver $G^-$ of type $C_n$ in \cite{HL13}.

Let ${\bf t}={\bf t}_{1}\cup {\bf t}_{2}$, where
\begin{align}\label{at1}
{\bf t_1} =\{t^{(-n+i-[i]_{\text{mod} 2}+1)}_{0,\ldots,0,k_i,0,\ldots,0}\mid i\in\{1,\ldots,n-1\},\ k_i\in \mathbb{Z}_{\geq0}\},
\end{align}

\begin{gather}\label{at2}
\begin{align}
{\bf t_2} =
\begin{cases}
     \{\widetilde{t}^{(-4k_{n}+4)}_{0,\ldots, 0, \ldots,0, k_n},\  \widetilde{t}^{(-4k_{n}+2)}_{0,\ldots, 0, \ldots,0, k_n}\mid n \text{ is } \text{odd},\ k_{n}\in \mathbb{Z}_{\geq1} \},\\
      \\
     \{\widetilde{t}^{(-4k_{n}+3)}_{0,\ldots, 0, \ldots,0, k_n},\  \widetilde{t}^{(-4k_{n}+1)}_{0,\ldots, 0, \ldots,0, k_n}\mid n \text { is } \text{even},\ k_{n}\in \mathbb{Z}_{\geq1} \}, \\
     \\
     \{\widetilde{t}^{(-2k_i-n+i+[i]_{\text{mod} 2}+1)}_{0,\ldots, 0,k_i,0, \ldots,0}\mid i\in\{1,\ldots,n-1\},\ k_i\in \mathbb{Z}_{\geq0}\}.
     \end{cases}
\end{align}
\end{gather}

Let $\mathscr{A}$ be the cluster algebra defined by the initial seed $({\bf t}, Q)$. The cluster algebra $\mathscr{A}$ of type $C_{n}$ is the same as the cluster algebra for $U_q \hat{\mathfrak{g}}$ of type $C_{n}$ introduced in \cite{HL13}.

\subsection{Mutation sequences} For the quiver of type $C_{n}$. We use $``C_{2i-1}"$ to denote the column of vertices $(2i-1, 0)$, $(2i-1, -2)$, $(2i-1, -4)$, $\ldots$ in $Q$, where $i\in \mathbb{Z}_{\geq1}$. We use $``C_{2i}"$ to denote the column of vertices $(2i, -1)$, $(2i, -3)$, $(2i, -5),\ldots$ in $Q$, where $i\in \mathbb{Z}_{\geq1}$. If $n$ is even, we use $``C_{n}"$ to denote the column of vertices $(n, -1)$, $(n, -5),\ldots, (n, -4i-1),\ldots$ in $Q$, where $i\in \mathbb{Z}_{\geq0}$ and $``C_{n+1}"$ denote the column of vertices $(n, -3)$, $(n, -7),\ldots, (n, -4i-3),\ldots$ in $Q$, where $i\in \mathbb{Z}_{\geq0}$. If $n$ is odd, we use $``C_{n}"$ to denote the column of vertices $(n, 0)$, $(n, -4),\ldots, (n, -4i),\ldots$ in $Q$, where $i\in \mathbb{Z}_{\geq0}$ and $``C_{n+1}"$ denote the column of vertices $(n, -2)$, $(n, -6), \ldots,(n, -4i-2),\ldots$ in $Q$, where $i\in \mathbb{Z}_{\geq0}$. Let $C_{1}, \ldots, C_{n}, C_{n+1}$ be the columns of the quiver. By saying that we mutate the column $C_i$, $i \in \{1,\ldots, n+1\}$, we mean that we mutate the vertices of $C_i$ as follows. First we mutate at the first vertex of $C_i$, then the second vertex of $C_i$, an so on until the vertex at infinity. By saying that the mutate $(C_{i_1}, C_{i_2}, \ldots, C_{i_m} )$, where $i_1, \ldots, i_m \in \{1, 2, \ldots, n+1\}$, we mean that we first mutate the column $C_{i_1}$, then the column $C_{i_2}$, an so on up to the column $C_{i_m}$.

\textbf{Case 1.}  Let $k_{1},k_{2},\ldots,k_{n} \in \mathbb{Z}_{\geq 0}$ and let $k_l$ be the first non-zero integer in $k_1, k_2, \ldots, k_{n-1}$ from the right. We define some variables $t_{k_{1},k_{2},\ldots,k_{n-1},0}^{(s)}$, where
\begin{align*}
s=-n+l-[l]_{\text{mod} 2}+1,
\end{align*}
recursively as follows. The variables $t^{(s)}_{0,\ldots, 0,k_{i},0, \ldots, 0}$ ($k_{i}\in \mathbb{Z}_{\geq0}$) in \ref{at1} are already defined. They are cluster variables in the initial seed of $\mathscr{A}$ define in Section \ref{definition of cluster algebra A}.

We use $\emptyset$ to denote the empty mutation sequence, and use
\begin{align*}
\prod_{k=1}^{\frac{n-1}{2}}(C_{2k-1}, C_{2k-2}, C_{2k-3}, \ldots, C_{1})
\end{align*}
to denote the mutation sequence
\begin{align*}
(C_{1}; C_3, C_2, C_1;\ldots;C_{n-4},C_{n-5},\ldots,C_{1};C_{n-2},C_{n-3},\ldots,C_{1}).
\end{align*}

Let
\begin{align*}
M_{l}^{(1)} =
\begin{cases}
    \emptyset,  & l=1,\ 2 , \\
    \\
    \prod_{k=1}^{\frac{l-1}{2}}(C_{2k-1}, C_{2k-2}, C_{2k-3}, \ldots, C_{1}),  & l \equiv 1 \pmod2,\ 1<l\leq n-1,  \\
    \\
     \prod_{k=1}^{\frac{l}{2}}(C_{2k-1},C_{2k-2},C_{2k-3},\ldots,C_{1}) , & l\equiv 0\pmod2,\ 2<l\leq n-1.
     \end{cases}
\end{align*}

Let $k_1, k_2, \ldots, k_{n-1} \in \mathbb{Z}_{\geq 0}$ and $k_l$ be the first non-zero integer in $k_1, k_2, \ldots, k_{n-1}$ from the right.
 Let $\seq$ be the mutation sequence: first we mutate $M_l^{(1)}$ starting from the initial quiver $Q$, then we mutate $(C_{l-1},C_{l-2},\ldots,C_{1})$ $k_{l}$ times, and then we mutate $(C_{l-2},C_{l-3},\ldots,C_{1})$ $k_{l-1}$ times; continue this procedure, we mutate $(C_{t-1},C_{t-2},\ldots,C_{1})$ $k_{t}$ times, $t=l-2, l-3, \ldots,2$. If $k_t = 0$, then ``we mutate $(C_{t-1},C_{t-2},\ldots,C_{1})$ $k_{t}$ times" means ``we do not mutate $(C_{t-1},C_{t-2},\ldots,C_{1})$".

We define
\begin{align}\label{variables 1}
\begin{split}
t^{(-n+l-[l]_{\text{mod} 2}+1)}_{0,\ldots,0,k_{p},0,\ldots,0,k_{q},k_{q+1},\ldots,k_{n-1},0}& = t'^{(-n+l-[l]_{\text{mod} 2}+1)}_{0,\ldots,0,k_{p},0,\ldots,k_{q}-1,k_{q+1},\ldots,k_{n-1},0} \quad (1\leq p<q\leq n-1),\\
\end{split}
\end{align}
where
\begin{gather}\label{variables 2}
\begin{split}
&t'^{(-n+l-[l]_{\text{mod} 2}+1)}_{k_{1},k_{2}-1,k_3,\ldots,k_{n-1},0}\\
&=\frac{t^{(-n+l-[l]_{\text{mod} 2}+1)}_{k_{1}-1,k_{2},k_3,\ldots,k_{n-1},0}t^{(-n+l-[l]_{\text{mod} 2}+1)}_{k_{1}+1,k_{2}-1,k_3,\ldots,k_{n-1},0}+
t^{(-n+l-[l]_{\text{mod} 2}+1)}_{0,k_{1}+k_{2},k_{3},\ldots,k_{n-1},0}
t^{(-n+l-[l]_{\text{mod} 2}+1)}_{0,k_{2}-1,k_{3},\ldots,k_{n-1},0}}{t^{(-n+l-[l]_{\text{mod} 2}+1)}_{k_{1},k_{2}-1,k_3,\ldots,k_{n-1},0}},
\end{split}
\end{gather}
where $k_1, k_2\geq1;$
\begin{gather}\label{variables 3}
\begin{split}
&t'^{(-n+l-[l]_{\text{mod} 2}+1)}_{0,\ldots,0,k_{i},k_{i+1}-1,k_{i+2},\ldots,k_{n-1},0}\\
&=\frac{t^{(-n+l-[l]_{\text{mod} 2}+1)}_{0,\ldots,0,k_{i}-1,k_{i+1},k_{i+2},\ldots,k_{n-1},0}
t^{(-n+l-[l]_{\text{mod} 2}+1)}_{0,\ldots, 0,k_{i}+1,k_{i+1}-1,k_{i+2}\ldots,k_{n-1},0}+
t^{(-n+l-[l]_{\text{mod} 2}+1)}_{0,\ldots,0,\underset{i+1}{k_{i}+k_{i+1}},k_{i+2},\ldots,k_{n-1},0}t^{(-n+l-[l]_{\text{mod} 2}+1)}_{0,\ldots,0,\underset{i-1}{k_{i}},0,\underset{i+1}{k_{i+1}-1},\ldots,k_{n-1},0}}{t^{(-n+l-[l]_{\text{mod} 2}+1)}_{0,\ldots,0,k_{i},k_{i+1}-1,k_{i+2},\ldots,k_{n-1},0}},
\end{split}
\end{gather}
where $k_{i},k_{i+1}\geq1$, $1 < i\leq n-1;$
\begin{gather}\label{variables 4}
\begin{split}
&t'^{(-n+l-[l]_{\text{mod} 2}+1)}_{k_{1},0,\ldots,0,k_{j}-1,k_{j+1},\ldots,k_{n-1},0}\\
&=\frac{t^{(-n+l-[l]_{\text{mod} 2}+1)}_{k_{1}+1,0,\ldots,0,k_{j}-1,k_{j+1},\ldots,k_{n-1},0}
t^{(-n+l-[l]_{\text{mod} 2}+1)}_{k_{1}-1,0,\ldots,0,k_{j},k_{j+1}\ldots,k_{n-1},0}+
t^{(-n+l-[l]_{\text{mod} 2}+1)}_{0,k_{1},0,\ldots,0,k_{j},k_{j+1},\ldots,k_{n-1},0}t^{(-n+l-[l]_{\text{mod} 2}+1)}_{0,\ldots,0,k_{j}-1,k_{j+1},\ldots,k_{n-1},0}}{t^{(-n+l-[l]_{\text{mod} 2}+1)}_{k_{1},0,\ldots,0,k_{j}-1,k_{j+1},\ldots,k_{n-1},0}},
\end{split}
\end{gather}
where $k_1, k_j>0$, $2<j\leq n-1$;
\begin{gather}\label{variables 5}
\begin{split}
&t'^{(-n+l-[l]_{\text{mod} 2}+1)}_{0,\ldots,0,\underset{i}{k_{i}},0,\ldots,0,\underset{j}{k_{j}-1},k_{j+1},\ldots,k_{n-1},0}\\
&=\frac{t^{(-n+l-[l]_{\text{mod} 2}+1)}_{0,\ldots, 0,\underset{i}{k_{i}-1},0,\ldots,0,\underset{j}{k_{j}},k_{j+1},\ldots,k_{n-1},0}
t^{(-n+l-[l]_{\text{mod} 2}+1)}_{0,\ldots, 0,\underset{i}{k_{i}+1},0,\ldots,0,\underset{j}{k_{j}-1},k_{j+1}\ldots,k_{n-1},0}+
t^{(-n+l-[l]_{\text{mod} 2}+1)}_{0,\ldots,0,\underset{i+1}{k_{i}},0, \ldots,0,\underset{j}{k_{j}},k_{j+1},\ldots,k_{n-1},0}
t^{(-n+l-[l]_{\text{mod} 2}+1)}_{0,\ldots,0,\underset{i-1}{k_{i}},0,\ldots,0,\underset{j}{k_{j}-1},\ldots,k_{n-1},0}}{t^{(-n+l-[l]_{\text{mod} 2}+1)}_{0,\ldots,0,\underset{i}{k_{i}},0,\ldots,0,\underset{j}{k_{j}-1},k_{j+1},\ldots,k_{n-1},0}},
\end{split}
\end{gather}
where $k_i, k_j>0$, $2< i+1< j\leq n-1$;
are mutation equations which occur when we mutate $\seq$. The variables (\ref{variables 1}) are defined in the order according to the mutation sequence $\seq$. In this order, every variable in (\ref{variables 1}) is defined by an equation of (\ref{variables 2})--(\ref{variables 5}) using variables in ${\bf t}$ and those variables in (\ref{variables 1}) which are already defined.

\textbf{Case 2.}  For $k_{1},k_{2},\ldots,k_{n} \in \mathbb{Z}_{\geq 0}$, let $k_r$ be the first non-zero integer in $k_{n-1}, k_{n-2}, \ldots, k_{1}$ from the right.

Let
\begin{align*}
N_{n, r}^{(1)} =
\begin{cases}
    \emptyset,  & r=n, \\
    \\
    C_{n}, &n \equiv 1 \pmod2,\ r=n-1, \\
    \\
    \prod_{k=1}^{\frac{n-r}{2}}(C_{n-2k+2}, C_{n-2k+3}, \ldots, C_{n-1}, C_{n-[k]_{\text{mod} 2}+1}),  &n \equiv 1 \pmod2 ,\ r \equiv 1 \pmod2,\\ & r\leq n-2,
    \\
    \prod_{k=1}^{\frac{n-r+1}{2}}(C_{n-2k+2},C_{n-2k+3}, \ldots, C_{n-1}, C_{n-[k]_{\text{mod} 2}+1}) , & n \equiv 1 \pmod2 ,\ r \equiv 0 \pmod2,\\ & r< n-2,
    \\
    \prod_{k=1}^{\frac{n-r-1}{2}}(C_{n-2k+1},C_{n-2k+2}, \ldots, C_{n-1}, C_{n-[k]_{\text{mod} 2}+1}) , & n \equiv 0 \pmod2 ,\ r \equiv 1 \pmod2,\\ & r< n-2,\\
    \prod_{k=1}^{\frac{n-r}{2}}(C_{n-2k+1}, C_{n-2k+2}, \ldots, C_{n-1}, C_{n-[k]_{\text{mod} 2}+1}),  &n \equiv 0 \pmod2 ,\ r \equiv 0 \pmod2,\\ & r\leq n-2.  \\
     \end{cases}
\end{align*}


Let $S$, $S'$ be two mutation sequences. We use $\underset{2k+1}{\underbrace{SS'\cdots SS'S}}$ to denote the mutation sequence $ SS'\cdots SS'S $, where the number of $S$ is $k+1$ and use $\underset{2k}{\underbrace{SS'\cdots SS'}}$ to denote the mutation sequence $ SS'\cdots SS' $, where the number of $S$ is $k$.

Let $k_{1}, \ldots, k_{n-1} \in \mathbb{Z}_{\geq 0}, k_n \in \mathbb{Z}_{\geq1}$ and let $k_r$ be the first non-zero integer in $k_{n-1}, \ldots, k_{1}$ from the right. For $n$ is even and $r=n-1\neq0$, let $\seq$ be the mutation sequence: we mutate
\begin{align*}
\begin{cases}
\underset{k_r}{\underbrace{(C_{n},C_{n+1},\ldots,C_{n},C_{n+1},C_{n})}},\ &\text{if}\quad k_r \equiv 1 \pmod2 ,\\
\underset{k_r}{\underbrace{(C_{n},C_{n+1},\ldots,C_{n},C_{n+1})}},\ &\text{if}\quad k_r \equiv 0 \pmod2,\\
\end{cases}
\end{align*}
starting from the initial quiver $Q$.

For $n$ is odd and $r=n-1\neq0$, let $\seq$ be the mutation sequence: first we mutate $N_{n, r}^{(1)}$ starting from the initial quiver $Q$, then we mutate
\begin{align*}
\begin{cases}
\underset{k_r}{\underbrace{(C_{n+1},C_n,\ldots,C_{n+1},C_n,C_{n+1})}},\ &\text{if}\quad k_r \equiv 1 \pmod2 ,\\
\underset{k_r}{\underbrace{(C_{n+1},C_n,\ldots,C_{n+1},C_{n})}},\ &\text{if}\quad k_r \equiv 0 \pmod2.\\
\end{cases}
\end{align*}

For $1\leq r\leq n-2$, let $\seq$ be the mutation sequence: first we mutate $N_{n, r}^{(1)}$ starting from the initial quiver $Q$, then we mutate
\begin{table}[H] \resizebox{.6\width}{.7\height}{
\begin{tabular}{|c|c|c|c|c|c|c|c|c|}
\hline

\  &\multicolumn{2}{c|}{$\frac{n-r}{2}$ is odd}  &\multicolumn{2}{c|}{$\frac{n-r}{2}$ is even}  &\multicolumn{2}{c|}{$\frac{n-r+1}{2}$ is odd} & \multicolumn{2}{c|}{$\frac{n-r+1}{2}$ is even}  \\
\hline
                \  & $k_r$ is odd & $k_r$ is even & $k_r$ is odd & $k_r$ is even & $k_r$ is odd & $k_r$ is even & $k_r$ is odd & $k_r$ is even \\
\hline
                $n$ is odd,\ $r$ is odd & $\underset{k_r}{\underbrace{S'_rS_r\cdots S'_rS_rS'_r}}$ & $\underset{k_r}{\underbrace{S'_rS_r\cdots S'_rS_r}}$ & $\underset{k_r}{\underbrace{S_rS'_r\cdots S_rS'_rS_r}}$ & $\underset{k_r}{\underbrace{S_rS'_r\cdots S_rS'_r}}$ & \  & \  & \  & \  \\
\hline
                $n$ is odd,\ $r$ is even & \  & \  & \  & \  & $\underset{k_r}{\underbrace{S'_rS_r\cdots S'_rS_rS'_r}}$ & $\underset{k_r}{\underbrace{S'_rS_r\cdots S'_rS_r}}$ & $\underset{k_r}{\underbrace{S_rS'_r\cdots S_rS'_rS_r}}$ & $\underset{k_r}{\underbrace{S_rS'_r\cdots S_rS'_r}}$ \\
\hline
                $n$ is even,\ $r$ is even & $\underset{k_r}{\underbrace{S'_rS_r\cdots S'_rS_rS'_r}}$ & $\underset{k_r}{\underbrace{S'_rS_r\cdots S'_rS_r}}$ & $\underset{k_r}{\underbrace{S_rS'_r\cdots S_rS'_rS_r}}$ & $\underset{k_r}{\underbrace{S_rS'_r\cdots S_rS'_r}}$ & \  & \  & \  & \  \\
\hline
                $n$ is even,\ $r$ is odd & \  & \  & \  & \  & $\underset{k_r}{\underbrace{S_rS'_r\cdots S_rS'_rS_r}}$ & $\underset{k_r}{\underbrace{S_rS'_r\cdots S_rS'_r}}$ & $\underset{k_r}{\underbrace{S'_rS_r\cdots S'_rS_rS'_r}}$ & $\underset{k_r}{\underbrace{S'_rS_r\cdots S'_rS_r}}$ \\
                \hline
\end{tabular}}
\end{table}
where $S_r=(C_{r+1},C_{r+2},\ldots,C_{n-1},C_n),\ S'_{r}=(C_{r+1},C_{r+2},\ldots,C_{n-1},C_{n+1})$. Then we mutate
\begin{table}[H] \resizebox{.4\width}{.6\height}{
\begin{tabular}{|c|c|c|c|c|c|c|c|c|}
\hline

\  &\multicolumn{2}{c|}{$(\frac{n-r}{2}+k_r)$ is odd}  &\multicolumn{2}{c|}{$(\frac{n-r}{2}+k_r)$ is even}  &\multicolumn{2}{c|}{$(\frac{n-r+1}{2}+k_r)$ is odd} & \multicolumn{2}{c|}{$(\frac{n-r+1}{2}+k_r)$ is even}  \\
\hline
\  & $k_{r+1}$ is odd & $k_{r+1}$ is even & $k_{r+1}$ is odd & $k_{r+1}$ is even & $k_{r+1}$ is odd & $k_{r+1}$ is even & $k_{r+1}$ is odd & $k_{r+1}$ is even \\
\hline
$n$ is odd,\ $r$ is odd & $\underset{k_{r+1}}{\underbrace{S'_{r+1}S_{r+1}\cdots S'_{r+1}S_{r+1}S'_{r+1}}}$ & $\underset{k_{r+1}}{\underbrace{S'_{r+1}S_{r+1}\cdots S'_{r+1}S_{r+1}}}$ & $\underset{k_{r+1}}{\underbrace{S_{r+1}S'_{r+1}\cdots S_{r+1}S'_{r+1}S_{r+1}}}$ & $\underset{k_{r+1}}{\underbrace{S_{r+1}S'_{r+1}\cdots S_{r+1}S'_{r+1}}}$ & \  & \  & \  & \  \\
\hline
$n$ is odd,\ $r$ is even & \  & \  & \  & \  & $\underset{k_{r+1}}{\underbrace{S'_{r+1}S_{r+1}\cdots S'_{r+1}S_{r+1}S'_{r+1}}}$ & $\underset{k_{r+1}}{\underbrace{S'_{r+1}S_{r+1}\cdots S'_{r+1}S_{r+1}}}$ & $\underset{k_{r+1}}{\underbrace{S_{r+1}S'_{r+1}\cdots S_{r+1}S'_{r+1}S_{r+1}}}$ & $\underset{k_{r+1}}{\underbrace{S_{r+1}S'_{r+1}\cdots S_{r+1}S'_{r+1}}}$ \\
\hline
$n$ is even,\ $r$ is even & $\underset{k_{r+1}}{\underbrace{S'_{r+1}S_{r+1}\cdots S'_{r+1}S_{r+1}S'_{r+1}}}$ & $\underset{k_{r+1}}{\underbrace{S'_{r+1}S_{r+1}\cdots S'_{r+1}S_{r+1}}}$ & $\underset{k_{r+1}}{\underbrace{S_{r+1}S'_{r+1}\cdots S_{r+1}S'_{r+1}S_{r+1}}}$ & $\underset{k_{r+1}}{\underbrace{S_{r+1}S'_{r+1}\cdots S_{r+1}S'_{r+1}}}$ & \  & \  & \  & \  \\
\hline
$n$ is even,\ $r$ is odd & \  & \  & \  & \  & $\underset{k_{r+1}}{\underbrace{S_{r+1}S'_{r+1}\cdots S_{r+1}S'_{r+1}S_{r+1}}}$ & $\underset{k_{r+1}}{\underbrace{S_{r+1}S'_{r+1}\cdots S_{r+1}S'_{r+1}}}$ & $\underset{k_{r+1}}{\underbrace{S'_{r+1}S_{r+1}\cdots S'_{r+1}S_{r+1}S'_{r+1}}}$ & $\underset{k_{r+1}}{\underbrace{S'_{r+1}S_{r+1}\cdots S'_{r+1}S_{r+1}}}$\\
\hline
\end{tabular}}
\end{table}
where $S_{r+1}=(C_{r+2},C_{r+3},\ldots,C_{n-1},C_n),\ S'_{r+1}=(C_{r+2},C_{r+3},\ldots,C_{n-1},C_{n+1})$. Continue this procedure, we mutate
\begin{table}[H] \resizebox{.6\width}{.7\height}{
\begin{tabular}{|c|c|c|c|c|c|c|c|c|}
\hline

\  &\multicolumn{2}{c|}{$(\frac{n-r}{2}+k_r+\cdots+k_{t-1})$ is odd}  &\multicolumn{2}{c|}{$(\frac{n-r}{2}+k_r+\cdots+k_{t-1})$ is even}  &\multicolumn{2}{c|}{$(\frac{n-r+1}{2}+k_r+\cdots+k_{t-1})$ is odd} & \multicolumn{2}{c|}{$(\frac{n-r+1}{2}+k_r+\cdots+k_{t-1})$ is even}  \\
\hline
                \  & $k_t$ is odd & $k_t$ is even & $k_t$ is odd & $k_t$ is even & $k_t$ is odd & $k_t$ is even & $k_t$ is odd & $k_t$ is even \\
\hline
                $n$ is odd,\ $r$ is odd & $\underset{k_t}{\underbrace{S'_tS_t\cdots S'_tS_tS'_t}}$ & $\underset{k_t}{\underbrace{S'_tS_t\cdots S'_tS_t}}$ & $\underset{k_t}{\underbrace{S_tS'_t\cdots S_tS'_tS_t}}$ & $\underset{k_t}{\underbrace{S_tS'_t\cdots S_tS'_t}}$ & \  & \  & \  & \  \\
\hline
                $n$ is odd,\ $r$ is even & \  & \  & \  & \  & $\underset{k_t}{\underbrace{S'_tS_t\cdots S'_tS_tS'_t}}$ & $\underset{k_t}{\underbrace{S'_tS_t\cdots S'_tS_t}}$ & $\underset{k_t}{\underbrace{S_tS'_t\cdots S_tS'_tS_t}}$ & $\underset{k_t}{\underbrace{S_tS'_t\cdots S_tS'_t}}$ \\
\hline
                $n$ is even,\ $r$ is even & $\underset{k_t}{\underbrace{S'_tS_t\cdots S'_tS_tS'_t}}$ & $\underset{k_t}{\underbrace{S'_tS_t\cdots S'_tS_t}}$ & $\underset{k_t}{\underbrace{S_tS'_t\cdots S_tS'_tS_t}}$ & $\underset{k_t}{\underbrace{S_tS'_t\cdots S_tS'_t}}$ & \  & \  & \  & \  \\
\hline
                $n$ is even,\ $r$ is odd & \  & \  & \  & \  & $\underset{k_t}{\underbrace{S_tS'_t\cdots S_tS'_tS_t}}$ & $\underset{k_t}{\underbrace{S_tS'_t\cdots S_tS'_t}}$ & $\underset{k_t}{\underbrace{S'_tS_t\cdots S'_tS_tS'_t}}$ & $\underset{k_t}{\underbrace{S'_tS_t\cdots S'_tS_t}}$ \\
                \hline
\end{tabular}}
\end{table}
where $S_t=(C_{t+1},C_{t+2},\ldots,C_{n-1},C_n),\ S'_t=(C_{t+1},C_{t+2},\ldots,C_{n-1},C_{n+1})$ and $t=r+2, r+3, \ldots,n-2$.

We define
\begin{gather}\label{variables 7}
\begin{split}
\widetilde{t}^{(-4k_{n}-2(\sum_{i=r}^{n-1}k_{i})-n+r+[r]_{\text{mod} 2}+1)}_{k_1,\ldots,k_p,0,\ldots,0,\underset{q}1,0,\ldots,0,k_n}=&
\widetilde{t}'^{(-4k_{n}-2(\sum_{i=r}^{n-1}k_{i})-n+r+[r]_{\text{mod} 2}+5)}_{k_1,\ldots,k_p-1,0,\ldots,0,\ldots,0,k_n}\quad (1\leq p<q\leq n-1,\ 0\leq k_1+\cdots+k_p\leq n-q+1),\\
\widetilde{t}^{(-4k_{n}-2(\sum_{i=r}^{n-1}k_{i})-n+r+[r]_{\text{mod} 2}+1)}_{k_1,\ldots,k_p,0,\ldots,0,k_q,0,\ldots,0,k_n}=&
\widetilde{t}'^{(-4k_{n}-2(\sum_{i=r}^{n-1}k_{i})-n+r+[r]_{\text{mod} 2}+5)}_{k_1,\ldots,k_p,0,\ldots,0,k_q-2,0,\ldots,0,k_n}\quad (1\leq p<q\leq n-1,\ k_q\geq2,\ 0\leq k_1+\cdots+k_p\leq n-q+1),\\
\widetilde{s}^{(-4k_{n}-2(\sum_{i=r}^{n-1}k_{i})-n+r+[r]_{\text{mod} 2}+1)}_{k_1,\ldots,k_p,0,\ldots,0,\underset{q}1,0,\ldots,0,2k_n,0}=& \widetilde{t}^{(-4k_{n}-2(\sum_{i=r}^{n-1}k_{i})-n+r+[r]_{\text{mod} 2}+1)}_{k_1,\ldots,k_p,0,\ldots,0,\underset{q}1,0,\ldots,0,2k_n,0}\widetilde{t}^{(-2(\sum_{i=r}^{n-2}k_{i})-n+r+[r]_{\text{mod} 2}+5)}_{k_1,\ldots,k_p-1,0,\ldots,0,\ldots,0,0}\quad (1\leq p<q\leq n-2,\ 0\leq k_1+\cdots+k_p\leq n-q+1),\\
\widetilde{s}^{(-4k_{n}-2(\sum_{i=r}^{n-1}k_{i})-n+r+[r]_{\text{mod} 2}+1)}_{k_1,\ldots,k_p,0,\ldots,0,k_q,0,\ldots,0,2k_n,0}=& \widetilde{t}^{(-4k_{n}-2(\sum_{i=r}^{n-1}k_{i})-n+r+[r]_{\text{mod} 2}+1)}_{k_1,\ldots,k_p,0,\ldots,0,k_q,0,\ldots,0,2k_n,0}\widetilde{t}^{(-2(\sum_{i=r}^{n-2}k_{i})-n+r+[r]_{\text{mod} 2}+5)}_{k_1,\ldots,k_p,0,\ldots,0,k_q-2,0,\ldots,0,0}\quad (1\leq p<q\leq n-2,\ k_q\geq2,\ 0\leq k_1+\cdots+k_p\leq n-q+1),
\end{split}
\end{gather}
where
\begin{gather}\label{variables 8}
\begin{split}
&\widetilde{t}'^{(-4k_{n}-2(\sum_{i=r}^{n-1}k_{i})-n+r+[r]_{\text{mod} 2}+5)}_{k_1,\ldots,k_m,0,\ldots,0,k_{n-1}-2,k_n}\\
&=\frac{\widetilde{t}^{(-4k_{n}-2(\sum_{i=r}^{n-1}k_{i})-n+r+[r]_{\text{mod} 2}+5)}_{k_1,\ldots,k_m,0,\ldots,0,k_{n-1},k_{n}-1}\widetilde{t}^{(-4k_{n}-2(\sum_{i=r}^{n-1}k_{i})-n+r+[r]_{\text{mod} 2}+1)}_{k_1,\ldots,k_m,0,\ldots,0,k_{n-1}-2,,k_{n}+1}+
\widetilde{t}^{(-4k_{n}-2(\sum_{i=r}^{n-1}k_{i})-n+r+[r]_{\text{mod} 2}+1)}_{k_1,\ldots,k_m,0,\ldots,0,2k_n+k_{n-1},0}
\widetilde{t}^{(-2(\sum_{i=r}^{n-2}k_{i})-n+r+[r]_{\text{mod} 2}+5)}_{k_1,\ldots,k_m,0,\ldots,0,k_{n-1}-2,0}}{\widetilde{t}^{(-4k_{n}-2(\sum_{i=r}^{n-1}k_{i})-n+r+[r]_{\text{mod} 2}+5)}_{k_1,\ldots,k_m,0,\ldots,0,k_{n-1}-2,k_n}},
\end{split}
\end{gather}
where $k_{n-1}\geq2,\ 1\leq m\leq n-2$ and $0\leq k_1+\cdots+k_m\leq2;$
\begin{gather}\label{variables 9}
\begin{split}
&\widetilde{t}'^{(-4k_{n}-2(\sum_{i=r}^{n-1}k_{i})-n+r+[r]_{\text{mod} 2}+5)}_{k_1,\ldots,k_m-1,0,\ldots,0,k_n}\\
&=\frac{\widetilde{t}^{(-4k_{n}-2(\sum_{i=r}^{n-1}k_{i})-n+r+[r]_{\text{mod} 2}+5)}_{k_1,\ldots,k_m,0,\ldots,0,1,k_n-1}\widetilde{t}^{(-4k_{n}-2(\sum_{i=r}^{n-1}k_{i})-n+r+[r]_{\text{mod} 2}+1)}_{k_1,\ldots,k_m-1,0,\ldots,0,k_n+1}
+\widetilde{t}^{(-2(\sum_{i=r}^{n-2}k_{i})-n+r+[r]_{\text{mod} 2}+5)}_{k_1,\ldots,k_m-1,0,\ldots,0,0} \widetilde{t}^{(-4k_{n}-2(\sum_{i=r}^{n-1}k_{i})-n+r+[r]_{\text{mod} 2}+1)}_{k_1,\ldots,k_m,0,\ldots,0,2k_n+1,0}}{\widetilde{t}^{(-4k_{n}-2(\sum_{i=r}^{n-1}k_{i})-n+r+[r]_{\text{mod} 2}+5)}_{k_1,\ldots,k_m-1,0,\ldots,0,k_n}},
\end{split}
\end{gather}
where $1\leq m\leq n-2,\ 0\leq k_1+\cdots+k_m\leq2;$
\begin{gather}\label{variables 10}
\begin{split}
&\widetilde{t}'^{(-4k_{n}-2(\sum_{i=r}^{n-1}k_{i})-n+r+[r]_{\text{mod} 2}+5)}_{k_1,\ldots,k_m,0,\ldots,0,k_l-2,0,\ldots,0,k_n}\\
&=\frac{\widetilde{t}^{(-4k_{n}-2(\sum_{i=r}^{n-1}k_{i})-n+r+[r]_{\text{mod} 2}+5)}_{k_1,\ldots,k_m,0,\ldots,0,k_l,0,\ldots,0,k_n-1}\widetilde{t}^{(-4k_{n}-2(\sum_{i=r}^{n-1}k_{i})-n+r+[r]_{\text{mod} 2}+1)}_{k_1,\ldots,k_m,0,\ldots,0,k_l-2,0,\ldots,0,k_n+1}
+\widetilde{s}^{(-4k_{n}-2(\sum_{i=r}^{n-1}k_{i})-n+r+[r]_{\text{mod} 2}+1)}_{k_1,\ldots,k_m,0,\ldots,0,k_l,0,\ldots,0,2k_n,0}}{\widetilde{t}^{(-4k_{n}-2(\sum_{i=r}^{n-1}k_{i})-n+r+[r]_{\text{mod} 2}+5)}_{k_1,\ldots,k_m,0,\ldots,0,k_l-2,0,\ldots,0,k_n}},
\end{split}
\end{gather}
where $1\leq m<l\leq n-2,\ k_l\geq2,\ 0\leq k_1+\cdots+k_m\leq n-l+1;$
\begin{gather}\label{variables 11}
\begin{split}
&\widetilde{t}'^{(-4k_{n}-2(\sum_{i=r}^{n-1}k_{i})-n+r+[r]_{\text{mod} 2}+5)}_{k_1,\ldots,k_m-1,0,\ldots,0,\ldots,0,k_n}\\
&=\frac{\widetilde{t}^{(-4k_{n}-2(\sum_{i=r}^{n-1}k_{i})-n+r+[r]_{\text{mod} 2}+5)}_{k_1,\ldots,k_m,0,\ldots,0,\underset{l}1,0,\ldots,0,k_n-1}\widetilde{t}^{(-4k_{n}-2(\sum_{i=r}^{n-1}k_{i})-n+r+[r]_{\text{mod} 2}+1)}_{k_1,\ldots,k_m-1,0,\ldots,0,\ldots,0,k_n+1}
+\widetilde{s}^{(-4k_{n}-2(\sum_{i=r}^{n-1}k_{i})-n+r+[r]_{\text{mod} 2}+1)}_{k_1,\ldots,k_m,0,\ldots,0,\underset{l}1,0,\ldots,0,2k_n,0}}{\widetilde{t}^{(-4k_{n}-2(\sum_{i=r}^{n-1}k_{i})-n+r+[r]_{\text{mod} 2}+5)}_{k_1,\ldots,k_m-1,0,\ldots,0,\ldots,0,k_n}},
\end{split}
\end{gather}
where $1\leq m<l\leq n-2,\ 0\leq k_1+\cdots+k_m\leq n-l+1;$ are mutation equations which occur when we mutate $\seq$. The variables (\ref{variables 7}) are defined in the order according to the mutation sequence $\seq$. In this order, every variable in (\ref{variables 7}) is defined by an equation of (\ref{variables 8})--(\ref{variables 11}) using variables in ${\bf t}$ and those variables in (\ref{variables 7}) which are already defined.

\subsection{The equations in the system in Theorem \ref{M system of type Cn} correspond to mutations in the cluster algebra $\mathscr{A}$}

By (\ref{variables 2}), we have

\begin{gather}\label{eqn10}
\begin{split}
&t^{(s)}_{k_{1},k_{2},k_{3},\ldots,k_{n-1},0}=t'^{(s)}_{k_{1},k_{2}-1,k_3,\ldots,k_{n-1},0}\\
&=\frac{t^{(s)}_{k_{1}-1,k_{2},k_3,\ldots,k_{n-1},0}t^{(s)}_{k_{1}+1,k_{2}-1,k_3,\ldots,k_{n-1},0}+
t^{(s)}_{0,k_{1}+k_{2},k_{3},\ldots,k_{n-1},0}
t^{(s)}_{0,k_{2}-1,k_{3},\ldots,k_{n-1},0}}{t^{(s)}_{k_{1},k_{2}-1,k_3,\ldots,k_{n-1},0}},
\end{split}
\end{gather}
where $s =-n+l-[l]_{\text{mod} 2}+1$. Equations (\ref{eqn10}) correspond to Equations (\ref{eqn1}) in the system in Theorem \ref{M system of type Cn}.

By (\ref{variables 3}),  we have
\begin{gather}\label{eqn20}
\begin{split}
&t^{(s)}_{0,\ldots, 0,k_{i},k_{i+1},k_{i+2},\ldots,k_{n-1},0}=t'^{(s)}_{0,\ldots,0,k_{i},k_{i+1}-1,k_{i+2},\ldots,k_{n-1},0}\\
&=\frac{t^{(s)}_{0,\ldots,0,k_{i}-1,k_{i+1},k_{i+2},\ldots,k_{n-1},0}
t^{(s)}_{0,\ldots, 0,k_{i}+1,k_{i+1}-1,k_{i+2}\ldots,k_{n-1},0}+
t^{(s)}_{0,\ldots,0,\underset{i+1}{k_{i}+k_{i+1}},k_{i+2},\ldots,k_{n-1},0}t^{(s)}_{0,\ldots,0,\underset{i-1}{k_{i}},0,\underset{i+1}{k_{i+1}-1},\ldots,k_{n-1},0}}{t^{(s)}_{0,\ldots,0,k_{i},k_{i+1}-1,k_{i+2},\ldots,k_{n-1},0}},
\end{split}
\end{gather}
where $s =-n+l-[l]_{\text{mod} 2}+1$. Equations (\ref{eqn20}) correspond to Equations (\ref{eqn2}) in the system in Theorem \ref{M system of type Cn}.

By \ref{variables 4},  we have
\begin{gather}\label{eqn30}
\begin{split}
&t^{(s)}_{k_{1},0,\ldots,0,k_{j},k_{j+1},\ldots,k_{n-1},0}=t'^{(s)}_{k_{1},0,\ldots,0,k_{j}-1,k_{j+1},\ldots,k_{n-1},0}\\
&=\frac{t^{(s)}_{k_{1}+1,0,\ldots,0,k_{j}-1,k_{j+1},\ldots,k_{n-1},0}
t^{(s)}_{k_{1}-1,0,\ldots,0,k_{j},k_{j+1}\ldots,k_{n-1},0}+
t^{(s)}_{0,k_{1},0,\ldots,0,k_{j},k_{j+1},\ldots,k_{n-1},0}t^{(s)}_{0,\ldots,0,k_{j}-1,k_{j+1},\ldots,k_{n-1},0}}{t^{(s)}_{k_{1},0,\ldots,0,k_{j}-1,k_{j+1},\ldots,k_{n-1},0}},
\end{split}
\end{gather}
where $s =-n+l-[l]_{\text{mod} 2}+1$. Equations (\ref{eqn30}) correspond to Equations (\ref{eqn3}) in the system in Theorem \ref{M system of type Cn}.

By (\ref{variables 5}) we have

\begin{gather}\label{eqn40}
\begin{split}
&t^{(s)}_{0,\ldots, 0,\underset{i}{k_{i}},0,\ldots,0,\underset{j}{k_{j}},k_{j+1},\ldots,k_{n-1},0}=t'^{(s)}_{0,\ldots,0,\underset{i}{k_{i}},0,\ldots,0,\underset{j}{k_{j}-1},k_{j+1},\ldots,k_{n-1},0}\\
&=\frac{t^{(s)}_{0,\ldots, 0,\underset{i}{k_{i}-1},0,\ldots,0,\underset{j}{k_{j}},k_{j+1},\ldots,k_{n-1},0}
t^{(s)}_{0,\ldots, 0,\underset{i}{k_{i}+1},0,\ldots,0,\underset{j}{k_{j}-1},k_{j+1}\ldots,k_{n-1},0}+
t^{(s)}_{0,\ldots,0,\underset{i+1}{k_{i}},0, \ldots,0,\underset{j}{k_{j}},k_{j+1},\ldots,k_{n-1},0}
t^{(s)}_{0,\ldots,0,\underset{i-1}{k_{i}},0,\ldots,0,\underset{j}{k_{j}-1},\ldots,k_{n-1},0}}{t^{(s)}_{0,\ldots,0,\underset{i}{k_{i}},0,\ldots,0,\underset{j}{k_{j}-1},k_{j+1},\ldots,k_{n-1},0}},
\end{split}
\end{gather}
where $s =-n+l-[l]_{\text{mod} 2}+1$. Equations (\ref{eqn40}) correspond to Equations (\ref{eqn4}) in the system in Theorem \ref{M system of type Cn}.

By (\ref{variables 8}), we have

\begin{gather}\label{variables 17}
\begin{split}
&\widetilde{t}^{(s-4)}_{k_1,\ldots,k_m,0,\ldots,0,k_{n-1},k_n} = \widetilde{t}'^{(s)}_{k_1,\ldots,k_m,0,\ldots,0,k_{n-1}-2,k_n}\\
&=\frac{\widetilde{t}^{(s)}_{k_1,\ldots,k_m,0,\ldots,0,k_{n-1},k_{n}-1}\widetilde{t}^{(s-4)}_{k_1,\ldots,k_m,0,\ldots,0,k_{n-1}-2,,k_{n}+1}+
\widetilde{t}^{(s-4)}_{k_1,\ldots,k_m,0,\ldots,0,2k_n+k_{n-1},0}
\widetilde{t}^{(s+4k_n)}_{k_1,\ldots,k_m,0,\ldots,0,k_{n-1}-2,0}}{\widetilde{t}^{(s)}_{k_1,\ldots,k_m,0,\ldots,0,k_{n-1}-2,k_n}},
\end{split}
\end{gather}
where $s = {-4k_{n}-2(\sum_{i=r}^{n-1}k_{i})-n+r+[r]_{\text{mod} 2}+5}$. Equations (\ref{variables 17}) correspond to Equations (\ref{eqn512}) in the system in Theorem \ref{M system of type Cn}.

By (\ref{variables 9}), we have

\begin{gather}\label{variables 18}
\begin{split}
&\widetilde{t}^{(s-4)}_{k_1,\ldots,k_m,0,\ldots,0,1,k_n} =\widetilde{t}'^{(s)}_{k_1,\ldots,k_m-1,0,\ldots,0,k_n}\\
&=\frac{\widetilde{t}^{(s)}_{k_1,\ldots,k_m,0,\ldots,0,1,k_n-1}\widetilde{t}^{(s-4)}_{k_1,\ldots,k_m-1,0,\ldots,0,k_n+1}
+\widetilde{t}^{(s+4k_n)}_{k_1,\ldots,k_m-1,0,\ldots,0,0} \widetilde{t}^{(s-4)}_{k_1,\ldots,k_m,0,\ldots,0,{2k_n+1,0}}}{t^{(s)}_{k_1,\ldots,k_m-1,0,\ldots,0,k_n}},
\end{split}
\end{gather}
where $s = {-4k_{n}-2(\sum_{i=r}^{n-1}k_{i})-n+r+[r]_{\text{mod} 2}+5}$. Equations (\ref{variables 18}) correspond to Equations (\ref{eqn511}) in the system in Theorem \ref{M system of type Cn}.

By (\ref{variables 10}), we have

\begin{gather}\label{variables 19}
\begin{split}
&\widetilde{t}^{(s-4)}_{k_1,\ldots,k_m,0,\ldots,0,k_l,0,\ldots,0,k_n}=\widetilde{t}'^{(s)}_{k_1,\ldots,k_m,0,\ldots,0,k_l-2,0,\ldots,0,k_n}\\
&=\frac{\widetilde{t}^{(s)}_{k_1,\ldots,k_m,0,\ldots,0,k_l,0,\ldots,0,k_n-1}\widetilde{t}^{(s-4)}_{k_1,\ldots,k_m,0,\ldots,0,k_l-2,0,\ldots,0,k_n+1}
+\widetilde{s}^{(s-4)}_{k_1,\ldots,k_m,0,\ldots,0,k_l,0,\ldots,0,2k_n,0}}{\widetilde{t}^{(s)}_{k_1,\ldots,k_m,0,\ldots,0,k_l-2,0,\ldots,0,k_n}},
\end{split}
\end{gather}
where $s = {-4k_{n}-2(\sum_{i=r}^{n-1}k_{i})-n+r+[r]_{\text{mod} 2}+5}$. Equations (\ref{variables 19}) correspond to Equations (\ref{eqn5221}) in the system in Theorem \ref{M system of type Cn}.

By (\ref{variables 11}), we have

\begin{gather}\label{variables 18}
\begin{split}
&\widetilde{t}^{(s-4)}_{k_1,\ldots,k_m,0,\ldots,0,\underset{l}1,0,\ldots,0,k_n}=\widetilde{t}'^{(s)}_{k_1,\ldots,k_m-1,0,\ldots,0,\ldots,0,k_n}\\
&=\frac{\widetilde{t}^{(s)}_{k_1,\ldots,k_m,0,\ldots,0,\underset{l}1,0,\ldots,0,k_n-1}\widetilde{t}^{(s-4)}_{k_1,\ldots,k_m-1,0,\ldots,0,\ldots,0,k_n+1}
+\widetilde{s}^{(s-4)}_{k_1,\ldots,k_m,0,\ldots,0,\underset{l}1,0,\ldots,0,2k_n,0}}{\widetilde{t}^{(s)}_{k_1,\ldots,k_m-1,0,\ldots,0,\ldots,0,k_n}},
\end{split}
\end{gather}
where $s = {-4k_{n}-2(\sum_{i=r}^{n-1}k_{i})-n+r+[r]_{\text{mod} 2}+5}$. Equations (\ref{variables 18}) correspond to Equations (\ref{eqn5211}) in the system in Theorem \ref{M system of type Cn}.





Therefore we have the following theorem.
\begin{theorem}\label{connection with cluster algebra 1}
Every equation in the system in Theorem \ref{M system of type Cn} corresponds to a mutation equation of the cluster algebra $\mathscr{A}$. Each minimal affinization in Theorem \ref{M system of type Cn} corresponds to a cluster variable in $\mathscr{A}$ defined in Section \ref{definition of cluster algebra A}. Therefore the Hernandez-Leclerc conjecture (Conjecture 13.2 in \cite{HL10} and Conjecture 9.1 in \cite{Le10}) is true for the minimal affinizations in Theorem \ref{M system of type Cn}.
\end{theorem}

\section{The dual system of Theorem \ref{M system of type Cn}} \label{dual M system}
In what follows, we study the dual system of the system in Theorem \ref{M system of type Cn}.

\begin{theorem}[Theorem 3.9, \cite{Her07}]
For $s\in \mathbb{Z}, k_1,\ldots,k_{n-1}\in \mathbb{Z}_{\geq 0}$, the module $\widetilde{\mathcal T}_{k_1, k_2,\ldots, k_{n-1},0}^{(s)}$ is anti-special.
\end{theorem}

\begin{theorem}\label{anti-special}
For $k_1,\ldots,k_{n-1}\in \mathbb{Z}_{\geq0}$, $k_n\in \mathbb{Z}_{\geq 1}$ and $s\in \mathbb{Z}$,  the modules
\begin{align*}
&\mathcal T^{(s)}_{0,\ldots,0,k_{n-i},0,\ldots,0,k_n}\ (1\leq i\leq n-1),\\
&\mathcal T^{(s)}_{0,\ldots,0,k_{n-j},0,\ldots,0,k_{n-i},0,\ldots,0,k_n}\ (1\leq i<j\leq n-1,\ 0\leq k_{n-j}\leq i+1) , \\
&\mathcal T^{(s)}_{k_1,\ldots,k_m,0,\ldots,0,k_l,0,\ldots,0,k_n}\ (1\leq m<l\leq n-1,\ 0\leq k_1+\cdots+k_m\leq n-l+1),\\
&\mathcal S^{(s)}_{k_1,\ldots,k_m,0,\ldots,0,k_l,0,\ldots,0,k_{n-1},0}\ (1\leq m<l\leq n-2,\ 0\leq k_1+\cdots+k_m\leq n-l+1),
\end{align*}
are anti-special.
\end{theorem}
\begin{proof}
The proof of the theorem follows from dual arguments in the proof of Theorem \ref{The special modules of $C_n$}.
\end{proof}

\begin{lemma}\label{lemma2}
Let $\iota: \mathbb{Z}\mathcal{P}\rightarrow \mathbb{Z}\mathcal{P}$ be a homomorphism of rings such that $Y_{1, aq^{s}}\mapsto Y_{1, aq^{2n-s+2}}^{-1}$, $Y_{2, aq^{s}}\mapsto Y_{2, aq^{2n-s+2}}^{-1},\ldots, Y_{n, aq^{s}}\mapsto Y_{n, aq^{2n-s+2}}^{-1}$ for all $a\in \mathbb{C}^{\times}, s\in \mathbb{Z}$. Then
\begin{align*}
\chi_{q}(\widetilde{\mathcal T}_{k_{1},k_{2},\ldots,k_{n}}^{(s)})=\iota(\chi_{q}(\mathcal T_{k_{1},k_{2},\ldots,k_{n}}^{(s)})),\  \chi_{q}(\mathcal T_{k_{1},k_{2},\ldots,k_{n}}^{(s)})=\iota(\chi_{q}(\widetilde{\mathcal T}_{k_{1},k_{2},\ldots,k_{n}}^{(s)})).
\end{align*}
\end{lemma}

\begin{proof}
The proof is similar to the proof of Lemma 7.3 in \cite{LM13}.
\end{proof}

\begin{theorem}\label{dual-M-system}
Let $s\in \mathbb{Z}$, $k_{1},k_{2},\ldots, k_{n} \in \mathbb{Z}_{\geq 0}$. We have

\begin{gather}\label{eqn6}
\begin{split}
[\widetilde{\mathcal T}^{(s)}_{k_{1},k_{2}-1,k_{3},\ldots,k_{n-1},0}][\widetilde{\mathcal T}^{(s)}_{k_{1},k_{2},k_{3},\ldots,k_{n-1},0}]=&
[\widetilde{\mathcal T}^{(s)}_{k_{1}-1,k_{2},k_{3},\ldots,k_{n-1},0}][\widetilde{\mathcal T}^{(s)}_{k_{1}+1,k_{2}-1,k_{3},\ldots,k_{n-1},0}]\\
&+[\widetilde{\mathcal T}^{(s)}_{0,k_{2}-1,k_{3},\ldots,k_{n-1},0}][\widetilde{\mathcal T}^{(s)}_{0,k_{1}+k_{2},k_{3},\ldots,k_{n-1},0}],
\end{split}
\end{gather}
where $k_{1},k_{2}>0;$

\begin{gather}\label{eqn7}
\begin{split}
[\widetilde{\mathcal T}^{(s)}_{0,\ldots,0,k_{i},k_{i+1}-1,k_{i+2},\ldots,k_{n-1},0}][\widetilde{\mathcal T}^{(s)}_{0,\ldots, 0,k_{i},k_{i+1},k_{i+2},\ldots,k_{n-1},0}]=
&[\widetilde{\mathcal T}^{(s)}_{0,\ldots, 0,k_{i}+1,k_{i+1}-1,k_{i+2},\ldots,k_{n-1},0}][\widetilde{\mathcal T}^{(s)}_{0,\ldots, 0,k_{i}-1,k_{i+1},k_{i+2}\ldots,k_{n-1},0}]\\
&+[\widetilde{\mathcal T}^{(s)}_{0,\ldots,0,\underset{i+1}{k_{i}+k_{i+1}},k_{i+2},\ldots,k_{n-1},0}][\widetilde{\mathcal T}^{(s)}_{0,\ldots,0,\underset{i-1}{k_{i}},0,\underset{i+1}{k_{i+1}-1},k_{i+2},,\ldots,k_{n-1},0}],
\end{split}
\end{gather}
where $k_{i},k_{i+1}>0$, $1 < i \leq n-2;$

\begin{gather}\label{eqn8}
\begin{split}
[\widetilde{\mathcal T}^{(s)}_{k_{1},0,\ldots,0,k_{j}-1,k_{j+1},\ldots,k_{n-1},0}][\widetilde{\mathcal T}^{(s)}_{k_{1},0,\ldots,0,k_{j},k_{j+1},\ldots,k_{n-1},0}]=
&[\widetilde{\mathcal T}^{(s)}_{k_{1}+1,0,\ldots,0,k_{j}-1,k_{j+1},\ldots,k_{n-1},0}]
[\widetilde{\mathcal T}^{(s)}_{k_{1}-1,0,\ldots, 0,k_{j},k_{j+1}\ldots,k_{n-1},0}]\\
&+[\widetilde{\mathcal T}^{(s)}_{0,k_{1},0,\ldots,0,k_{j},k_{j+1},\ldots,k_{n-1},0}][\widetilde{\mathcal T}^{(s)}_{0,\ldots,0,k_{j}-1,k_{j+1},\ldots,k_{n-1},0}],
\end{split}
\end{gather}
where $2< j\leq n-1;$

\begin{gather}\label{eqn9}
\begin{split}
[\widetilde{\mathcal T}^{(s)}_{0,\ldots,0,\underset{i}{k_{i}},0,\ldots,0,\underset{j}{k_{j}-1},k_{j+1},\ldots,k_{n-1},0}][\widetilde{\mathcal T}^{(s)}_{0,\ldots, 0,\underset{i}{k_{i}},0,\ldots,0,\underset{j}{k_{j}},k_{j+1},\ldots,k_{n-1},0}]=
&[\widetilde{\mathcal T}^{(s)}_{0,\ldots, 0,\underset{i}{k_{i}-1},0,\ldots,0,\underset{j}{k_{j}},k_{j+1},\ldots,k_{n-1},0}] [\widetilde{\mathcal T}^{(s)}_{0,\ldots, 0,\underset{i}{k_{i}+1},0,\ldots,0,\underset{j}{k_{j}-1},k_{j+1}\ldots,k_{n-1},0}]\\
&+[\widetilde{\mathcal T}^{(s)}_{0,\ldots,0,\underset{i+1}{k_{i}},0, \ldots, 0,\underset{j}{k_{j}},k_{j+1},\ldots,k_{n-1},0}][\widetilde{\mathcal T}^{(s)}_{0,\ldots,0,\underset{i-1}{k_{i}},0,\ldots, 0,\underset{j}{k_{j}-1},k_{j+1},\ldots,k_{n-1},0}],
\end{split}
\end{gather}
where $2< i+1< j \leq n-1.$
\vskip 0.3cm
In the following, let $k_{1},\ldots, k_{n-1}\in  \mathbb{Z}_{\geq 0},\ k_n\in \mathbb{Z}_{>0}$ and $s\in \mathbb{Z}$.

\begin{gather}\label{eqn511d}
\begin{split}
[\mathcal T^{(s)}_{k_1,\ldots,k_m-1,0,\ldots,0,k_n}][\mathcal T^{(s-4)}_{k_1,\ldots,k_m,0,\ldots,0,1,k_n}]=&[\mathcal T^{(s)}_{k_1,\ldots,k_m,0,\ldots,0,1,k_n-1}][\mathcal T^{(s-4)}_{k_1,\ldots,k_m-1,0,\ldots,0,k_n+1}]\\
&+[\mathcal T^{(s+4k_n)}_{k_1,\ldots,k_m-1,0,\ldots,0,0}][\mathcal T^{(s-4)}_{k_1,\ldots,k_m,0,\ldots,0,{2k_n+1},0}],
\end{split}
\end{gather}
where $1\leq m\leq n-2$ and $0\leq k_1+\cdots+k_m\leq2;$

\begin{gather}\label{eqn512d}
\begin{split}
[\mathcal T^{(s)}_{k_1,\ldots,k_m,0,\ldots,0,k_{n-1}-2,k_n}][\mathcal T^{(s-4)}_{k_1,\ldots,k_m,0,\ldots,0,k_{n-1},k_n}]=&[\mathcal T^{(s)}_{k_1,\ldots,k_m,0,\ldots,0,k_{n-1},k_{n}-1}][\mathcal T^{(s-4)}_{k_1,\ldots,k_m,0,\ldots,0,k_{n-1}-2,,k_{n}+1}]\\
&+[\mathcal T^{(s+4k_n)}_{k_1,\ldots,k_m,0,\ldots,0,k_{n-1}-2,0}][\mathcal T^{(s-4)}_{k_1,\ldots,k_m,0,\ldots,0,2k_n+k_{n-1},0}],
\end{split}
\end{gather}
where $1\leq m\leq n-2,\ k_{n-1}\geq2$ and $0\leq k_1+\cdots+k_m\leq2;$

\begin{gather}\label{eqn5211d}
\begin{split}
[\mathcal T^{(s)}_{k_1,\ldots,k_m-1,0,\ldots,0,\ldots,0,k_n}][\mathcal T^{(s-4)}_{k_1,\ldots,k_m,0,\ldots,0,\underset{l}1,0,\ldots,0,k_n}]=&[\mathcal T^{(s)}_{k_1,\ldots,k_m,0,\ldots,0,\underset{l}1,0,\ldots,0,k_n-1}][\mathcal T^{(s-4)}_{k_1,\ldots,k_m-1,0,\ldots,0,\ldots,0,k_n+1}]\\
&+[\mathcal S^{(s-4)}_{k_1,\ldots,k_m,0,\ldots,0,\underset{l}1,0,\ldots,0,2k_n,0}],
\end{split}
\end{gather}
where $1\leq m<l\leq n-2,\ 0\leq k_1+\cdots+k_m\leq n-l+1;$


\begin{gather}\label{eqn5221d}
\begin{split}
[\mathcal T^{(s)}_{k_1,\ldots,k_m,0,\ldots,0,k_l-2,0,\ldots,0,k_n}][\mathcal T^{(s-4)}_{k_1,\ldots,k_m,0,\ldots,0,k_l,0,\ldots,0,k_n}]=&[\mathcal T^{(s)}_{k_1,\ldots,k_m,0,\ldots,0,k_l,0,\ldots,0,k_n-1}][\mathcal T^{(s-4)}_{k_1,\ldots,k_m,0,\ldots,0,k_l-2,0,\ldots,0,k_n+1}]\\
&+[\mathcal S^{(s-4)}_{k_1,\ldots,k_m,0,\ldots,0,k_l,0,\ldots,0,2k_n,0}],
\end{split}
\end{gather}
where $1\leq m<l\leq n-2,\ k_l\geq2$ and $0\leq k_1+\cdots+k_m\leq n-l+1$.
\end{theorem}

\begin{proof}
The lowest weight monomial of $ \chi_q(\widetilde{\mathcal{T}}_{k_{1},\ldots,k_{n}}^{(s)})$ is obtained from the highest weight monomial of $ \chi_q(\widetilde{\mathcal{T}}_{k_{1},\ldots,k_{n}}^{(s)}) $ by the substitutions:
 \begin{align*}
1_s \mapsto 1^{-1}_{2n+s+2},\ 2_s \mapsto 2^{-1}_{2n+s+2}, \ldots, n_s \mapsto n^{-1}_{2n+s+2}.
\end{align*}
After we apply $\iota$ to $\chi_{q}(\widetilde{\mathcal T}_{k_{1},\ldots,k_{n}}^{(s)})$, the lowest weight monomial of
$\chi_q(\widetilde{\mathcal{T}}_{k_{1},\ldots,k_{n}}^{(s)})$
becomes the highest weight monomial of $\iota(\chi_{q}(\widetilde{\mathcal T}_{k_{1},\ldots,k_{n}}^{(s)})$. Therefore by Lemma \ref{lemma2}, the highest weight monomial of $\iota(\chi_{q}(\widetilde{\mathcal T}_{k_{1},\ldots,k_{n}}^{(s)})$ is obtained from the lowest weight monomial of $ \chi_q(\widetilde{\mathcal{T}}_{k_{1},\ldots,k_{n}}^{(s)})$ by the substitutions:
\begin{align*}
1_s \mapsto 1^{-1}_{2n-s+2},\ 2_s \mapsto 2^{-1}_{2n-s+2}, \ldots,\ n_s \mapsto n^{-1}_{2n-s+2}.
\end{align*}
It follows that the highest weight monomial of $\iota(\chi_{q}(\widetilde{\mathcal T}_{k_{1},\ldots,k_{n}}^{(s)})$ is obtained from the highest weight monomial of $ \chi_q(\widetilde{\mathcal{T}}_{k_{1},\ldots,k_{n}}^{(s)}) $ by the substitutions:
\begin{align*}
1_s \mapsto 1_{-s},\ 2_s \mapsto 2_{-s}, \ldots,\ n_s \mapsto n_{-s}.
\end{align*}
Therefore the dual system of type $C_n$ is obtained applying $\iota$ to both sides of every equation of the system in Theorem \ref{M system of type Cn}.

The simplify of every module in the summands on the right hand side of every equation in the dual system follows from Theorem $\ref{irreducible}$ and Lemma $\ref{lemma2}$.
\end{proof}

\begin{example}
The following are some equations in the system in Theorem \ref{dual-M-system}.
\begin{align*}
&[1_{2}][2_{1}1_{4}]=[1_{2}1_{4}][2_{1}]+[2_{1}2_{3}],\\
&[1_{2}1_{4}][2_{1}1_{4}1_{6}]=[2_{1}1_{4}][1_{2}1_{4}1_{6}]+[2_{1}2_{3}2_{5}],\\
&[3_{2}][2_{1}3_{6}]=[2_{1}][3_{2}3_{6}]+[2_{1}2_{3}2_{5}],\\
&[3_{2}3_{6}][2_{1}3_{6}3_{10}]=[2_{1}3_{6}][3_{2}3_{6}3_{10}]+[2_{1}2_{3}2_{5}2_{7}2_{9}],\\
&[1_{0}3_{6}][3_{2}]=[1_{0}][3_{2}3_{6}]+[1_{0}2_{3}2_{5}],\\
&[3_{2}3_{6}][1_{0}3_{6}3_{10}]=[1_{0}3_{6}][3_{2}3_{6}3_{10}]+[1_{0}2_{3}2_{5}2_{7}2_{9}],\\
&[3_{4}][1_{0}2_{3}3_{8}]=[1_{0}2_{3}][3_{4}3_{8}]+[1_{0}2_{3}2_{5}2_{7}],\\
&[3_{4}3_{8}][1_{0}2_{3}3_{8}3_{12}]=[1_{0}2_{3}3_{8}][3_{4}3_{8}3_{12}]+[1_{0}2_{3}2_{5}2_{7}2_{9}2_{11}],\\
&[4_{3}][1_{0}4_{7}]=[1_{0}][4_{3}4_{7}]+[1_{0}3_{4}3_{6}],\\
&[4_{5}][1_{0}2_{3}4_{9}]=[1_{0}2_{3}][4_{5}4_{9}]+[1_{0}2_{3}3_{6}3_{8}],\\
&[1_{0}4_{7}][1_{0}1_{2}2_{5}4_{11}]=[1_{0}1_{2}2_{5}][1_{0}4_{7}4_{11}]+[1^2_{0}1_{2}2_{5}3_{8}3_{10}].
\end{align*}
\end{example}

\subsection{A system corresponding to the system in Theorem \ref{dual-M-system}}
By replacing each $[\mathcal{\widetilde{T}}_{k_1,\ldots,k_n}^{(s)}]$ (resp. $[\mathcal{T}_{k_1,\ldots,k_n}^{(s)}]$) in the system of Theorem \ref{dual-M-system} with  $\chi(\widetilde{\mathfrak{m}}_{k_1,\ldots,k_n})$ (resp. $\chi(\mathfrak{m}_{k_1,\ldots,k_n})$), we obtain a system of equations consisting of the characters of $U_{q}\mathfrak{g}$-modules. The following are two equations in the system.
\begin{equation*}
\begin{split}
&\chi(\widetilde{\mathfrak m}_{k_{1},k_{2}-1,k_{3},\ldots,k_{n-1},0})\chi(\widetilde{\mathfrak m}_{k_{1},k_{2},k_{3},\ldots,k_{n-1},0})=\chi(\widetilde{\mathfrak m}_{k_{1}-1,k_{2},k_{3},\ldots,k_{n-1},0})\chi(\widetilde{\mathfrak m}_{k_{1}+1,k_{2}-1,k_{3},\ldots,k_{n-1},0})+\chi(\widetilde{\mathfrak m}_{0,k_{2}-1,k_{3},\ldots,k_{n-1},0})\chi(\widetilde{\mathfrak m}_{0,k_{1}+k_{2},k_{3},\ldots,k_{n-1},0}),\\
&\chi(\mathfrak m_{k_1,\ldots,k_m-1,0,\ldots,0,k_n})\chi(\mathfrak m_{k_1,\ldots,k_m,0,\ldots,0,1,k_n})=\chi(\mathfrak m_{k_1,\ldots,k_m,0,\ldots,0,1,k_n-1})\chi(\mathfrak m_{k_1,\ldots,k_m-1,0,\ldots,0,k_n+1})
+\chi(\mathfrak m_{k_1,\ldots,k_m-1,0,\ldots,0,0})\chi(\mathfrak m_{k_1,\ldots,k_m,0,\ldots,0,2k_n+1,0}).
\end{split}
\end{equation*}

\subsection{Relation between the systems in Theorem \ref{dual-M-system} and cluster algebras}
Let $I_1=\{1,2,\ldots,n-1 \}$ and
\begin{align*}
&S_{1}=\{2i+1\mid i\in \mathbb{Z}_{\geq 0}\},\quad S_{2}=\{2i\mid i\in \mathbb{Z}_{\geq 0}\},\\
&S_{3}=\{4i+1\mid i\in \mathbb{Z}_{\geq 0}\},\quad S_{4}=\{4i+3\mid i\in \mathbb{Z}_{\geq 0}\},\\
&S_{5}=\{4i\mid i\in \mathbb{Z}_{\geq 0}\},\ \ \ \ \ \quad S_{6}=\{4i+2\mid i\in \mathbb{Z}_{\geq 0}\}.
\end{align*}
Let
\begin{gather}
\begin{align*}
V=((S_1\cap I_1)\times S_2)\cup((S_2\cap I_1)\times S_1)\cup((\{n\}\cap S_2)\times S_3)\cup((\{n\}\cap S_2)\times S_4)\cup((\{n\}\cap S_1)\times S_5)\cup((\{n\}\cap S_1)\times S_6).
\end{align*}
\end{gather}
We define $\widetilde{Q}$ with vertex set $V$ as follows. The arrows of $\widetilde{Q}$ from the vertex $(i,r)$ to the vertex $(j,s)$ if and only if $b_{ij}\neq 0$ and $s=r-b_{ij}+d_{i}-d_{j}$.

Let ${\bf \widetilde{t}}={\bf \widetilde{t}}_{1}\cup {\bf \widetilde{t}}_{2}$, where
\begin{align}\label{dt1}
{\bf \widetilde{t}}_{1} =\{\widetilde{t}^{(-n+i-[i]_{\text{mod} 2}+1)}_{0,\ldots,0,k_i,0,\ldots,0}\mid i=1,\ldots,n-1, k_1,\ldots,k_{n-1}\in \mathbb{Z}_{\geq0}\},
\end{align}

\begin{align}\label{dt2}
{\bf \widetilde{t}}_{2} =
\begin{cases}
     \{t^{(-4k_{n}+4)}_{0,\ldots, 0, \ldots,0, k_n}, t^{(-4k_{n}+2)}_{0,\ldots, 0, \ldots,0, k_n}\mid \ n \text{ is } \text{odd},\  k_{n}\in \mathbb{Z}_{\geq1} \}, \\
    \\
     \{t^{(-4k_{n}+3)}_{0,\ldots, 0, \ldots,0, k_n}, t^{(-4k_{n}+1)}_{0,\ldots, 0, \ldots,0, k_n}\mid \ n \text { is } \text{even},\ k_{n}\in \mathbb{Z}_{\geq1} \}, \\
     \\
     \{t^{(-2k_i-n+i+[i]_{\text{mod} 2}+1)}_{0,\ldots, 0,k_i,0, \ldots,0}\mid i=1,\ldots,n-1,\ k_1,\ldots,k_{n-1}\in \mathbb{Z}_{\geq0}\}.
     \end{cases}
\end{align}

Let $\widetilde{\mathscr{A}}$ be the cluster algebra defined by the initial seed $({\bf \widetilde{t}}, \widetilde{Q})$.
By similar arguments in Section \ref{definition of cluster algebra A}, we have the following theorem.
\begin{theorem}\label{minimal affinizations correspond to cluster variablesII}
Every equation in the system in Theorem \ref{dual-M-system} corresponds to a mutation equation of the cluster algebra $\widetilde{\mathscr{A}}$. Every minimal affinization in the system in Theorem \ref{dual-M-system} corresponds to a cluster variable of the cluster algebra $\widetilde{\mathscr{A}}$. Therefore the Hernandez-Leclerc conjecture (Conjecture 13.2 in \cite{HL10} and Conjecture 9.1 in \cite{Le10}) is true for the minimal affinizations Theorem \ref{dual-M-system}.
\end{theorem}

\section{Proof of theorem \ref{The special modules of $C_n$}}\label{proof of special}

In this section, we prove Theorem \ref{The special modules of $C_n$}. Without loss of generality, we suppose that $s=0$ in $\widetilde{\mathcal T}^{(s)}$, where $\widetilde{\mathcal T}$ is a module in Theorem \ref{The special modules of $C_n$}.

\subsection{The cases of $ \widetilde{\mathcal T}^{(0)}_{0,\ldots,0,k_{n-i},0,\ldots,0,k_n}\ (1\leq i\leq n-1)$}
\ \\
\textbf{Case 1}. $i=1$. Let $m_+=\widetilde{T}^{(0)}_{0,\ldots,0, k_{n-1},k_n}$. Then

\begin{align*}
m_+=n_0n_4\ldots n_{4k_n-4}(n-1)_{4k_n+1}(n-1)_{4k_n+3}\ldots (n-1)_{4k_n+2k_{n-1}-1}.
\end{align*}
Let
\begin{align*}
U=I\times \{aq^s : s\in \mathbb{Z}, s\leq 4k_n+2k_{n-1}-1\}.
\end{align*}
Since all monomials in $\mathscr{M}(\chi_q(m_+)-$trunc$_{m_+\mathcal{Q}^{-}_{U}}\chi_q(m_+))$ are right-negative, it is sufficient to show that trunc$_{m_+\mathcal{Q}^{-}_{U}}\chi_q(m_+)$ is special. Let
\begin{align*}
\mathscr{M}=\{m_+\prod^{s-1}_{j=0}A^{-1}_{n, 4k_n-4j-2}: 0\leq s\leq k_n-1\}.
\end{align*}
It is easy to see that $\mathscr{M}$ satisfies the conditions in Theorem \ref{truncated $q$-characters}. Therefore
\vskip 0.2in
\begin{center}
trunc$_{m_+\mathcal{Q}^{-}_{U}}\chi_q(m_+)=\sum_{m\in \mathscr M}m$
\end{center}
\vskip 0.2in
and hence trunc$_{m_+\mathcal{Q}^{-}_{U}}\chi_q(m_+)$ is special.

\textbf{Case 2}. $i=2$. Let $m_+=\widetilde{T}^{(0)}_{0,\ldots,0,, k_{n-2},0,k_n}$. Then
\begin{align*}
m_+=n_0n_4\ldots n_{4k_n-4}(n-2)_{4k_n+2}(n-2)_{4k_n+4}\ldots(n-2)_{4k_n+2k_{n-2}}.
\end{align*}

\textbf{Case 2.1}. $k_n=1$. Let $g_+=\widetilde{T}^{(0)}_{0,\ldots,0,k_{n-2},0,1}$. Then
\begin{align*}
g_+=n_0(n-2)_6(n-2)_8\ldots(n-2)_{2k_{n-2}+4}.
\end{align*}
Let $U=I\times\{aq^s: s\in \mathbb{Z}, s\leq 2m+4\}$. The monomials in $\mathscr{M}(\chi_q(g_+))$ are
\begin{align*}
g_0=g_+,\ g_1=g_0A^{-1}_{n,2},\ g_2=g_1A^{-1}_{n-1, 4},\ g_3=g_2A^{-1}_{n-1,2},\ g_4=g_3A^{-1}_{n,4}.
\end{align*}
Therefore, the only dominant monomial in $\chi_q(g_+)$ is $g_+$.

\textbf{Case 2.2}. $k_n>1$. We write $m_+$ into two different monomials products, that is, $m_+=m'_1m'_2=m''_1m''_2$ for some monomials $m'_1, m'_2, m''_1, m''_2$. Since each monomial in the product is special, we can use the Frenkel-Mukhin algorithm to compute the $q$-characters of the monomial. By Lemma \ref{contains in a larger set}, we have $\mathscr{M}(\chi_q(m_+))\subset \mathscr{M}(\chi_q(m'_1)\chi_q(m'_2))\cap \mathscr{M}(\chi_q(m''_1)\chi_q(m''_2))$ and we will show that the only dominant monomial in $\chi_q(m_+)$ is $m_+$. Hence $\chi_q(m_+)$ is special.

Let $m_+=m'_1m'_2$, where
\begin{align*}
&m^{'}_{1}=n_0n_4\ldots n_{4k_n-8},\\
&m^{'}_{2}=n_{4k_n-4}(n-2)_{4k_n+2}(n-2)_{4k_n+4}\ldots(n-2)_{4k_n+2k_{n-2}}.
\end{align*}
In case 2.1, we have shown that $m^{'}_{2}$ is special. Therefore the Frenkel-Mukhin algorithm works for $m^{'}_{2}$. We will use the Frenkel-Mukhin algorithm to compute $\chi_{q}(m^{'}_{1})\chi_{q}(m^{'}_{2})$. Let $m=m_1m_2$ be a dominant monomial, where $m_i\in \mathscr M(\chi_q(m^{'}_{i})), i=1, 2$.

Suppose that $m_2\neq m^{'}_{2}$. If $m_2$ is right-negative, then $m$ is a right-negative monomial and therefore $m$ is not dominant. This is a contradiction. Hence $m_2$ is not right-negative. By Case 2.1, $m_2$ is one of the following monomials\\
\begin{gather}
\begin{align*}
\begin{split}
&\overline{m}_{1}=m^{'}_{2}A^{-1}_{n, 4k_n-2}=n^{-1}_{4k_n}(n-1)_{4k_n-3}(n-1)_{4k_n-1}(n-2)_{4k_n+2}(n-2)_{4k_n+4}\ldots(n-2)_{4k_n+2k_{n-2}},\\
&\overline{m}_{2}=\overline{m}_{1}A^{-1}_{n-1, 4k_n}=(n-1)_{4k_n-3}(n-1)^{-1}_{4k_n+1}(n-2)_{4k_n}(n-2)_{4k_n+2}(n-2)_{4k_n+4}\ldots(n-2)_{4k_n+2k_{n-2}},\\
&\overline{m}_{3}=\overline{m}_{2}A^{-1}_{n-1,4k_n-2}=(n-1)^{-1}_{4k_n-1}(n-1)^{-1}_{4k_n+1}n_{4k_n-2}(n-2)_{4k_n-2}(n-2)_{4k_n}(n-2)_{4k_n+2}(n-2)_{4k_n+4}\ldots(n-2)_{4k_n+2k_{n-2}},\\
&\overline{m}_{4}=\overline{m}_{3}A^{-1}_{n, 4k_n}=n^{-1}_{4k_n+2}(n-2)_{4k_n-2}(n-2)_{4k_n}(n-2)_{4k_n+2}(n-2)_{4k_n+4}\ldots(n-2)_{4k_n+2k_{n-2}}.\\
\end{split}
\end{align*}
\end{gather}
We can see that $n^{-1}_{4k_n}$ cannot be canceled by any monomial in $\chi_{q}(m^{'}_{1})$. Therefore $m=m_1m_2 (m_1\in \chi_{q}(m^{'}_{1}))$ is not dominant. This is a contradiction. Hence $m_2\neq \overline{m}_{1}$. Similarly, $m_2$ cannot be $\overline{m}_{i}, i=2, 3, 4$. This is a contradiction. Therefore $m_2=m^{'}_{2}$.

If $m_1\neq m^{'}_1$, then $m_1$ is right-negative. Since $m$ is dominant, each factor with a negative power in $m_1$ needs to be canceled by a factor in $m^{'}_2$. The only factor in $m^{'}_{2}$ which can be canceled is $n_{4k_n-4}$. We have $\mathscr{M}(\chi_q(m^{'}_{1}))\subset \mathscr{M}(\chi_{q}(n_0n_4\ldots n_{4k_n-12}))\chi_q(n_{4k_n-8})$. Only monomials in $\chi_q(n_{4k_n-8})$ can cancel $n_{4k_n-4}$. The only monomial in $\chi_q(n_{4k_n-8})$ which can cancel $n_{4k_n-4}$ is $n^{-1}_{4k_n-4}(n-1)_{4k_n-7}(n-1)_{4k_n-5}$. Therefore $m_1$ is in the set
\begin{align*}
\mathscr{M}(\chi_q(n_0n_4\ldots n_{4k_n-12}))n^{-1}_{4k_n-4}(n-1)_{4k_n-7}(n-1)_{4k_n-5}.
\end{align*}

If $m_1=(n_0n_4\ldots n_{4k_n-12})n^{-1}_{4k_n-4}(n-1)_{4k_n-7}(n-1)_{4k_n-5}$, then
\begin{align}\label{m}
\begin{split}
m=m_1m_2=&n_0n_4\ldots n_{4k_n-12}(n-1)_{4k_n-7}(n-1)_{4k_n-5}\\
  &(n-2)_{4k_n+2}(n-2)_{4k_n+4}\ldots(n-2)_{4k_n+2k_{n-2}}
\end{split}
\end{align}
is dominant. Suppose that
\begin{align*}
m_1\neq (n_0n_4\ldots n_{4k_n-12})n^{-1}_{4k_n-4}(n-1)_{4k_n-7}(n-1)_{4k_n-5}.
\end{align*}
Then $m_1=g_1n^{-1}_{4k_n-4}(n-1)_{4k_n-7}(n-1)_{4k_n-5}$, where $g_1$ is a nonhighest monomial in $\chi_q(n_0n_4\ldots\\ n_{4k_n-12})$. Since $g_1$ is right-negative, $(n-1)_{4k_n-7}$ or $(n-1)_{4k_n-5}$ should cancel a factor of $g_1$ with a negative power. It is easy to see that there exists either a factor $(n-1)^{2}_{4k_n-9}$ or $(n-1)^{2}_{4k_n-7}$ in a monomial in $\chi_q(n_0n_4\ldots n_{4k_n-12})n^{-1}_{4k_n-4}(n-1)_{4k_n-7}(n-1)_{4k_n-5}$ by using the Frenkel-Mukhin algorithm. Therefore we need a factor $(n-1)_{4k_n-9}$ or $(n-1)_{4k_n-7}$ in a monomial in $\chi_q(n_0n_4\ldots n_{4k_n-12})$. We have
\begin{align*}
\begin{split}
\chi_q&(n_0n_4\ldots n_{4k_n-12})n^{-1}_{4k_n-4}(n-1)_{4k_n-7}(n-1)_{4k_n-5}\\
      &\subseteq \chi_q(n_0n_4\ldots n_{4k_n-16})\chi_q(n_{4k_n-12})n^{-1}_{4k_n-4}(n-1)_{4k_n-7}(n-1)_{4k_n-5}.
\end{split}
\end{align*}
The factors $(n-1)_{4k_n-9}$ and $(n-1)_{4k_n-7}$ can only come from the monomials in $\chi_q(n_{4k_n-12})$. But notice that any monomial in $\chi_q(n_{4k_n-12})$ dose not have a factor $(n-1)_{4k_n-9}$. The only monomial in $\chi_q(n_{4k_n-12})$ which contains a factor $(n-1)_{4k_n-7}$ is
\begin{align*}
(n-3)_{4k_n-7}(n-2)_{4k_n-10}(n-2)^{-1}_{4k_n-6}(n-1)_{4k_n-7}n^{-1}_{4k_n-6}.
\end{align*}
Therefore $m_1$ is in the set
\begin{align*}
\begin{split}
\mathscr{M}&(\chi_q(n_0n_4\ldots n_{4k_n-16}))(n-3)_{4k_n-7}(n-2)_{4k_n-10}(n-2)^{-1}_{4k_n-6}(n-1)_{4k_n-7}n^{-1}_{4k_n-6}\\
           &n^{-1}_{4k_n-4}(n-1)_{4k_n-7}(n-1)_{4k_n-5}\\
          =&(\chi_q(n_0n_4\ldots n_{4k_n-16}))(n-3)_{4k_n-7}(n-2)_{4k_n-10}(n-2)^{-1}_{4k_n-6}n^{-1}_{4k_n-6}\\
           &n^{-1}_{4k_n-4}(n-1)^{2}_{4k_n-7}(n-1)_{4k_n-5}.
\end{split}
\end{align*}
Since $m=m_1m_2$ is dominant, $(n-2)^{-1}_{4k_n-6}n^{-1}_{4k_n-6}$ should be canceled by some monomial in $\chi_q(n_0n_4\ldots n_{4k_n-16})$. But notice that $(n-2)^{-1}_{4k_n-6}n^{-1}_{4k_n-6}$ cannot be canceled by any monomial in $\chi_q(n_0n_4\ldots n_{4k_n-16})$. This is a contradiction.

Therefore the only dominant monomials in $\chi_q(m^{'}_1)\chi_q(m^{'}_{2})$ are $m_+$ and \ref{m}.

On the other hand, let $m_+=m''_1m''_2$, where
\begin{align*}
&m^{''}_{1}=n_0n_4\ldots n_{4k_n-4},\\
&m^{''}_{2}=(n-2)_{4k_n+2}(n-2)_{4k_n+4}\ldots(n-2)_{4k_n+2k_{n-2}}.
\end{align*}

The monomial \ref{m} is
\begin{align}\label{g}
g=m_+A^{-1}_{n, 4k_n-6}.
\end{align}
Since $A_{i, a}, i\in I, a\in \mathbb{C}^{\times}$ are algebraically independent, the expression \ref{g} of $g$ of the form $m_+\prod_{i\in I, a\in \mathbb{C}^{\times}}A^{-v_{i, a}}_{i, a}$, where $v_{i, a}$ are some integers, is unique. Suppose that the monomial $g$ is in $\chi_q(m^{''}_1)\chi_q(m^{''}_{2})$. Then $g=g_1g_2$, where $g_i\in \mathscr{M}(\chi_q(m^{''}_{i})), i=1, 2$. By the expression \ref{g}, we have $g_2=m^{''}_{2}$ and $g_1=m^{''}_{1}A^{-1}_{n, 4k_n-6}$. By the Frenkel-Mukhin algorithm, the monomial $ m^{''}_{1}A^{-1}_{n, 4k_n-6}$ is not in $\mathscr{M}(\chi_q(m^{''}_{1}))$. This contradicts the fact that $g_1\in \mathscr{M}(\chi_q(m^{''}_{1}))$. Therefore $g$ is not in $\chi_q(m^{''}_1)\chi_q(m^{''}_{2})$.

\textbf{Case 3}. $i\geq3$. Let $m_+=\widetilde{T}^{(0)}_{0,\ldots,0,k_{n-i},0,\ldots,0,k_n}$. Then
\begin{align*}
m_+=n_0n_4\ldots n_{4k_n-4}(n-i)_{4k_n+i}(n-i)_{4k_n+i+2}\ldots(n-i)_{4k_n+2k_{n-i}+i-2}.
\end{align*}

\textbf{Case 3.1}. $k_n=1$. Let
\begin{align*}
g_+=n_0(n-i)_{i+4}(n-i)_{i+6}\cdots (n-i)_{2k_{n-i}+i+2}.
\end{align*}
Since $\mathscr{M}(\chi_q(n_0(n-i)_{i+4}))\subset \mathscr{M}(\chi_q(n_0)\chi_q((n-i)_{i+4}))$, by using the Frenkel-Mukhin algorithm, we show that the only possible dominant monomials in $\chi_q(n_0)\chi_q((n-i)_{i+4})$ are $n_0(n-i)_{i+4}$ and $(n-1)_1(n-i-1)_{i+3}$.

Let
\begin{align*}
m'=(n-1)_1(n-i-1)_{i+3} ,\quad M=n_0A^{-1}_{n,2}A^{-1}_{n-1,4}A^{-1}_{n-2,5}\cdots A^{-1}_{n-i+1,i+2}.
\end{align*}
By Theorem \ref{Elimination Theorem}, $m'$ is not in $\mathscr{M}(\chi_q(n_0(n-i)_{i+4}))$. Therefore the only dominant monomial in $\mathscr{M}(\chi_q(n_0(n-i)_{i+4}))$ is $n_0(n-i)_{i+4}$.

In the following, let
\begin{align*}
&g''_1=n_0,\\
&g'_1=n_0(n-i)_{i+4},\\
&g'_2=(n-i)_{i+6}\cdots(n-i)_{2k_{n-i}+i+2},\\
&g''_2=(n-i)_{i+4}(n-i)_{i+6}\cdots(n-i)_{2k_{n-i}+i+2}.\\
\end{align*}
Then $\mathscr{M}(\chi_q(g_+))\subset \mathscr{M}(\chi_q(g'_1)\chi_q(g'_2))\cap \mathscr{M}(\chi_q(g''_1)\chi_q(g''_2))$. By using similar arguments as the case of $i=2$, if we expand $(n-i)_{i+4}$ in $g'_1$, suppose that there are dominant monomials $\overline{g}_1,\overline{g}_2,\ldots,$ in $\chi_q(g'_1)\chi_q(g'_2)$, they will not be in $\chi_q(g''_1)\chi_q(g''_2)$ since $(n-i)_{i+4}$ dose not be expanded in $g''_2$. So we only need to expand $n_0$ in $g'_1$. Let
\begin{align*}
g_1=n_0(n-i)_{i+4}A^{-1}_{n,2},\ g_2=g_1A^{-1}_{n-1,4},\ g_3=g_2A^{-1}_{n-2,5},\ \ldots,\ g_i=g_{i-1}A^{-1}_{n-i+1,i+2}.
\end{align*}
At this point, $(n-i)_{i+2}(n-i)_{i+4}$ is in the expression of $g_i$, so $(n-i)_{i+2}$ can not be expanded. Notice that this situation will not occur if we expand $n_0$ in $g''_1$. Hence, except $g^+$, even through there are dominant monomials in $\chi_q(g'_1)\chi_q(g'_2)$, they will not be in $\chi_q(g''_1)\chi_q(g''_2)$. Therefore, the only dominant monomial in $\chi_q(g_+)$ is $g_+$.

\textbf{Case 3.2}. $k_n>1$. Let
\begin{align*}
&h'_1=n_0n_4\cdots n_{4k_n-4},\\
&h'_2=(n-i)_{4k_n+i},\\
&h''_1=n_0n_4\cdots n_{4k_n-8},\\
&h''_2=n_{4k_n-4}(n-i)_{4k_n+i}.
\end{align*}
Then $\mathscr{M}(\chi_q(n_0\cdots n_{4k_n-4}(n-i)_{4k_n+i}))\subset \mathscr{M}(\chi_q(h'_1)\chi_q(h'_2))\cap \mathscr{M}(\chi_q(h''_1)\chi_q(h''_2))$.

By using similar arguments as the case of $k_n=1$, we can see that the only possible dominant monomial in $\mathscr{M}(\chi_q(h'_1)\chi_q(h'_2))\cap \mathscr{M}(\chi_q(h''_1)\chi_q(h''_2))$ is $h'_1h'_2$.

In what follows, let
\begin{align*}
&m'_1=n_0\cdots n_{4k_n-4}(n-i)_{4k_n+i},\\
&m''_1=n_0\cdots n_{4k_n-4},\\
&m'_2=(n-i)_{4k_n+i+2}\cdots(n-i)_{4k_n+2k_{n-i}+i-2},\\
&m''_2=(n-i)_{4k_n+i}(n-i)_{4k_n+i+2}\cdots(n-i)_{4k_n+2k_{n-i}+i-2}.
\end{align*}
Then $\mathscr{M}(\chi_q(m_+))\subset \mathscr{M}(\chi_q(m'_1)\chi_q(m'_2))\cap \mathscr{M}(\chi_q(m''_1)\chi_q(m''_2))$.

By using similar arguments as the case of above, we can see that the only possible dominant monomial in $\mathscr{M}(\chi_q(m'_1)\chi_q(m'_2))\cap \mathscr{M}(\chi_q(m''_1)\chi_q(m''_2))$ is $m_+$. Therefore the only dominant monomial in $\chi_q(m_+)$ is $m_+$.

\subsection{The case of $\widetilde{\mathcal T}^{(0)}_{0,\ldots,0,k_{n-j},0,\ldots,0,k_{n-i},0,\ldots,0,k_n}\ (1\leq i<j\leq n-1,\ 0\leq k_{n-j}\leq i+1)$}
\ \\
\textbf{Case 1}. If $i=1$, then $0\leq k_{n-j}\leq2$. Without loss of generality, suppose that $k_{n-j}=1$. When $k_{n-j}=2$, the proof is similar. If $k_{n-j}=0$, the proof is coincide with the case of $\widetilde{\mathcal T}^{(0)}_{0,\ldots,0,k_{n-i},0,\ldots,0,k_n}\ (1\leq i\leq n-1)$.  Let $m_+=\widetilde{T}^{(0)}_{0,\ldots,0,k_{n-j},0,\ldots,0,k_{n-1},k_n}$. Then
\begin{align*}
m_+=&n_0n_4\ldots n_{4k_n-4}(n-1)_{4k_n+1}(n-1)_{4k_n+3}\ldots \\
    &(n-1)_{4k_n+2k_{n-1}-1}(n-j)_{4k_n+2k_{n-1}+j}.
\end{align*}


Let
\begin{align*}
\begin{split}
&m'_1=n_0n_4\ldots n_{4k_n-4},\\
&m'_2=(n-1)_{4k_n+1}(n-1)_{4k_n+3}\ldots (n-1)_{4k_n+2k_{n-1}-1},\\
&m'_3=(n-j)_{4k_n+2k_{n-1}+j}.
\end{split}
\end{align*}
Then $\mathscr{M}(\chi_q(m_+))\subset \mathscr{M}(\chi_q(m'_1)\chi_q(m'_2m'_3))\cap \mathscr{M}(\chi_q(m'_1m'_2)\chi_q(m'_3))$. Note that we have shown that $m'_1m'_2$ is special in the case of $\widetilde{\mathcal{T}}^{(0)}_{0,\ldots,0,k_{n-1},k_n}$. Therefore the Frenkel-Mukhin algorithm applies to $m'_1m'_2$.

By using similar arguments as the case of $\widetilde{\mathcal{T}}^{(0)}_{0,\ldots,0,k_{n-i},0,\ldots,0,k_n}$, we show that the only possible dominant monomials in $\chi_q(m'_1m'_2)\chi_q(m'_3)$ are $m_+$ and
\begin{align*}
\begin{split}
g_1=&n_0n_4\ldots n_{4k_n-4}n_{4k_n+2k_{n-1}}(n-1)_{4k_n+1}\ldots(n-1)_{4k_n+2k_{n-1}-3}(n-j-1)_{4k_n+2k_{n-1}+j-1}\\
   =&m_+A^{-1}_{n-j, 4k_n+2k_{n-1}+j-1}A^{-1}_{n-j+1, 4k_n+2k_{n-1}+j-2}A^{-1}_{n-j+2, 4k_n+2k_{n-1}+j-3}\ldots A^{-1}_{n-1, 4k_n+2k_{n-1}}.
\end{split}
\end{align*}

Furthermore, $g_1$ is not in $\chi_q(m'_1)\chi_q(m'_2m'_3)$. Therefore the only dominant monomial in $\chi_q(m_+)$ is $m_+$.

\textbf{Case 2}. If $i=2$, then $0\leq k_{n-j}\leq3$. Since the proof of each case of is similar to each other, we give a detailed proof of the case of $k_{n-j}=2$. Let $m_+=\widetilde{T}^{(0)}_{0,\ldots,0,k_{n-j},0,\ldots,0,k_{n-2},0,k_n}$. Then
\begin{gather}
\begin{align*}
m_+=&n_0n_4\ldots n_{4k_n-4}(n-2)_{4k_n+2}(n-2)_{4k_n+4}\ldots (n-2)_{4k_n+2k_{n-2}}\\
&(n-j)_{4k_n+2k_{n-2}+j}(n-j)_{4k_n+2k_{n-2}+j+2}.
\end{align*}
\end{gather}


Moreover, let
\begin{align*}
&m'_1=n_0n_4\ldots n_{4k_n-4}(n-2)_{4k_n+2},\\
&m''_1=n_0n_4\ldots n_{4k_n-4},\\
&m'_2=(n-2)_{4k_n+4}\ldots (n-2)_{4k_n+2k_{n-2}}(n-j)_{4k_n+2k_{n-2}+j}(n-j)_{4k_n+2k_{n-2}+j+2} ,\\
&m''_2=(n-2)_{4k_n+2}(n-2)_{4k_n+4}\ldots(n-2)_{4k_n+2k_{n-2}}(n-j)_{4k_n+2k_{n-2}+j}(n-j)_{4k_n+2k_{n-2}+j+2}.
\end{align*}
Then $\mathscr{M}(\chi_q(m_+))\subset \mathscr{M}(\chi_q(m'_1)\chi_q(m'_2))\cap \mathscr{M}(\chi_q(m''_1)\chi_q(m''_2))$. Since we have shown that $m'_1$ are special in the case of $\widetilde{\mathcal T}^{(0)}_{0,\ldots,0,k_{n-2},0,k_n}$, the Frenkel-Mukhin algorithm applies to $m'_1$.

By using similar arguments as the case of $\widetilde{\mathcal{T}}^{(0)}_{0,\ldots,0,k_{n-i},0,\ldots,0,k_n}$, if we expand $(n-2)_{4k_n+2}$ in $m'_1$, suppose that there are dominant monomials $\overline{m}_1,\overline{m}_2,\ldots$ in $\chi_q(m'_1)\chi_q(m'_2)$, they will not be in $\chi_q(m''_1)\chi_q(m''_2)$ since $(n-2)_{4k_n+2}$ dose not be expanded in $m''_2$. So we only need to expand $n_{4k_n-4}$ in $m'_1$. Let
\begin{align*}
m_1=m'_1A^{-1}_{n,4k_n-2},\quad m_2=m_1A^{-1}_{n-1,4k_n}.
\end{align*}
At this point, $(n-2)_{4k_n}(n-2)_{4k_n+2}$ is in the expression of $m_2$, so $(n-2)_{4k_n}$ can not be expanded. Notice that this situation will not occur if we expand $n_{4k_n-4}$ in $m''_1$. Hence, except $m^+$, even through there are dominant monomials in $\chi_q(m'_1)\chi_q(m'_2)$, they will not be in $\chi_q(m''_1)\chi_q(m''_2)$. Therefore the only dominant monomial in $\chi_q(m_+)$ is $m_+$.

\textbf{Case 3}. Suppose that $2< i<j\leq n-1$. Let $m_+=\widetilde{T}^{(0)}_{0,\ldots,0,k_{n-j},0,\ldots,0,k_{n-i},0,\ldots,0,k_n}$, where $0\leq k_{n-j}\leq i+1$. Then
\begin{align*}
m_+=&n_0\cdots n_{4k_n-4}(n-i)_{4k_n+i}\cdots(n-i)_{4k_n+2k_{n-i}+i-2}\\
&(n-j)_{4k_n+2k_{n-i}+j}\cdots(n-j)_{4k_n+2k_{n-i}+2k_{n-j}+j-2}.
\end{align*}
Let
\begin{gather}
\begin{align*}
m'_1=&n_0n_4\ldots n_{4k_n-4}(n-i)_{4k_n+i},\\
m''_1=&n_0n_4\ldots n_{4k_n-4},\\
m'_2=&(n-i)_{4k_n+i+2}\ldots (n-i)_{4k_n+2k_{n-i}+i-2}(n-j)_{4k_n+2k_{n-i}+j}\cdots(n-j)_{4k_n+2k_{n-i}+2k_{n-j}+j-2},\\
m''_2=&(n-i)_{4k_n+i}(n-i)_{4k_n+i+2}\ldots (n-i)_{4k_n+2k_{n-i}+i-2}(n-j)_{4k_n+2k_{n-i}+j}\cdots(n-j)_{4k_n+2k_{n-i}+2k_{n-j}+j-2}.
\end{align*}
\end{gather}
Then $\mathscr{M}(\chi_q(m_+))\subset \mathscr{M}(\chi_q(m'_1)\chi_q(m'_2))\cap \mathscr{M}(\chi_q(m''_1)\chi_q(m''_2))$. We have shown that $m'_1$ is special in the case $\widetilde{\mathcal T}^{(0)}_{0,\ldots,0,k_{n-i},0,\ldots,0,k_n}$, the Frenkel-Mukhin algorithm applies to $m'_1$.

By using similar arguments as the case of $i=2$, we show that the only possible dominant monomials in $\mathscr{M}(\chi_q(m'_1)\chi_q(m'_2))\cap \mathscr{M}(\chi_q(m''_1)\chi_q(m''_2))$ is $m_+$.
Therefore the only dominant monomial in $\chi_q(m_+)$ is $m_+$.

\subsection{ The case of $\widetilde{\mathcal T}^{(0)}_{k_1,\ldots,k_m,0,\ldots,0,k_l,0,\ldots,0,k_n}\ (1\leq m<l\leq n-1,\ 0\leq k_1+\cdot\cdot\cdot+k_m\leq n-l+1)$}
\ \\
\textbf{Case 1}. If $l=n-1$, then $0\leq k_1+\cdots+k_m\leq2$. In this case, it is sufficient to show the module $\widetilde{\mathcal T}^{(0)}_{0,\ldots,0,k_{n-j},0,\ldots,0,k_{n-i},0,\ldots,0,k_{n-1},k_n}$ ($0\leq k_{n-i}+k_{n-j}\leq2$) is special. Note that if $k_{n-i}=0$ (resp. $k_{n-j}=0$), then $0\leq k_{n-j}\leq2$ (resp. $0\leq k_{n-i}\leq2$), the case is coincide with the case of $\widetilde{\mathcal T}^{(0)}_{0,\ldots,0,k_{n-i},0,\ldots,0,k_{n-1},k_n}\ (0\leq k_{n-i}\leq2)$. Therefore, let $k_{n-i}=k_{n-j}=1$ and $m_+=\widetilde{T}^{(0)}_{0,\ldots,0,\underset{n-j}1,0,\ldots,0,\underset{n-i}1,0,\ldots,0,k_{n-1},k_n}$. Then
\begin{gather}
\begin{align*}
m_+=&n_0n_4\ldots n_{4k_n-4}(n-1)_{4k_n+1}(n-1)_{4k_n+3}\ldots (n-1)_{4k_n+2k_{n-1}-1}\\
&(n-i)_{4k_n+2k_{n-1}+i}(n-j)_{4k_n+2k_{n-1}+j+2}.
\end{align*}
\end{gather}


Let
\begin{align*}
\begin{split}
&m'_1=n_0n_4\ldots n_{4k_n-4},\quad \\
&m'_2=(n-1)_{4k_n+1}(n-1)_{4k_n+3}\ldots (n-1)_{4k_n+2k_{n-1}-1},\\
&m'_3=(n-i)_{4k_n+2k_{n-1}+i},\\
&m'_4=(n-j)_{4k_n+2k_{n-1}+j+2}.
\end{split}
\end{align*}
Then $\mathscr{M}(\chi_q(m_+))\subset \mathscr{M}(\chi_q(m'_1m'_2m'_3)\chi_q(m'_4))\cap \mathscr{M}(\chi_q(m'_1m'_2)\chi_q(m'_3m'_4))$. Note that we have shown that $m'_1m'_2$ is special in the case $\widetilde{\mathcal{T}}^{(0)}_{0,\ldots,0,k_{n-1},k_n}$. Therefore the Frenkel-Mukhin algorithm applies to $m'_1m'_2$.

By using similar arguments as the case of $\widetilde{\mathcal{T}}^{(0)}_{0,\ldots,0,k_{n-j},0,\ldots,0,k_{n-1},k_n}$, we show that the only possible dominant monomials in $\chi_q(m'_1m'_2)\chi_q(m'_3m'_4))$ are $m_+$ and
\begin{gather}
\begin{align*}
\begin{split}
g_1=&n_0n_4\ldots n_{4k_n-4}n_{4k_n+2k_{n-1}}(n-1)_{4k_n+1}\ldots(n-1)_{4k_n+2k_{n-1}-3}\\
    &(n-i-1)_{4k_n+2k_{n-1}+i-1}(n-j)_{4k_n+2k_{n-1}+j+2}\\
   =&m_+A^{-1}_{n-i, 4k_n+2k_{n-1}+i-1}A^{-1}_{n-i+1, 4k_n+2k_{n-1}+i-2}A^{-1}_{n-i+2, 4k_n+2k_{n-1}+i-3}\ldots A^{-1}_{n-1, 4k_n+2k_{n-1}}.
\end{split}
\end{align*}
\end{gather}

Furthermore, $g_1$ is not in $\chi_q(m'_1m'_2m'_3)\chi_q(m'_4)$. Therefore the only dominant monomial in $\chi_q(m_+)$ is $m_+$.

\textbf{Case 2}. If $l=n-2$, then $0\leq k_1+\cdots+k_m\leq3$. When $k_1+\cdots+k_m=2$, by using the similar arguments as the case of $\widetilde{\mathcal T}^{(0)}_{0,\ldots,0,k_{n-j},0,\ldots,0,k_{n-i},0,\ldots,0,k_{n-1},k_n}$, we can show that the module $\widetilde{\mathcal T}^{(0)}_{0,\ldots,0,k_{n-j},0,\ldots,0,k_{n-i},0,\ldots,0,k_{n-2},0,k_n}$ is special. Now, we consider the case of \begin{align*}
\widetilde{\mathcal T}^{(0)}_{0,\ldots,0,k_{n-t},0,\ldots,0,k_{n-j},0,\ldots,0,k_{n-i},0,\ldots,0,k_{n-2},0,k_n}\  (k_{n-t}=k_{n-j}=k_{n-i}=1).
\end{align*}

Let
\begin{gather}
\begin{align*}
m_+=&n_0\ldots n_{4k_n-4}(n-2)_{4k_n+2}\ldots(n-2)_{4k_n+2k_{n-2}}(n-i)_{4k_n+2k_{n-2}+i}\\
&(n-j)_{4k_n+2k_{n-2}+j+2}(n-t)_{4k_n+2k_{n-2}+t+4}.
\end{align*}
\end{gather}


Moreover, let
\begin{align*}
&m'_1=n_0\ldots n_{4k_n-4}(n-2)_{4k_n+2}\ldots(n-2)_{4k_n+2k_{n-2}}, \\ &m'_2=(n-i)_{4k_n+2k_{n-2}+i}(n-j)_{4k_n+2k_{n-2}+j+2}(n-t)_{4k_n+2k_{n-2}+t+4},\\
&m''_1=n_0\ldots n_{4k_n-4}(n-2)_{4k_n+2}\ldots(n-2)_{4k_n+2k_{n-2}}(n-i)_{4k_n+2k_{n-2}+i},\\ &m''_2=(n-j)_{4k_n+2k_{n-2}+j+2}(n-t)_{4k_n+2k_{n-2}+t+4}.
\end{align*}
Then $\mathscr{M}(\chi_q(m_+))\subset \mathscr{M}(\chi_q(m'_1)\chi_q(m'_2))\cap \mathscr{M}(\chi_q(m''_1)\chi_q(m''_2))$. Note that we have shown that $m''_1$ is special in the case $\widetilde{\mathcal{T}}^{(0)}_{0,\ldots,0,k_{n-j},0,\ldots,0,k_{n-2},0,k_n}$. Therefore the Frenkel-Mukhin algorithm applies to $m''_1$. By using similar arguments as the case of $\widetilde{\mathcal T}^{(0)}_{0,\ldots,0,k_{n-j},0,\ldots,0,k_{n-i},0,\ldots,0,k_n}$, we show that the only possible dominant monomials in $\mathscr{M}(\chi_q(m'_1)\chi_q(m'_2))\cap \mathscr{M}(\chi_q(m''_1)\chi_q(m''_2))$ is $m_+$. Therefore the only dominant monomial in $\chi_q(m_+)$ is $m_+$.

Note that when $l=2$, the case $\widetilde{\mathcal T}^{(0)}_{k_1,k_2,0,\ldots,0,k_n}\ (0\leq k_1\leq n-1)$ is coincide with the case of $\widetilde{\mathcal T}^{(0)}_{0,\ldots,0,k_{n-j},0,\ldots,0,k_{n-i},0,\ldots,0,k_n}\ (0\leq k_{n-j}\leq i+1)$; when $l=1$, the case $\widetilde{\mathcal T}^{(0)}_{k_1,0,\ldots,0,k_n}$ is coincide with the case of $ \widetilde{\mathcal T}^{(0)}_{0,\ldots,0,k_{n-i},0,\ldots,0,k_n}\ (1\leq i\leq n-1)$. Therefore, the remaining case is $l\in \{n-3, n-4,\ldots,3\}$.  Since the proof of each case is similar to the case of $l=n-2$, by using similar arguments as the above, we deduce that the module \begin{align*}
\widetilde{\mathcal T}^{(0)}_{k_1,\ldots,k_m,0,\ldots,0,k_l,0,\ldots,0,k_n}\ (1\leq m<l\leq n-1,\ 0\leq k_1+\cdot\cdot\cdot+k_m\leq n-l+1),
\end{align*}
is special.

\subsection{ The case of $\widetilde{\mathcal{S}}^{(0)}_{k_1,\ldots,k_m,0,\ldots,0,k_l,0,\ldots,0,k_{n-1},0}\ (1\leq m<l\leq n-2,\ 0\leq k_1+\cdot\cdot\cdot+k_m\leq n-l+1)$}
Since the proof of each case is similar to each other, we give a detailed proof of the case of $l=n-2$ and $k_l=1$. Let $m_+=\widetilde{S}^{(0)}_{0,\ldots,0,k_i,0,\ldots,0,k_j,0,\ldots,0,k_t,0,\ldots,0,\underset{l}1,k_{n-1},0}$, where $k_i=k_j=k_t=1$. Then
\begin{align*}
m_+=&\widetilde{T}^{(0)}_{0,\ldots,0,k_i,0,\ldots,0,k_j,0,\ldots,0,k_t,0,\ldots,0,1,k_{n-1},0}\widetilde{T}^{(2k_{n-1}+4)}_{0,\ldots,0,k_i,0,\ldots,0,k_j,0,\ldots,\underset{k_t-1}0,\ldots,0,0,0}\\
   =&(n-1)_1(n-1)_3\ldots(n-1)_{2k_{n-1}-1}(n-2)_{2k_{n-1}+2}\\
    &t_{2k_{n-1}+n-t+2}j^2_{2k_{n-1}+n-j+4}i^2_{2k_{n-1}+n-i+6}.
\end{align*}
Let
\begin{align*}
&m'_1=(n-1)_1(n-1)_3\ldots(n-1)_{2k_{n-1}-1}(n-2)_{2k_{n-1}+2}t_{2k_{n-1}+n-t+2},\\ &m'_2=j^2_{2k_{n-1}+n-j+4}i^2_{2k_{n-1}+n-i+6},\\
&m''_1=(n-1)_1(n-1)_3\ldots(n-1)_{2k_{n-1}-1}(n-2)_{2k_{n-1}+2},\\ &m''_2=t_{2k_{n-1}+n-t+2}j^2_{2k_{n-1}+n-j+4}i^2_{2k_{n-1}+n-i+6}.
\end{align*}
Then $\mathscr{M}(\chi_q(m_+))\subset \mathscr{M}(\chi_q(m'_1)\chi_q(m'_2))\cap \mathscr{M}(\chi_q(m''_1)\chi_q(m''_2))$. It is easy to see that $m'_2$ and $m''_2$ are special. Therefore the Frenkel-Mukhin algorithm applies to $m'_2, m''_2$.

By using similar arguments as the case of $\widetilde{\mathcal T}^{(0)}_{k_1,\ldots,k_m,0,\ldots,0,k_l,0,\ldots,0,k_n}$, we show that the only possible dominant monomials in $\mathscr{M}(\chi_q(m'_1)\chi_q(m'_2))\cap \mathscr{M}(\chi_q(m''_1)\chi_q(m''_2))$ is $m_+$. Therefore the only dominant monomial in $\chi_q(m_+)$ is $m_+$.

\section{Proof of Theorem $\ref{M system of type Cn}$} \label{proof main1}

\subsection{Classification of dominant monomials in the summands on both sides of the system}In Section \ref{main results}, we have shown that for $s\in\mathbb{Z}, k_1,\ldots,k_{n-1}\in\mathbb{Z}_{\geq0},\ k_n \in\mathbb{Z}_{\geq1}$, the modules
in Theorem \ref{The special modules of $C_n$} are special.
Now we use the Frenkel-Mukhin algorithm to classify dominant monomials in the summands on both sides of the system in Theorem $\ref{M system of type Cn}$.
\begin{lemma}\label{lemma1}
The dominant monomials in each summand on the left and right hand sides of every equation in the system in Theorem \ref{M system of type Cn} are given in Table \ref{dominant monomials in the M-system of type C_n}.
\end{lemma}
We will prove Lemma \ref{lemma1} in Section \ref{proof classification of dominant monomials}.

In Table \ref{dominant monomials in the M-system of type C_n}, $M\prod_{0\leq j\leq r}A^{-1}_{i,s} = M$ for $r = -1$, $s \in \mathbb{Z}$.

\begin{table}[H] \resizebox{.8\width}{.8\height}{
\begin{tabular}{|c|c|c|c|}
\hline %
Equations & Summands in the equations & $M$ & Dominant monomials \\
\hline %
(\ref{eqn1}) & $\substack{\chi_{q}(\mathcal T^{(s)}_{k_{1},k_{2}-1,k_{3},\ldots,k_{n-1},0}) \times \\
\times \chi_{q}(\mathcal T^{(s)}_{k_{1},k_{2},k_{3},\ldots,k_{n-1},0})}$ & $\substack{M = T^{(s)}_{k_{1},k_{2}-1,k_{3},\ldots,k_{n-1},0}\times  \\ \quad \quad \times T^{(s)}_{k_{1},k_{2},k_{3},\ldots,k_{n-1},0})}$ & $\substack{ M\prod_{0\leq j\leq r}A^{-1}_{1,-s-2(\sum^{n-1}_{i=2}k_i)-n-2j+2},\\ -1\leq r\leq k_{1}-1}$ \\
\hline %
(\ref{eqn1}) & $\substack{\chi_{q}(\mathcal T^{(s)}_{k_{1}-1,k_{2},k_{3},\ldots,k_{n-1},0})\times\\
\times \chi_{q}(\mathcal T^{(s-2)}_{k_{1}+1,k_{2}-1,k_{3},\ldots,k_{n-1},0})}$ & $\substack{M= T^{(s)}_{k_{1}-1,k_{2},k_{3},\ldots,k_{n-1},0}\times\\ \quad \quad
\times  T^{(s-2)}_{k_{1}+1,k_{2}-1,k_{3},\ldots,k_{n-1},0}}$ & $\substack{ M\prod_{0\leq j\leq r}A^{-1}_{1,-s-2(\sum^{n-1}_{i=2}k_i)-n-2j+2}, \\ -1\leq r\leq k_{1}-2}$ \\
\hline %
(\ref{eqn1}) & $\substack{\chi_{q}(\mathcal T^{(s)}_{0,k_{1}+k_{2},k_{3},\ldots,k_{n-1},0})\times\\
\times \chi_{q}(\mathcal T^{(s)}_{0,k_{2}-1,k_{3},\ldots,k_{n-1}},0)}$ & $ \substack{M= T^{(s)}_{0,k_{1}+k_{2},k_{3},\ldots,k_{n-1},0}\times\\ \quad \quad \times T^{(s)}_{0,k_{2}-1,k_{3},\ldots,k_{n-1},0}}$  & $M$\\
\hline %
(\ref{eqn2}) & $\substack{\chi_{q}(\mathcal T^{(s)}_{0,\ldots,0,k_{i},k_{i+1}-1,k_{i+2},\ldots,k_{n-1},0} )\times\\
\times\chi_{q}(\mathcal T^{(s)}_{0,\ldots, 0,k_{i},k_{i+1},k_{i+2},\ldots,k_{n-1},0})}$  &$\substack{ M= T^{(s)}_{0,\ldots, 0,k_{i},k_{i+1},k_{i+2},\ldots,k_{n-1},0}\times\\ \quad \quad
\times T^{(s)}_{0,\ldots,0,k_{i},k_{i+1}-1,k_{i+2},\ldots,k_{n-1},0}}$  & $\substack{M\prod_{0\leq j\leq r}A^{-1}_{i,-s-2(\sum^{n-1}_{p=i+1}k_p)-n-2j+i+1}, \\ -1\leq r\leq k_{i}-1}$ \\
\hline %
(\ref{eqn2}) & $\substack{\chi_{q}(\mathcal T^{(s)}_{0,\ldots, 0,k_{i}+1,k_{i+1}-1,k_{i+2},\ldots,k_{n-1},0} )\times\\
\times\chi_{q}(\mathcal T^{(s)}_{0,\ldots, 0,k_{i}-1,k_{i+1},k_{i+2}\ldots,k_{n-1},0})}$  &$\substack{ M= T^{(s)}_{0,\ldots, 0,k_{i}+1,k_{i+1}-1,k_{i+2},\ldots,k_{n-1},0}\times\\ \quad \quad
\times T^{(s)}_{0,\ldots, 0,k_{i}-1,k_{i+1},k_{i+2}\ldots,k_{n-1},0}}$  & $\substack{M\prod_{0\leq j\leq r}A^{-1}_{i,-s-2(\sum^{n-1}_{p=i+1}k_p)-n-2j+i+1}, \\ -1\leq r\leq k_{i}-2}$ \\
\hline %
(\ref{eqn2}) & $\substack{\chi_{q}(\mathcal T^{(s)}_{0,\ldots,0,\underset{i+1}{k_{i}+k_{i+1}},k_{i+2},\ldots,k_{n-1},0})\times\\
\times\chi_{q}(\mathcal T^{(s)}_{0,\ldots,0,\underset{i-1}{k_{i}},0,\underset{i+1}{k_{i+1}-1},k_{i+2},,\ldots,k_{n-1},0})}$  &$\substack{M= T^{(s)}_{0,\ldots,0,\underset{i+1}{k_{i}+k_{i+1}},k_{i+2},\ldots,k_{n-1},0}\times\\ \quad \quad
\times T^{(s)}_{0,\ldots,0,\underset{i-1}{k_{i}},0,\underset{i+1}{k_{i+1}-1},k_{i+2},,\ldots,k_{n-1},0}} $  &$M$ \\
\hline %
(\ref{eqn3}) & $\substack{\chi_{q}(\mathcal T^{(s)}_{k_{1},0,\ldots,0,k_{j}-1,k_{j+1},\ldots,k_{n-1},0})\times\\
\times\chi_{q}(\mathcal T^{(s)}_{k_{1},0,\ldots,0,k_{j},k_{j+1},\ldots,k_{n-1},0})}$  & $\substack{M= T^{(s)}_{k_{1},0,\ldots,0,k_{j}-1,k_{j+1},\ldots,k_{n-1},0}\times\\ \quad \quad\times T^{(s)}_{k_{1},0,\ldots,0,k_{j},k_{j+1},\ldots,k_{n-1},0}}$ & $\substack{M\prod_{0\leq p\leq r}A^{-1}_{1,-s-2(\sum^{n-1}_{i=j}k_i)-n-2p+2}, \\-1\leq r\leq k_{1}-1}$ \\
\hline %
(\ref{eqn3}) & $\substack{\chi_{q}(\mathcal T^{(s)}_{k_{1}+1,0,\ldots,0,k_{j}-1,k_{j+1},\ldots,k_{n-1},0})\times\\
\times\chi_{q}(\mathcal T^{(s)}_{k_{1}-1,0,\ldots, 0,k_{j},k_{j+1}\ldots,k_{n-1},0})}$  &  $\substack{M= T^{(s)}_{k_{1}+1,0,\ldots,0,k_{j}-1,k_{j+1},\ldots,k_{n-1},0}\times\\ \quad \quad
\times T^{(s)}_{k_{1}-1,0,\ldots, 0,k_{j},k_{j+1}\ldots,k_{n-1},0}}$   & $\substack{ M\prod_{0\leq p\leq r}A^{-1}_{1,-s-2(\sum^{n-1}_{i=j}k_i)-n-2p+2}, \\-1\leq r\leq k_{1}-2}$ \\
\hline %
(\ref{eqn3}) & $\substack{\chi_{q}(\mathcal T^{(s)}_{0,k_{1},0,\ldots,0,k_{j},k_{j+1},\ldots,k_{n-1},0})\times\\
\times\chi_{q}(\mathcal T^{(s)}_{0,\ldots,0,k_{j}-1,k_{j+1},\ldots,k_{n-1},0})}$  & $\substack{M= T^{(s)}_{0,k_{1},0,\ldots,0,k_{j},k_{j+1},\ldots,k_{n-1},0}\times\\ \quad \quad
\times T^{(s)}_{0,\ldots,0,k_{j}-1,k_{j+1},\ldots,k_{n-1},0}}$  & $M$\\
\hline %
(\ref{eqn4}) & $\substack{\chi_{q}(\mathcal T^{(s)}_{0,\ldots,0,\underset{i}{k_{i}},0,\ldots,0,\underset{j}{k_{j}-1},k_{j+1},\ldots,k_{n-1},0})\times\\
\times \chi_{q}(\mathcal T^{(s)}_{0,\ldots, 0,\underset{i}{k_{i}},0,\ldots,0,\underset{j}{k_{j}},k_{j+1},\ldots,k_{n-1},0})}$& $\substack{M=T^{(s)}_{0,\ldots,0,\underset{i}{k_{i}},0,\ldots,0,\underset{j}{k_{j}-1},k_{j+1},\ldots,k_{n-1},0}\times\\ \quad \quad \times T^{(s)}_{0,\ldots, 0,\underset{i}{k_{i}},0,\ldots,0,\underset{j}{k_{j}},k_{j+1},\ldots,k_{n-1},0}}$ & $ \substack{M\prod_{0\leq p\leq r}A^{-1}_{i,-s-2(\sum^{n-1}_{q=j}k_q)-n-2p+i+1},\\ -1\leq r\leq k_{i}-1}$ \\
\hline %
(\ref{eqn4}) & $\substack{\chi_{q}(\mathcal T^{(s)}_{0,\ldots,0,\underset{i}{k_{i}+1},0,\ldots,0,\underset{j}{k_{j}-1},k_{j+1}\ldots,k_{n-1},0})\times\\
\times \chi_{q}(\mathcal T^{(s)}_{0,\ldots, 0,\underset{i}{k_{i}-1},0,\ldots,0,\underset{j}{k_{j}},k_{j+1},\ldots,k_{n-1},0})}$ & $\substack{M= T^{(s)}_{0,\ldots,0,\underset{i}{k_{i}+1},0,\ldots,0,\underset{j}{k_{j}-1},k_{j+1}\ldots,k_{n-1},0}\times\\ \quad \quad
\times  T^{(s)}_{0,\ldots, 0,\underset{i}{k_{i}-1},0,\ldots,0,\underset{j}{k_{j}},k_{j+1},\ldots,k_{n-1},0}}$  & $\substack{M\prod_{0\leq p\leq r}A^{-1}_{i,-s-2(\sum^{n-1}_{q=j}k_q)-n-2p+i+1}, \\-1\leq r\leq k_{i}-2}$\\
\hline %
(\ref{eqn4}) & $\substack{\chi_{q}(\mathcal T^{(s)}_{0,\ldots,0,\underset{i+1}{k_{i}},0, \ldots, 0,\underset{j}{k_{j}},k_{j+1},\ldots,k_{n-1},0})\times\\
\times\chi_{q}(\mathcal T^{(s)}_{0,\ldots,0,\underset{i-1}{k_{i}},0,\ldots, 0,\underset{j}{k_{j}-1},k_{j+1},\ldots,k_{n-1},0})}$  &$\substack{M= T^{(s)}_{0,\ldots,0,\underset{i+1}{k_{i}},0, \ldots, 0,\underset{j}{k_{j}},k_{j+1},\ldots,k_{n-1},0}\times\\ \quad \quad \times T^{(s)}_{0,\ldots,0,\underset{i-1}{k_{i}},0,\ldots, 0,\underset{j}{k_{j}-1},k_{j+1},\ldots,k_{n-1},0}}$ & $M$\\
\hline %
(\ref{eqn511}) & $\substack{\chi_{q}(\widetilde{\mathcal T}^{(s)}_{k_1,\ldots,k_m-1,0,\ldots,0,k_n})\times\\
\times \chi_{q}(\widetilde{\mathcal T}^{(s-4)}_{k_1,\ldots,k_m,0,\ldots,0,1,k_n})}$  &$\substack{ M= \widetilde{T}^{(s)}_{k_1,\ldots,k_m-1,0,\ldots,0,k_n}\times\\
\times \widetilde{T}^{(s-4)}_{k_1,\ldots,k_m,0,\ldots,0,1,k_n}}$  & $\substack{M\prod_{0\leq j\leq r}A^{-1}_{n,s+4k_{n}-4j-6}, \\-1\leq r\leq k_{n}-1}$ \\
\hline %
(\ref{eqn511}) & $\substack{\chi_{q}(\widetilde{\mathcal T}^{(s)}_{k_1,\ldots,k_m,0,\ldots,0,1,k_n-1})\times\\
\times \chi_{q}(\widetilde{\mathcal T}^{(s-4)}_{k_1,\ldots,k_m-1,0,\ldots,0,k_n+1})}$  &$\substack{ M= \widetilde{T}^{(s)}_{k_1,\ldots,k_m,0,\ldots,0,1,k_n-1}\times\\
\times \widetilde{T}^{(s-4)}_{k_1,\ldots,k_m-1,0,\ldots,0,k_n+1}}$  & $\substack{M\prod_{0\leq j\leq r}A^{-1}_{n,s+4k_{n}-4j-6}, \\ -1\leq r\leq k_{n}-2}$ \\
\hline %
(\ref{eqn511}) & $\substack{\chi_{q}(\widetilde{\mathcal T}^{(s+4k_n)}_{k_1,\ldots,k_m-1,0,\ldots,0,0})\times\\
\times \chi_{q}(\widetilde{\mathcal T}^{(s-4)}_{k_1,\ldots,k_m,0,\ldots,0,{2k_n+1},0})}$  &$\substack{M= \widetilde{T}^{(s+4k_n)}_{k_1,\ldots,k_m-1,0,\ldots,0,0}\times\\
\times \widetilde{T}^{(s-4)}_{k_1,\ldots,k_m,0,\ldots,0,{2k_n+1},0}} $  &$M$ \\
\hline %
(\ref{eqn512}) & $\substack{\chi_{q}(\widetilde{\mathcal T}^{(s)}_{k_1,\ldots,k_m,0,\ldots,0,k_{n-1}-2,k_n})\times\\
\times \chi_{q}(\widetilde{\mathcal T}^{(s-4)}_{k_1,\ldots,k_m,0,\ldots,0,k_{n-1},k_n})}$  &$\substack{M= \widetilde{T}^{(s)}_{k_1,\ldots,k_m,0,\ldots,0,k_{n-1}-2,k_n}\times\\
\times \widetilde{T}^{(s-4)}_{k_1,\ldots,k_m,0,\ldots,0,k_{n-1},k_n}} $  &$\substack{M\prod_{0\leq j\leq r}A^{-1}_{n,s+4k_{n}-4j-6}, \\-1\leq r\leq k_{n}-1}$ \\
\hline %
(\ref{eqn512}) & $\substack{\chi_{q}(\widetilde{\mathcal T}^{(s)}_{k_1,\ldots,k_m,0,\ldots,0,k_{n-1},k_{n}-1})\times\\
\times \chi_{q}(\widetilde{\mathcal T}^{(s-4)}_{k_1,\ldots,k_m,0,\ldots,0,k_{n-1}-2,,k_{n}+1})}$  &$\substack{M= \widetilde{T}^{(s)}_{k_1,\ldots,k_m,0,\ldots,0,k_{n-1},k_{n}-1}\times\\
\times \widetilde{T}^{(s-4)}_{k_1,\ldots,k_m,0,\ldots,0,k_{n-1}-2,,k_{n}+1}} $  &$\substack{M\prod_{0\leq j\leq r}A^{-1}_{n,s+4k_{n}-4j-6}, \\-1\leq r\leq k_{n}-2}$ \\
\hline %
(\ref{eqn512}) & $\substack{\chi_{q}(\widetilde{\mathcal T}^{(s+4k_n)}_{k_1,\ldots,k_m,0,\ldots,0,k_{n-1}-2,0})\times\\
\times \chi_{q}(\widetilde{\mathcal T}^{(s-4)}_{k_1,\ldots,k_m,0,\ldots,0,2k_n+k_{n-1},0})}$  &$\substack{M= \widetilde{T}^{(s+4k_n)}_{k_1,\ldots,k_m,0,\ldots,0,k_{n-1}-2,0}\times\\
\times \widetilde{T}^{(s-4)}_{k_1,\ldots,k_m,0,\ldots,0,2k_n+k_{n-1},0}} $  &$M$ \\
\hline %
(\ref{eqn5211}) & $\substack{\chi_{q}(\widetilde{\mathcal T}^{(s)}_{k_1,\ldots,k_m-1,0,\ldots,0,\ldots,0,k_n})\times\\
\times \chi_{q}(\widetilde{\mathcal T}^{(s-4)}_{k_1,\ldots,k_m,0,\ldots,0,\underset{l}1,0,\ldots,0,k_n})}$  &$\substack{M= \widetilde{T}^{(s)}_{k_1,\ldots,k_m-1,0,\ldots,0,\ldots,0,k_n}\times\\
\times \widetilde{T}^{(s-4)}_{k_1,\ldots,k_m,0,\ldots,0,\underset{l}1,0,\ldots,0,k_n}} $  &$\substack{M\prod_{0\leq j\leq r}A^{-1}_{n,s+4k_{n}-4j-6}, \\-1\leq r\leq k_{n}-1}$  \\
\hline %
(\ref{eqn5211}) & $\substack{\chi_{q}(\widetilde{\mathcal T}^{(s)}_{k_1,\ldots,k_m,0,\ldots,0,\underset{l}1,0,\ldots,0,k_n-1})\times\\
\times \chi_{q}(\widetilde{\mathcal T}^{(s-4)}_{k_1,\ldots,k_m-1,0,\ldots,0,\ldots,0,k_n+1})}$  &$\substack{M= \widetilde{T}^{(s)}_{k_1,\ldots,k_m,0,\ldots,0,\underset{l}1,0,\ldots,0,k_n-1}\times\\
\times \widetilde{T}^{(s-4)}_{k_1,\ldots,k_m-1,0,\ldots,0,\ldots,0,k_n+1}} $  &$\substack{M\prod_{0\leq j\leq r}A^{-1}_{n,s+4k_{n}-4j-6}, \\-1\leq r\leq k_{n}-2}$  \\
\hline %
(\ref{eqn5211}) & $\substack{\chi_{q}(\widetilde{\mathcal S}^{(s-4)}_{k_1,\ldots,k_m,0,\ldots,0,\underset{l}1,0,\ldots,0,2k_n,0})}$  &$\substack{M= \widetilde{S}^{(s-4)}_{k_1,\ldots,k_m,0,\ldots,0,\underset{l}1,0,\ldots,0,2k_n,0}} $  &$M$  \\
\hline %
(\ref{eqn5221}) & $\substack{\chi_{q}(\widetilde{\mathcal T}^{(s)}_{k_1,\ldots,k_m,0,\ldots,0,k_l-2,0,\ldots,0,k_n})\times\\
\times \chi_{q}(\widetilde{\mathcal T}^{(s-4)}_{k_1,\ldots,k_m,0,\ldots,0,k_l,0,\ldots,0,k_n})}$  &$\substack{M= \widetilde{T}^{(s)}_{k_1,\ldots,k_m,0,\ldots,0,k_l-2,0,\ldots,0,k_n}\times\\
\times \widetilde{T}^{(s-4)}_{k_1,\ldots,k_m,0,\ldots,0,k_l,0,\ldots,0,k_n}} $  &$\substack{M\prod_{0\leq j\leq r}A^{-1}_{n,s+4k_{n}-4j-6}, \\-1\leq r\leq k_{n}-1}$  \\
\hline %
(\ref{eqn5221}) & $\substack{\chi_{q}(\widetilde{\mathcal T}^{(s)}_{k_1,\ldots,k_m,0,\ldots,0,k_l,0,\ldots,0,k_n-1})\times\\
\times \chi_{q}(\widetilde{\mathcal T}^{(s-4)}_{k_1,\ldots,k_m,0,\ldots,0,k_l-2,0,\ldots,0,k_n+1})}$  &$\substack{M= \widetilde{T}^{(s)}_{k_1,\ldots,k_m,0,\ldots,0,k_l,0,\ldots,0,k_n-1}\times\\
\times \widetilde{T}^{(s-4)}_{k_1,\ldots,k_m,0,\ldots,0,k_l-2,0,\ldots,0,k_n+1}} $  &$\substack{M\prod_{0\leq j\leq r}A^{-1}_{n,s+4k_{n}-4j-6}, \\-1\leq r\leq k_{n}-2}$  \\
\hline %
(\ref{eqn5221}) & $\substack{\chi_{q}(\widetilde{\mathcal S}^{(s-4)}_{k_1,\ldots,k_m,0,\ldots,0,k_l,0,\ldots,0,2k_n,0})}$  &$\substack{M= \widetilde{S}^{(s-4)}_{k_1,\ldots,k_m,0,\ldots,0,k_l,0,\ldots,0,2k_n,0}} $  &$M$  \\
\hline %
\end{tabular} }
\caption{Classification of dominant monomials in the system in Theorem \ref{M system of type Cn}.}
\label{dominant monomials in the M-system of type C_n}
\end{table}

\subsection{Proof of Theorem \ref{M system of type Cn}}
By Table \ref{dominant monomials in the M-system of type C_n}, the dominant monomials in the $q$-characters of the left hand side and of the right hand side of every equation in Theorem \ref{M system of type Cn} are the same. Therefore  Theorem \ref{M system of type Cn} is true.

\subsection{Proof of Lemma \ref{lemma1}} \label{proof classification of dominant monomials}

We prove the case of
\begin{align*}
\chi_{q}(\widetilde{\mathcal T}^{(s)}_{k_1,\ldots,k_m,0,\ldots,0,k_{n-1}-2,k_{n}})\chi_{q}(\widetilde{\mathcal T}^{(s-4)}_{k_1,\ldots,k_m,0,\ldots,0,k_{n-1},k_{n}}),
\end{align*}
where $k_{n-1}\geq2, ~1\leq m\leq n-2$ and $0\leq k_1+\cdots+k_m\leq2$. The other cases are similar. Without loss of generality, suppose that $k_{n-2}=k_{n-3}=1$ and $s=0$.
\vskip 0.05cm
Let
\begin{align*}
&m_{1}'=\widetilde{T}^{(0)}_{0,\ldots,0,1,1,k_{n-1}-2,k_{n}}, \\
&m_{2}'=\widetilde{T}^{(-4)}_{0,\ldots,0,1,1,k_{n-1},k_{n}}.
\end{align*}

Then
\begin{equation*}
\begin{split}
m_{1}'=&n_{0}n_{4}\cdots n_{4k_n-4}(n-1)_{4k_n+1}(n-1)_{4k_n+3}\cdots (n-1)_{4k_n+2k_{n-1}-5}\\
       &(n-2)_{4k_n+2k_{n-1}-2}(n-3)_{4k_n+2k_{n-1}+1},\\
m_{2}'=&n_{-4}n_{0}\cdots n_{4k_n-8}(n-1)_{4k_n-3}(n-1)_{4k_n-1}\cdots (n-1)_{4k_n+2k_{n-1}-5}\\
       &(n-2)_{4k_n+2k_{n-1}-2}(n-3)_{4k_n+2k_{n-1}+1}.\\
\end{split}
\end{equation*}

Let $m=m_{1}m_{2}$ be a dominant monomial, where $m_{i}\in \chi_{q}(m_{i}')$, $i=1, 2$. We denote\\
\begin{equation*}
\begin{split}
m_{3}=&(n-1)_{4k_n+1}(n-1)_{4k_n+3}\cdots (n-1)_{4k_n+2k_{n-1}-5}(n-2)_{4k_n+2k_{n-1}-2}(n-3)_{4k_n+2k_{n-1}+1},\\
m_{4}=&(n-1)_{4k_n-3}(n-1)_{4k_n-1}\cdots (n-1)_{4k_n+2k_{n-1}-5}(n-2)_{4k_n+2k_{n-1}-2}(n-3)_{4k_n+2k_{n-1}+1}.\\
\end{split}
\end{equation*}
Suppose that $m_{1}\in \chi_{q}(n_{0}\cdots n_{4k_n-4})(\chi_{q}(m_{3})-m_{3})$ , then $m=m_{1}m_{2}$ is right negative and hence $m$ is not dominant. This contradicts our assumption. Therefore $m_{1}\in \chi_{q}(n_{0}n_{4}\cdots n_{4k_n-4})m_{3}$. Similarly, if $m_{2}\in \chi_{q}(n_{-4}n_{0}\cdots n_{4k_n-8})(\chi_{q}(m_{4})-m_{4})$, then $m=m_{1}m_{2}$ is right negative and hence $m$ is not dominant. Therefore $m_{2}\in \chi_{q}(n_{-4}n_{0}\cdots n_{4k_n-8})m_{4}$.

Suppose that $m_{1}\in \mathscr{M}(\chi_q(m_{1}'))\cap \mathscr{M}(\chi_{q}(n_{0}n_{4}\cdots n_{4k_n-8})(\chi_{q}(n_{4k_{n}-4})-n_{4k_{n}-4})m_{3})$. By the Frenkel-Mukhin algorithm for $\chi_q(m_{1}')$, $m_{1}$ must have the factor $n_{4k_{n}}^{-1}$. But by the Frenkel-Mukhin algorithm and the fact that $m_{2}\in \chi_{q}(n_{-4}n_{0}\cdots n_{4k_n-8})m_{4}$, $m_{2}$ does not have the factor $n_{4k_{n}}$. Therefore $m_{1}m_{2}$ is not dominant. Hence $m_{1}\in \chi_q(n_{0}n_{4}\cdots n_{4k_n-8})n_{4k_{n}-4}m_{3}$. It follows that $m_{1}=m_{1}'$.

By the Frenkel-Mukhin algorithm and the fact that $m_{2}\in \chi_{q}(n_{-4}n_{0}\cdots n_{4k_n-8})m_{4}$, $m_{2}$ must be one of the following monomials,

\begin{equation*}
\begin{split}
n_{1}=&m_{2}'A_{n,4k_{n}-6}^{-1}\\
     =&n_{-4}n_{0}\cdots n_{4k_n-12} n^{-1}_{4k_n-4}(n-1)_{4k_n-7}(n-1)_{4k_n-5}\\
      &(n-1)_{4k_n-3}(n-1)_{4k_n-1}\cdots (n-1)_{4k_n+2k_{n-1}-5}(n-2)_{4k_n+2k_{n-1}-2}(n-3)_{4k_n+2k_{n-1}+1},\\
n_{2}=&m_{2}'A_{n,4k_{n}-6}^{-1}A_{n,4k_{n}-10}^{-1}\\
     =&n_{-4}n_{0}\cdots n_{4k_n-16}n^{-1}_{4k_n-8}n^{-1}_{4k_n-4}(n-1)_{4k_n-11}(n-1)_{4k_n-9}(n-1)_{4k_n-7}(n-1)_{4k_n-5}\\
      &(n-1)_{4k_n-3}(n-1)_{4k_n-1}\cdots (n-1)_{4k_n+2k_{n-1}-5}(n-2)_{4k_n+2k_{n-1}-2}(n-3)_{4k_n+2k_{n-1}+1},\\
      &\cdots \\
n_{k_n}=&m_{2}'A_{n,4k_{n}-6}^{-1}A_{n,4k_{n}-10}^{-1}\cdots A_{n,-2}^{-1}\\
     =&n^{-1}_0 n^{-1}_4\cdots n^{-1}_{4k_n-4}(n-1)_{-3} (n-1)_{-1}\cdots (n-1)_{4k_n+2k_{n-1}-5}(n-2)_{4k_n+2k_{n-1}-2}(n-3)_{4k_n+2k_{n-1}+1}.\\
\end{split}
\end{equation*}

It follows that the dominant monomials in $\chi_{q}(\widetilde{\mathcal T}^{(0)}_{0,\ldots,0,1,1,k_{n-1}-2,k_{n}})\chi_{q}(\widetilde{\mathcal T}^{(-4)}_{0,\ldots,0,1,1,k_{n-1},k_{n}})$ are
\begin{eqnarray*}
\begin{split}
&M=m_{1}'m_{2}', \ M_{1}=n_{1}m_{1}'=MA_{n,4k_{n}-6}^{-1}, \ M_{2}=n_{2}m_{1}'=M\prod_{i=0}^{1}A_{n, 4k_{n}-4i-6}^{-1},\ldots, \\
&M_{k_n-1}=n_{k_n-1}m_{1}'=M\prod_{i=0}^{k_n-2}A_{n, 4k_{n}-4i-6}^{-1}, \ M_{k_n}=n_{k_n}m_{1}'=M\prod_{i=0}^{k_n-1}A_{n, 4k_{n}-4i-6}^{-1}.
\end{split}
\end{eqnarray*}

\section{Proof of Theorem $\ref{irreducible}$}  \label{proof irreducible}

By Lemma \ref{lemma1}, we have the following result.
\begin{corollary}
The modules in the second summand on the right hand side of every equation in Theorem \ref{M system of type Cn} are special. In particular, they are simple.
\end{corollary}

Therefore in order to prove Theorem \ref{irreducible}, we only need to prove that the modules in the first summand on the right hand side of every equation of the system in Theorem \ref{M system of type Cn} are simple. We will prove that
\begin{align} \label{one case of type An proof irreducible}
\chi_{q}(\widetilde{\mathcal T}^{(s)}_{k_1,\ldots,k_m,0,\ldots,0,k_{n-1},k_{n}-1})\chi_{q}(\widetilde{\mathcal T}^{(s-4)}_{k_1,\ldots,k_m,0,\ldots,0,k_{n-1}-2,k_{n}+1}),
\end{align}
where $k_{n-1}\geq2, ~1\leq m\leq n-2$ and $0\leq k_1+\cdots+k_m\leq2$, is simple. The other cases are similar.

The following is the proof of the case of (\ref{one case of type An proof irreducible}) in type $C_n$. Without loss of generality, suppose that $k_{n-2}=k_{n-3}=1$. We have
\begin{align} \label{the expression of T 1}
\begin{split}
&\widetilde{T}^{(s)}_{0,\ldots,0,1,1,k_{n-1},k_{n}-1} \\
=& n_{s}n_{s+4}\cdots n_{s+4k_{n}-8} (n-1)_{s+4k_{n}-3} (n-1)_{s+4k_{n}-1} \cdots (n-1)_{s+4k_{n}+2k_{n-1}-5}\\ &(n-2)_{s+4k_{n}+2k_{n-1}-2}(n-3)_{s+4k_{n}+2k_{n-1}+1},
\end{split}
\end{align}
\begin{align} \label{the expression of T 2}
\begin{split}
&\widetilde{T}^{(s-4)}_{0,\ldots,0,1,1,k_{n-1}-2,k_{n}+1} \\
=& n_{s-4}n_{s}\cdots n_{s+4k_{n}-8}n_{s+4k_{n}-4}(n-1)_{s+4k_{n}+1} (n-1)_{s+4k_{n}+3}\cdots(n-1)_{s+4k_{n}+2k_{n-1}-5}\\
 &(n-2)_{s+4k_n+2k_{n-1}-2}(n-3)_{s+4k_n+2k_{n-1}+1}.
\end{split}
\end{align}

By Lemma \ref{lemma1}, the dominant monomials in (\ref{one case of type An proof irreducible}) are
\begin{align*}
M_r = M\prod_{0\leq j\leq r}A^{-1}_{n,s+4k_{n}-4j-6},\ -1\leq r\leq k_{n}-2.
\end{align*}
We need to show that $\mathscr{M}(\chi_{q}(M_{r}))\nsubseteq \mathscr{M}(\chi_{q}(M))$ for $0\leq r\leq k_{n}-2$. We will prove the case of $\mathscr{M}(\chi_{q}(M_{0}))\nsubseteq \mathscr{M}(\chi_{q}(M))$. The other cases are similar.
\begin{align*}
M_{0}=&MA^{-1}_{n,s+4k_{n}-6}\\
 =&n^{-1}_{s+4k_{n}-4}(n-1)_{s+4k_{n}-7}(n-1)_{s+4k_{n}-5} n_{s}n_{s+4}\cdots n_{s+4k_{n}-12}\\
&(n-1)_{s+4k_{n}-3}(n-1)_{s+4k_{n}-1}\cdots (n-1)_{s+4k_{n}+2k_{n-1}-5}\\
&n_{s-4}n_{s}\cdots n_{s+4k_{n}-4} (n-1)_{s+4k_{n}+1}\cdots (n-1)_{s+4k_{n}+2k_{n-1}-5}\\ &(n-2)_{s+4k_n+2k_{n-1}-2}(n-2)_{s+4k_n+2k_{n-1}-2}(n-3)_{s+4k_n+2k_{n-1}+1}(n-3)_{s+4k_n+2k_{n-1}+1} \\
=&(n-1)_{s+4k_{n}-7}(n-1)_{s+4k_{n}-5} n_{s}n_{s+4}\cdots n_{s+4k_{n}-12}\\
&(n-1)_{s+4k_{n}-3}(n-1)_{s+4k_{n}-1}\cdots (n-1)_{s+4k_{n}+2k_{n-1}-5}\\
&n_{s-4}n_{s}\cdots n_{s+4k_{n}-8} (n-1)_{s+4k_{n}+1}\cdots (n-1)_{s+4k_{n}+2k_{n-1}-5}\\ &(n-2)_{s+4k_n+2k_{n-1}-2}(n-2)_{s+4k_n+2k_{n-1}-2}(n-3)_{s+4k_n+2k_{n-1}+1}(n-3)_{s+4k_n+2k_{n-1}+1} .
\end{align*}
We will show that $\mathscr{M}(\chi_{q}(M_{0}))\nsubseteq \mathscr{M}(\chi_{q}(M))$. By $U_q\mathfrak{sl}_2$ argument, the monomial
\begin{align*}
n_{1}=&(n-1)_{s+4k_{n}-7}(n-1)_{s+4k_{n}-5} n_{s}n_{s+4}\cdots n_{s+4k_{n}-12}\\
&(n-1)_{s+4k_{n}-3}(n-1)_{s+4k_{n}-1}\cdots (n-1)_{s+4k_{n}+2k_{n-1}-5}\\
&n_{s-4}n_{s}\cdots n_{s+4k_{n}-12}  n^{-1}_{s+4k_{n}-4}(n-1)_{s+4k_{n}-7}(n-1)_{s+4k_{n}-5}\\
&(n-1)_{s+4k_{n}+1}\cdots (n-1)_{s+4k_{n}+2k_{n-1}-5} (n-2)_{s+4k_n+2k_{n-1}-2}\\
&(n-2)_{s+4k_n+2k_{n-1}-2}(n-3)_{s+4k_n+2k_{n-1}+1}(n-3)_{s+4k_n+2k_{n-1}+1}\\
=&(n-1)^{2}_{s+4k_{n}-7}(n-1)^{2}_{s+4k_{n}-5} n_{s}n_{s+4}\cdots n_{s+4k_{n}-12}\\
&(n-1)_{s+4k_{n}-3}(n-1)_{s+4k_{n}-1}\cdots (n-1)_{s+4k_{n}+2k_{n-1}-5}\\
&n_{s-4}n_{s}\cdots n_{s+4k_{n}-12}n^{-1}_{s+4k_{n}-4}(n-1)_{s+4k_{n}+1}\cdots (n-1)_{s+4k_{n}+2k_{n-1}-5}\\ &(n-2)_{s+4k_n+2k_{n-1}-2}(n-2)_{s+4k_n+2k_{n-1}-2}(n-3)_{s+4k_n+2k_{n-1}+1}(n-3)_{s+4k_n+2k_{n-1}+1} \\
=&M_{0}A^{-1}_{n,s+4k_{n}-6}\\
=&MA^{-2}_{n,s+4k_{n}-6}
\end{align*}
is in $\chi_{q}(M_{0})$.

Suppose that $n_{1}\in \chi_{q}(\widetilde{\mathcal T}^{(s)}_{0,\ldots,0,1,1,k_{n-1},k_{n}-1})\chi_{q}(\widetilde{\mathcal T}^{(s-4)}_{0,\ldots,0,1,1,k_{n-1}-2,k_{n}+1})$. Then $n_1=m_{1}m_{2}$, where
\begin{align*}
 m_1 \in \chi_{q}(\widetilde{\mathcal T}^{(s)}_{0,\ldots,0,1,1,k_{n-1},k_{n}-1}), \quad
 m_2 \in \chi_{q}(\widetilde{\mathcal T}^{(s-4)}_{0,\ldots,0,1,1,k_{n-1}-2,k_{n}+1}).
\end{align*}
Since $n_1=MA^{-2}_{n,s+4k_{n}-6}$, by the expressions (\ref{the expression of T 1}) and (\ref{the expression of T 2}) we must have
\begin{align*}
m_1 = \widetilde{T}^{(s)}_{0,\ldots,0,1,1,k_{n-1},k_{n}-1} A^{-1}_{n,s+4k_{n}-6}.
\end{align*}
It follows that $m_{2} = \widetilde{T}^{(s-4)}_{0,\ldots,0,1,1,k_{n-1}-2,k_{n}+1} A^{-1}_{n,s+4k_{n}-6}$. But by the Frenkel-Mukhin algorithm and (\ref{the expression of T 2}), $\widetilde{T}^{(s-4)}_{0,\ldots,0,1,1,k_{n-1}-2,k_{n}+1} A^{-1}_{n,s+4k_{n}-6}$ is not in $\chi_{q}(\mathcal T^{(s-4)}_{0,\ldots,0,1,1,k_{n-1}-2,k_{n}+1})$. This is a contradiction. Hence $\mathscr{M}(\chi_{q}(M_{0}))\nsubseteq \mathscr{M}(\chi_{q}(M))$.

\section*{Acknowledgements}
This research is supported by the National Natural Science Foundation of China (no. 11371177, 11401275), and the Fundamental Research Funds for the Central Universities of China (no. lzujbky-2015-78).


\begin{thebibliography}{99999999}

\bibitem[Car05]{Car05} R. W. Carter, \textit{Lie algebras of finite and affine type}, Cambridge Studies in Advanced Mathematics, \textbf{96}, Cambridge University Press, Cambridge, 2005, xviii+632 pp.

\bibitem[C95]{C95} V. Chari, \textit{Minimal affinizations of representations of quantum groups: the rank $2$ case}, Publ. Res. Inst. Math. Sci. \textbf{31} (1995), no. 5, 873--911.

\bibitem[CG11]{CG11} V. Chari, J. Greenstein, \textit{Minimal affinizations as projective objects}, J. Geom. Phys. \textbf{61} (2011), no. 3, 594--609.

\bibitem[CMY13]{CMY13} V. Chari, A. Moura, C. A. S. Young, \textit{Prime representations from a homological perspective}, Math. Z. \textbf{274} (2013), 1-2, 613--645

\bibitem[CP91]{CP91} V. Chari, A. Pressley, \textit{Quantum affine algebras}, Comm. Math. Phys. \textbf{142} (1991), no. 2, 261--283.

\bibitem[CP94]{CP94} V. Chari, A. Pressley, \textit{A guide to quantum groups}, Cambridge University Press, Cambridge (1994), xvi+651 pp.

\bibitem[CP95a]{CP95a} V. Chari, A. Pressley, \textit{Quantum affine algebras and their representations}, Representations of groups (Banff, AB, 1994), 59--78, CMS Conf. Proc., \textbf{16}, Amer. Math. Soc., Providence, RI, 1995.

\bibitem[CP95b]{CP95b} V. Chari, A. Pressley, \textit{Minimal affinizations of representations of quantum groups: the nonsimple-laced case}, Lett. Math. Phys. \textbf{35} (1995), no. 2, 99--114.

\bibitem[CP96a]{CP96a} V. Chari, A. Pressley, \textit{Minimal affinizations of representations of quantum groups: the irregular case}, Lett. Math. Phys. \textbf{36} (1996), no. 3, 247--266.

\bibitem[CP96b]{CP96b} V. Chari, A. Pressley, \textit{Minimal affinizations of representations of quantum groups: the simply laced case}, J. Algebra \textbf{184} (1996), no. 1, 1--30.

\bibitem[CP97]{CP97} V. Chari, A. Pressley, \textit{Factorization of representations of quantum affine algebras}, Modular interfaces (Riverside, CA, 1995), AMS/IP Stud. Adv. Math., vol. 4, Amer. Math. Soc., Provdence, RI (1997), 33--40.


\bibitem[Dri87]{Dri87} V. G. Drinfeld, \textit{Quantum groups}, Proceedings of the International Congress of Mathematicians, Vol. \textbf{1}, \textbf{2} (Berkeley, Calif., 1986), Amer. Math. Soc., Providence, RI (1987), 798--820.

\bibitem[Dri88]{Dri88} V. G. Drinfeld, \textit{A new realization of Yangians and of quantum affine algebras}, (Russian) Dokl. Akad. Nauk SSSR \textbf{296} (1987), no. 1, 13--17;
translation in Soviet Math. Dokl. \textbf{36} (1988), no. 2, 212--216.

\bibitem[DLL15]{DLL15} B. Duan, J. R. Li, Y. F. Luo,  \textit{On the minimal affinizations of type $F_4$}, arXiv:1502.02423, 2015, 1--32.

\bibitem[FM01]{FM01} E. Frenkel, E. Mukhin, \textit{Combinatorics of $q$-characters of finite-dimensional representations of quantum affine algebras},
Comm. Math. Phys. \textbf{216} (2001), no. 1, 23--57.

\bibitem[FR98]{FR98} E. Frenkel, N. Yu. Reshetikin, \textit{The $q$-characters of representations of quantum affine algebras and deformations of W-algebras}, Recent developments in quantum affine algebras and related topics (Raleigh, NC, 1998), 163--205, Contemp. Math., \textbf{248}, Amer. Math. Soc., Providence, RI, 1999.


\bibitem[FZ02]{FZ02} S. Fomin, A. Zelevinsky, \textit{Cluster algebras I: Foundations}, J. Amer. math. Soc. \textbf{15} (2002), 497--529.



\bibitem[GG14]{GG14} J. Grabowski, S. Gratz, \textit{Cluster algebras of infinite rank, with an appendix by Michael Groechenig}, J. Lond. Math. Soc. \textbf{89} (2014), no. 2, 337--363.


\bibitem[Her06]{Her06} D. Hernandez, \textit{The Kirillov-Reshetikhin conjecture and solutions of T-systems},
J. Reine Angew. Math. \textbf{596} (2006), 63--87.

\bibitem[Her07]{Her07} D. Hernandez, \textit{On minimal affinizations of representations of quantum groups},
Comm. Math. Phys. \textbf{276} (2007), no. 1, 221--259.

\bibitem[Her08]{Her08} D. Hernandez, \textit{Smallness problem for quantum affine algebras and quiver varieties},
Ann. Sci. \'Ec. Norm. Sup\'er. \textbf{41} (2008), no. 2, 271--306.

\bibitem[HL10]{HL10} D. Hernandez, B. Leclerc, \textit{Cluster algebras and quantum affine algebras}, Duke Math. J. \textbf{154} (2010), no. 2, 265--341.

\bibitem[HL13]{HL13} D. Hernandez, B. Leclerc, \textit{A cluster algebra approach to $q$-characters of Kirillov-Reshetikhin modules}, arXiv:1303.0744, 1--45.




\bibitem[Jim85]{Jim85} M. Jimbo, \textit{A $q$-difference analogue of $U(\mathfrak{g})$ and the Yang-Baxter equation},
Lett. Math. Phys. \textbf{10} (1985), no. 1, 63--69.





\bibitem[Le10]{Le10} B. Leclerc, \textit{Cluster algebras and representation theory}, Proceedings of the International Congress of Mathematicians. Volume IV (2010). 2471--2488.

\bibitem[Li15]{Li15}  J. R. Li, \textit{On the extended T-system of type $C_{3}$}, Journal of Algebraic Combinatorics, \textbf{41} (2015), 577-617.



\bibitem[LM13]{LM13} J. R. Li, E. Mukhin, \textit{Extended T -system of type $G_{2}$}, SIGMA Symmetry Integrability Geom. Methods Appl. \textbf{9} (2013), 054, 28pp.

\bibitem[M10]{M10} A. Moura, \textit{Restricted limits of minimal affinizations}, Pacific J. Math. \textbf{244} (2010), no. 2, 359--397.

\bibitem[MP07]{MP07} M. G. Moakes, A. N. Pressley, \textit{q-characters and minimal affinizations}, Int. Electron. J. Algebra \textbf{1} (2007), 55--97.

\bibitem[MP11]{MP11} A. Moura, F. Pereira, \textit{Graded limits of minimal affinizations and beyond: the multiplicity free case for type $E_6$}, Algebra Discrete Math. \textbf{12} (2011), no. 1,  69--115.


\bibitem[MY12a]{MY12a} E. Mukhin, C. A. S. Young, \textit{Extended T-systems}, Selecta Math.(N.S.), \textbf{18} (2012), no. 3, 591--631.

\bibitem[MY12b]{MY12b} E. Mukhin, C. A. S. Young, \textit{Path description of type B q-characters}, Adv. Math. \textbf{231} (2012), no. 2, 1119--1150.

\bibitem[MY14]{MY14} E. Mukhin, C. A. S. Young, \textit{Affinization of category $\mathcal{O}$ for quantum groups}, Trans. Amer. Math. Soc. \textbf{366} (2014), 4815--4847.

\bibitem[Nao13]{Nao13}K. Naoi, \textit{Demazure modules and graded limits of minimal affinizations}, Represent. Theory, \textbf{17} (2013), 524--556.

\bibitem[Nao14]{Nao14} K. Naoi, \textit{Graded limits of minimal affinizations in type D}, SIGMA Symmetry Integrability Geom. Methods Appl. \textbf{10} (2014), Paper 047, 20 pp.

\bibitem[QL14]{QL14} L. Qiao, J. R. Li, \textit{Cluster algebras and minimal affinizations of representations of the quantum group of type $G_2$}, arXiv:1412.3884, 2014, 1--17.

\bibitem[SS14]{SS14} S. V. Sam, \textit{Jacobi-Trudi determinants and characters of minimal affinizations}, Pacific J. Math. \textbf{272} (2014), no. 1, 237--244.

\bibitem[ZDLL15]{ZDLL15} Q. Q. Zhang, B. Duan, J. R. Li, Y. F. Luo,  \textit{M-systems and cluster algebras}, arXiv:1501.00146, 2015, 1--44.

\end{thebibliography}
\end{document}